\documentclass[reqno]{amsart}

\usepackage{a4wide}
\usepackage{color}
\usepackage{mathrsfs}
\usepackage{mathtools}
\usepackage{tikz-cd}
\usepackage{amsmath}
\usepackage{amssymb}
\usepackage{bbm}
\numberwithin{equation}{section}
\usepackage[colorlinks,citecolor=green,linkcolor=red]{hyperref}
\usepackage{hyphenat}
\usepackage{enumerate}
\usepackage[utf8]{inputenc}

\newcommand{\N}{\mathbb{N}}
\newcommand{\R}{\mathbb{R}}

\newcommand{\sfd}{{\sf d}}

\newcommand{\restr}[1]{\lower3pt\hbox{\(|_{#1}\)}}

\newcommand{\nchi}{{\raise.3ex\hbox{\(\chi\)}}}

\newcommand{\fr}{\penalty-20\null\hfill\(\blacksquare\)}
\newcommand{\mm}{\mathfrak{m}}
\newcommand{\X}{{\rm X}}
\newcommand{\Y}{{\rm Y}}

\newcommand{\LIP}{{\rm LIP}}

\newtheorem{theorem}{Theorem}[section]
\newtheorem{corollary}[theorem]{Corollary}
\newtheorem{lemma}[theorem]{Lemma}
\newtheorem{proposition}[theorem]{Proposition}
\newtheorem{definition}[theorem]{Definition}
\newtheorem{example}[theorem]{Example}
\newtheorem{remark}[theorem]{Remark}

\title[A categorical pespective on extended metric-topological spaces]{A categorical perspective \\ on extended metric-topological spaces}

\author{Enrico Pasqualetto}
\address{Department of Mathematics and Statistics,
P.O.\ Box 35 (MaD), FI-40014 University of Jyvaskyla}
\email{enrico.e.pasqualetto@jyu.fi}

\author{Timo Schultz}
\address{Department of Mathematics and Statistics,
P.O.\ Box 35 (MaD), FI-40014 University of Jyvaskyla}
\email{timo.m.schultz@jyu.fi}

\author{Janne Taipalus}
\address{Department of Mathematics and Statistics,
P.O.\ Box 35 (MaD), FI-40014 University of Jyvaskyla}
\email{janne.m.m.taipalus@jyu.fi}

\begin{document}
\date{\today} 
\keywords{Extended metric-topological space; limit; colimit; compactification.}
\subjclass[2020]{18F60, 46M40, 46M15, 49J52, 53C23}
\begin{abstract}
Motivated by the analysis and geometry of metric-measure structures in infinite dimensions, we study the category
of extended metric-topological spaces, along with many of its distinguished subcategories (such as the one of compact spaces).
One of the main achievements is the proof of the bicompleteness (i.e.\ of the existence of all small limits and colimits)
of the aforementioned categories.
\end{abstract}
\maketitle
\tableofcontents
\section{Introduction}
\subsection*{General overview}
For the purpose of developing a robust framework for metric analysis and optimal transport problems in
(possibly non-smooth and infinite-dimensional) metric-measure structures, Ambrosio, Erbar and Savar\'{e}
introduced in \cite{AmbrosioErbarSavare16} the notions of \emph{extended metric-topological space} and
\emph{extended metric-topological measure space}, which we abbreviate to \emph{e.m.t.\ space} and
\emph{e.m.t.m.\ space}, respectively. The former are sets \(\X\) equipped with a topology \(\tau\)
and an extended distance \(\sfd\) that fulfill suitable compatibility conditions, while the latter
are e.m.t.\ spaces \((\X,\tau,\sfd)\) that are endowed with a (finite) Radon measure \(\mm\geq 0\)
defined on the Borel \(\sigma\)-algebra of \((\X,\tau)\).
\medskip

The objective of this paper is to examine e.m.t.\ spaces from a categorical perspective. Specifically,
we constructively demonstrate that the category \({\bf ExtMetTop}\) of e.m.t.\ spaces -- along with several
of its most relevant subcategories -- are both \emph{complete} and \emph{cocomplete}, in the sense that they
admit all small limits and colimits. In the aforementioned categories, as for the class of morphisms
we opt in favour of \emph{continuous-short maps} (scilicet, maps that are continuous between the given topologies
and non-expansive with respect to the given extended distances). The (co)completeness of
a class of e.m.t.\ spaces signifies its stability under a wide variety of (universal) constructions, including
e.g.\ products, quotients and inverse/direct limits.
\subsection*{Statement of results}

In accordance with \cite[Definition 4.1]{AmbrosioErbarSavare16}, we call \((\X,\tau,\sfd)\) an e.m.t.\ space if:
\begin{itemize}
\item The topology \(\tau\) coincides with the initial topology of the set \(\LIP_{b,1}(\X,\tau,\sfd)\) of all bounded continuous-short functions
from \((\X,\tau,\sfd)\) to \(\R\).
\item The extended distance \(\sfd\) can be \emph{\(\tau\)-recovered}, by which we mean that
\[
\sfd(x,y)=\sup\big\{|f(x)-f(y)|\;\big|\;f\in\LIP_{b,1}(\X,\tau,\sfd)\big\}\quad\text{ for every }x,y\in\X;
\]
\end{itemize}
see Definition \ref{def:emt}. To ground intuition for the concept of e.m.t.(m.) space, let us reflect on the theory of \emph{abstract Wiener spaces},
which is the archetype of metric-measure structure the axiomatisation of e.m.t.\ space is tailored to. Let us consider an infinite-dimensional,
separable Banach space \({\rm B}\) together with a centred, non-degenerate Gaussian measure \(\gamma\). Indicating with \(({\rm H},|\cdot|_{\rm H})\)
the \emph{Cameron--Martin space} of \(({\rm B},\gamma)\) (see \cite{Bogachev15}), the natural extended distance \(\sfd_{\rm H}\) on \({\rm B}\) is
\[
\sfd_{\rm H}(x,y)\coloneqq\left\{\begin{array}{ll}
|x-y|_{\rm H}\\
+\infty
\end{array}\quad\begin{array}{ll}
\text{ if }x,y\in{\rm B}\text{ satisfy }x-y\in{\rm H},\\
\text{ otherwise.}
\end{array}\right.
\]
It turns out that \(({\rm B},\tau_{\rm B},\sfd_{\rm H},\gamma)\) is an e.m.t.m.\ space \cite[Section 13.2]{AmbrosioErbarSavare16}, where \(\tau_{\rm B}\)
denotes either the strong or the weak topology of \({\rm B}\). This construct illustrates the necessity -- when aspiring to a description of an
infinite-dimensional space -- to allow for a decoupling of the r\^{o}les of the topology and of the distance. Indeed, in this context the (extended) distance
\(\sfd\) that is used to compute lengths and speeds induces a topology that is typically too fine to possess effective topological and measure-theoretical features,
whence the need for a coarser topology \(\tau\). In addition to abstract Wiener spaces, the class of e.m.t.m.\ spaces includes all \emph{metric measure spaces}
(where the topology is exactly the one induced by the distance), configuration spaces \cite{AlbeverioKondratievRockner,AmbrosioErbarSavare16,Br:Su},
extended Wasserstein spaces \cite[Section 5]{AmbrosioErbarSavare16} and Hellinger(--Kantorovich) spaces \cite{Ds:So:25}, amongst other spaces of interest.
\medskip

The main result of this paper states that
\[
\text{the category }{\bf ExtMetTop}\text{ is \emph{bicomplete} (i.e.\ both complete and cocomplete);}
\]
cf.\ Theorem \ref{thm:ExtMetTop_bicompl}. Let us briefly delineate our strategy to prove the bicompleteness of \({\bf ExtMetTop}\):
\begin{itemize}
\item We introduce the category \({\bf PreExt\Psi MetTop}\) of \emph{pre-extended pseudometric-topological spaces} (Definition \ref{def:pre-e.pm.t.}),
which are topological spaces \((\X,\tau)\) equipped with an extended pseudodistance \(\sfd\) (and no other requirements), with continuous-short maps as morphisms.
\item Leveraging our knowledge of all limits and colimits in the category \({\bf Top}\) of topological spaces and in the category \({\bf Ext\Psi Met}\) of extended
pseudometric spaces (see Section \ref{s:lim_Set_Top_ExtPMet}), we prove that \({\bf PreExt\Psi MetTop}\) is bicomplete (Theorem \ref{thm:PreExtPMetTop_bicompl}).
\item In Section \ref{s:emt-fication} we construct a functor \({\sf emt}\colon{\bf PreExt\Psi MetTop}\to{\bf ExtMetTop}\), which we name the \emph{e.m.t.-fication
functor}, that assigns to each pre-extended pseudometric-topological space an e.m.t.\ space in a canonical (or, better still, a universal) manner.
\item The functor \({\sf emt}\) exhibits \({\bf ExtMetTop}\) as a \emph{reflective} subcategory of \({\bf PreExt\Psi MetTop}\), i.e.\ \({\sf emt}\)
is the left adjoint to the inclusion functor \({\bf ExtMetTop}\hookrightarrow{\bf PreExt\Psi MetTop}\); cf.\ Theorem \ref{thm:emt-fication}. Consequently, a general
principle in category theory (which we report in Proposition \ref{prop:create_lim}) guarantees that the category \({\bf ExtMetTop}\) is bicomplete as well. The same
result would allow to supply an explicit description of all (co)limits in \({\bf ExtMetTop}\).
\end{itemize}
Along the same lines, we investigate many categories of e.m.t.\ spaces and their mutual connections. In the following table we list
those categories that we shall encounter throughout the paper:
\smallskip

{\small
\[\begin{tabular}{|p{2.9cm}||p{5.6cm}|p{3.3cm}|}
\hline
\multicolumn{3}{|c|}{List of categories} \\
\hline
Category & Objects & Morphisms\\
\hline
\({\bf Tych}\) & Tychonoff spaces & Continuous maps \\
\({\bf Ext\Psi Met}\) & Extended pseudometric spaces & Short maps \\
\({\bf ExtMet}\) & Extended metric spaces & Short maps \\
\({\bf PreExt\Psi MetTop}\) & Pre-extended pseudometric-topological spaces (abbrev.\ to \emph{pre-e.pm.t.\ spaces})
& Continuous-short maps \\
\({\bf ExtMetTop}\) & Extended metric-topological spaces (abbrev.\ to \emph{e.m.t.\ spaces})
& Continuous-short maps \\
\({\bf MetTop}_{\leq\lambda}\) & E.m.t.\ spaces with diameter at most \(\lambda\) & Continuous-short maps \\
\({\bf ExtMetTop}_{\rm mc}\) & Metrically-complete e.m.t.\ spaces & Continuous-short maps \\
\({\bf CptExtMetTop}\) & Compact e.m.t.\ spaces & Continuous-short maps \\
\({\bf CptExtMetTop}_{\rm geo}\) & Compact geodesic e.m.t.\ spaces & Continuous-short maps \\
\hline
\end{tabular}\]
}
\smallskip

Amongst all the above categories, we wish to draw particular attention to that of \emph{compact} e.m.t.\ spaces
(scilicet, e.m.t.\ spaces \((\X,\tau,\sfd)\) with \(\tau\) compact), denoted by \({\bf CptExtMetTop}\).
From a geometric or analytic viewpoint, the importance of \({\bf CptExtMetTop}\) is due to the fact that -- in combination
with the \emph{compactification functor} \(\gamma\colon{\bf ExtMetTop}\to{\bf CptExtMetTop}\), to which Section \ref{s:compactification} is devoted --
it allows us to reduce the study of several problems to the compact setting. This procedure is undeniably relevant even in the setting
of metric measure spaces, since it is not possible to define a compactification functor at a purely metric level.
It is also worth pointing out that \({\bf ExtMetTop}\) `incorporates' both the category \({\bf Tych}\) of Tychonoff
spaces (i.e.\ completely regular Hausdorff spaces) and the category \({\bf ExtMet}\) of extended metric spaces;
see Section \ref{s:Tych_and_ExtMet}.
\medskip

In the following table we collect the various functors between different categories of topological and/or metric structures
that will be introduced and studied in this paper:
\smallskip

{\small
\[\begin{tabular}{|p{1.1cm}||p{2.8cm}|p{2.9cm}|p{4.65cm}|}
\hline
\multicolumn{4}{|c|}{List of functors} \\
\hline
Functor & Domain & Codomain & Name/description \\
\hline
\({\sf emt}\) & \({\bf PreExt\Psi MetTop}\) & \({\bf ExtMetTop}\) & E.m.t.-fication functor \\
\(\gamma\) & \({\bf ExtMetTop}\) & \({\bf CptExtMetTop}\) & Compactification functor \\
\(\bar\gamma\) & \({\bf PreExt\Psi MetTop}\) & \({\bf CptExtMetTop}\) & Generalised version of the compactification functor \\
\(\mathfrak D_\infty\) & \({\bf Tych}\) & \({\bf ExtMetTop}\) & \(\infty\)-discretisation functor \\
\(\mathfrak D_\lambda\) & \({\bf Tych}\) & \({\bf MetTop}_{\leq\lambda}\) & \(\lambda\)-discretisation functor \\
\({\rm T}\) & \({\bf ExtMet}\) & \({\bf ExtMetTop}\) & Attaching the induced topology to an extended metric space \\
\(\mathfrak T_\lambda\) & \({\bf ExtMetTop}\) & \({\bf MetTop}_{\leq\lambda}\) & \(\lambda\)-truncation functor \\
\({\sf mc}\) & \({\bf ExtMetTop}\) & \({\bf ExtMetTop}_{\rm mc}\) & Metric completion functor\\
\({\sf geo}\) & \({\bf CptExtMetTop}\) & \({\bf CptExtMetTop}_{\rm geo}\) & Geodesification functor \\
\hline
\end{tabular}\]
}

The diagram below illustrates the relations amongst the various categories under consideration.
As for the notation we adopt in the diagram, by \(\iota\) we generically indicate an inclusion functor,
while \(U\) is used for forgetful functors. Let us also remind that by writing \(F\dashv G\) we mean that
the functor \(F\) is the left adjoint to the functor \(G\).
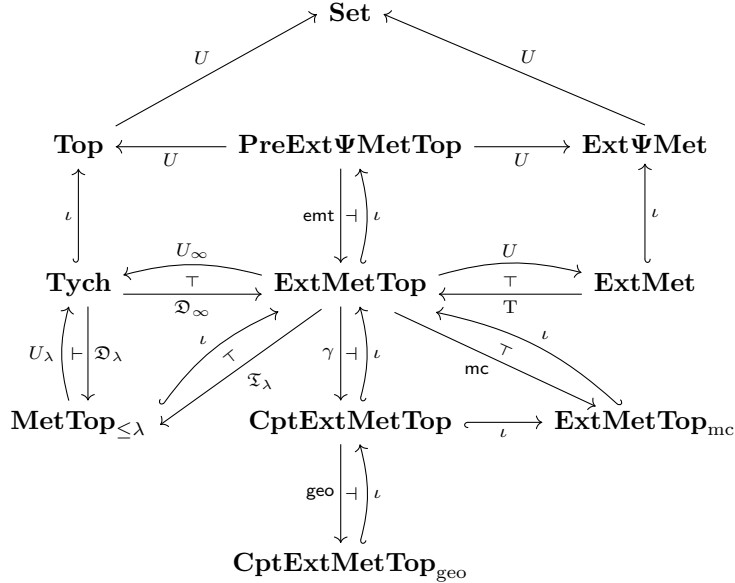
\smallskip

\[\begin{tikzcd}[row sep=huge]
& {\bf Set} & \\
{\bf Top} \arrow[ur,"U", end anchor=west] & {\bf PreExt\Psi MetTop} \arrow[l,"U"]
\arrow[d, "{\sf emt}"', "\dashv"{inner sep=.3ex} ,shift right = 1.5] \arrow[r,"U"']
& {\bf Ext\Psi Met} \arrow[ul, "U"', start anchor = north, end anchor = east] \\
{\bf Tych} \arrow[u,hook, "\iota"]
\arrow[r,"\mathfrak D_\infty"', "\mathbin{\rotatebox[origin=c]{90}{$\dashv$}}"{inner sep=.5ex}, shift right = 1.5]
\arrow[d, "\mathfrak D_\lambda", "\vdash"'{inner sep=.3ex},  start anchor = south, end anchor = north, shift left = 1.5]
& {\bf ExtMetTop} \arrow[l,"U_\infty"',bend right=15, start anchor = real west, end anchor = real east, shift right = 1.5]
\arrow[dl,"\mathfrak T_\lambda","\mathbin{\rotatebox[origin=c]{-51}{$\vdash$}}"'{inner sep=.03ex},end anchor=east,shift left=0.5]
\arrow[d,"\gamma"', "\dashv"{inner sep=.3ex}, shift right = 1.5]
\arrow[dr, "{\sf mc}"', "\mathbin{\rotatebox[origin=c]{66}{$\dashv$}}"{inner sep=.03ex}, shift right = 1]
\arrow[r,"U",bend left = 15, start anchor = real east, end anchor = real west, shift left = 1.5]
\arrow[u, "\iota"', hook', bend right = 15, start anchor = north, end anchor = south, shift right = 1.5]
& {\bf ExtMet} \arrow[l, "{\rm T}", "\mathbin{\rotatebox[origin=c]{90}{$\dashv$}}"'{inner sep=.5ex}, shift left = 1.5] \arrow[u,hook',"\iota"']\\
{\bf MetTop}_{\leq\lambda} \arrow[u, "U_\lambda", bend left = 15, start anchor = north, end anchor = south, shift left = 1.5]
\arrow[ur,hook,"\iota",bend left = 15, start anchor = north east, shift right = 1]
& {\bf CptExtMetTop} \arrow[d,"{\sf geo}"', "\dashv"{inner sep=.3ex}, shift right = 1.5] \arrow[r,hook',"\iota"']
\arrow[u, "\iota"', hook', bend right = 15, start anchor = north, end anchor = south, shift right = 1.5]
& {\bf ExtMetTop}_{\rm mc} \arrow[ul,hook', "\iota"', bend right = 15, end anchor = south east, shift left = 1] \\
& {\bf CptExtMetTop}_{\rm geo}
\arrow[u, hook', "\iota"', bend right = 15, start anchor = north, end anchor = south, shift right = 1.5] &
\end{tikzcd}\]
\smallskip

En route, we also obtain a few results of independent interest regarding e.m.t.\ spaces. For instance,
we show in Proposition \ref{prop:charact_cpt_emt} that all compact Hausdorff spaces equipped with a lower
semicontinuous extended distance are e.m.t.\ spaces, whereas in Appendix \ref{app:equiv_emt} we provide
several characterisations of an e.m.t.\ space (akin to the various definitions of completely regular space).
\subsection*{Pertinent literature}
As emerges from the previous sections of this introduction, the approach to extended metric-topological spaces that
we have presented was prompted by problems in geometric analysis and metric geometry. Nonetheless, kindred
metric-topological structures are objects of study in rather different research areas. In the first instance
we mention the theory of \emph{separated metric compact Hausdorff spaces}, which were introduced by Hofmann and Reis
in \cite{HR18} as the metric counterpart of Nachbin's notion of \emph{compact ordered space} \cite{Nac65}; cf.\ also
Tholen's paper \cite{Tho09} for a more general approach to ordered/metric/topological structures. The category
\({\bf MetCH_{sep}}\) of separated metric compact Hausdorff spaces has been comprehensively investigated e.g.\ also
by Gutierres and Hofmann \cite{GH13}, by Hofmann and Nora \cite{HN20}, and by Abbadini and Hofmann \cite{AbbadiniHofmann25}.
As a consequence of Proposition \ref{prop:charact_cpt_emt}, compact e.m.t.\ spaces fully agree with those separated metric
compact Hausdorff spaces that are \emph{symmetric}. The identification persists at the categorical level:
\[
{\bf CptExtMetTop}={\bf MetCH_{sep,sym}},
\]
where \({\bf MetCH_{sep,sym}}\) denotes the category of symmetric separated metric compact Hausdorff spaces, with
continuous non-expansive maps as morphisms; see Remark \ref{rmk:comparison_MetCH}.
\medskip

Another strictly related theory is the one of \emph{topometric spaces}. The definition of a compact topometric space,
which coincides with our notion of compact e.m.t.\ space (see Remark \ref{rmk:comparison_MetCH}), has been introduced
by Ben Yaacov and Usvyatsov in \cite{BU10} as a formalism for type spaces appearing in continuous first order logic.
The theory of (not necessarily compact) topometric spaces was then established by Ben Yacoov in \cite{Ben08b},
in connection with the analysis of perturbation structures on type spaces. As we shall discuss in Remark \ref{rmk:comparison_topometric},
the class of topometric spaces is strictly vaster than the one of e.m.t.\ spaces, and the latter correspond exactly
to those topometric spaces that are \emph{completely regular} (in the sense of \cite[Definition 2.3]{Ben13}).
Further results concerning topometric spaces were achieved e.g.\ in the works \cite{BBM13,BM15,BM25}.
\medskip

\noindent{\bf Acknowledgements.}
The first named author has been supported by the Research Council of Finland grant 362898. The third named author has been supported
by the Research Council of Finland grant 354241 and the Emil Aaltonen Foundation through the research group ``Quasiworld network''.
The authors are grateful to Zhenhao Li for having brought to their attention some relevant literature.
\section{Preliminaries}
\subsection{Topological notions}
We begin by reminding many concepts and results in topology, referring e.g.\ to \cite{Kelley75} for a detailed account of these topics.
\medskip

A topological space \((\X,\tau)\) is said to be \emph{completely regular} provided the following property holds: if \(C\subseteq\X\) is a closed
set and \(x\in\X\setminus C\), there exists a continuous function \(f\colon\X\to[0,1]\) such that \(f(y)=0\) for every \(y\in C\) and \(f(x)=1\).
A completely regular topological space that is also Hausdorff is called a \emph{Tychonoff space}. Every compact Hausdorff space is a Tychonoff space.
\medskip

Given two topological spaces \((\X,\tau_\X)\) and \((\Y,\tau_\Y)\), we denote by \(C((\X,\tau_\X);(\Y,\tau_\Y))\) the set of all continuous maps
from \((\X,\tau_\X)\) to \((\Y,\tau_\Y)\). In the particular case where \((\Y,\tau_\Y)\) is the real line \(\R\) equipped with the Euclidean topology,
we adopt the shorthand notation \(C(\X,\tau_\X)\coloneqq C((\X,\tau_\X);\R)\). We define \(C_b(\X,\tau_\X)\) as the space of all those functions
\(f\in C(\X,\tau_\X)\) that are also bounded. The space \(C_b(\X,\tau_\X)\) is a Banach space with respect to the supremum norm
\[
\|f\|_{C_b(\X,\tau_\X)}\coloneqq\sup_{x\in\X}|f(x)|\quad\text{ for every }f\in C_b(\X,\tau_\X).
\]
\subsubsection*{Quotient topologies}
Let \((\X,\tau)\) be a topological space and \(\mathcal R\) an equivalence relation on \(\X\). Denote
by \(\pi\colon\X\to\X/\mathcal R\) the projection to the quotient space \(\X/\mathcal R\).
Then the \emph{quotient topology} \(\tilde\tau\) on \(\X/\mathcal R\) is defined as the final topology
\[
\tilde\tau\coloneqq\big\{U\subseteq\X/\mathcal R\;\big|\;\pi^{-1}(U)\in\tau\big\}.
\]
The quotient space \((\X/\mathcal R,\tilde\tau)\) and the projection map \(\pi\colon\X\to\X/\mathcal R\)
are characterised by the following universal property: if \((\Y,\tau_\Y)\) is a topological space
and \(\varphi\colon\X\to\Y\) is a continuous map such that \(\varphi(x)=\varphi(y)\) whenever \(x,y\in\X\) satisfy \(x\mathcal R y\),
then there exists a unique continuous map
\(\tilde\varphi\colon\X/\mathcal R\to\Y\) such that
\[\begin{tikzcd}
\X \arrow[r,"\varphi"] \arrow[d,swap,"\pi"] & \Y \\
\X/\mathcal R \arrow[ur,swap,"\tilde\varphi"] &
\end{tikzcd}\]
is a commutative diagram.
\subsubsection*{Stone--\v{C}ech compactification}
We remind the Stone--\v{C}ech compactification of a topological space:
\begin{theorem}[Stone--\v{C}ech compactification]
Let \((\X,\tau)\) be a topological space. Then there exist a compact Hausdorff space \((\beta\X,\beta\tau)\) and a continuous map \(i\colon\X\to\beta\X\)
such that the following universal property holds: if \((\Y,\tau_\Y)\) is a compact Hausdorff space and \(\varphi\colon\X\to\Y\) is a continuous map,
then there exists a unique continuous map \(\varphi_\beta\colon\beta\X\to\Y\) such that the diagram
\[\begin{tikzcd}
\X \arrow[r,"i"] \arrow[dr,swap,"\varphi"] & \beta\X \arrow[d,"\varphi_\beta"] \\
& \Y
\end{tikzcd}\]
commutes. We say that \((\beta\X,\beta\tau)\) is the \emph{Stone--\v{C}ech compactification} of \((\X,\tau)\).
Moreover, the image \(i(\X)\) of \(i\) is dense in \(\beta\X\). The map \(i\colon\X\to\beta\X\) is a homeomorphism onto its image if and only if
\((\X,\tau)\) is a Tychonoff space.
\end{theorem}

We now remind one construction of the Stone--\v{C}ech compactification \(\beta\N\) of the space of natural numbers \(\N\) (endowed with the discrete topology)
that will prove useful in Example \ref{ex:betaX_vs_gammaX}. The space \(\beta\N\) can be defined as the set of all ultrafilters on \(\N\), equipped with its
\emph{Stone topology} (i.e.\ the topology generated by those sets of the form \(\{\mathcal F:E\in\mathcal F\}\), for \(E\) an arbitrary subset of \(\N\)).
The embedding \(i\colon\N\hookrightarrow\beta\N\) is given by the map associating to a natural number \(n\in\N\) the principal ultrafilter on \(\N\) whose
principal element is \(n\). Given a compact Hausdorff space \((\Y,\tau_\Y)\), a (continuous) map \(\varphi\colon\N\to\Y\) and an ultrafilter \(\mathcal F\) on \(\N\),
we have that \(\varphi_\beta(\mathcal F)\) is given by the ultralimit of \((\varphi(n))_{n\in\N}\subseteq\Y\) with respect to \(\mathcal F\) (whose existence
and uniqueness are guaranteed by the compactness and the Hausdorffness of \(\Y\), respectively).
\subsubsection*{Initial topologies}
Given a set \(\X\), a set of topological spaces \(\{(\Y_i,\tau_i)\}_{i\in I}\) and a set \(\{\varphi_i\}_{i\in I}\) of maps \(\varphi_i\colon\X\to\Y_i\),
we call \emph{initial topology} of \(\{\varphi_i\}_{i\in I}\) the coarsest topology on \(\X\) for which all maps \(\varphi_i\) are continuous. A basis of the initial topology of \(\{\varphi_i\}_{i\in I}\) is given by
\[
\bigg\{\bigcap_{i\in F}\varphi_i^{-1}(U_i)\;\bigg|\;F\subseteq I\text{ finite, }U_i\in\tau_i\text{ for every }i\in F\bigg\}.
\]
A topological space \((\X,\tau)\) is completely regular if and only if \(\tau\) is the initial topology of \(C_b(\X,\tau)\).
\begin{remark}\label{rmk:charact_initial_top}{\rm
The initial topology \(\tau\) of \(\{\varphi_i\}_{i\in I}\) is the unique topology with the following characteristic property: given a topological
space \(({\rm Z},\tau_{\rm Z})\) and a map \(\psi\colon{\rm Z}\to\X\), the map \(\psi\colon({\rm Z},\tau_{\rm Z})\to(\X,\tau)\) is continuous if
and only if \(\varphi_i\circ\psi\colon({\rm Z},\tau_{\rm Z})\to(\Y_i,\tau_i)\) is a continuous map for every \(i\in I\).
\fr}\end{remark}
\begin{remark}\label{rmk:criterion_initial_cpt}{\rm
Let \((\X,\tau)\) be a compact Hausdorff space. Let \(\mathcal S\) be any subset of \(C_b(\X,\tau)\) that \emph{separates the points} of \(\X\),
i.e.\ for any \(x,y\in\X\) with \(x\neq y\) there exists \(f\in\mathcal S\) such that \(f(x)\neq f(y)\). Then it is well known and easy to check
that \(\tau\) coincides with the initial topology of \(\mathcal S\).
\fr}\end{remark}
\subsection{Metric notions}
We now fix some terminology concerning (extended) metric/pseudometric spaces and present some related results.
Let \(\X\) be a set and \(\sfd\colon\X\times\X\to[0,+\infty]\) a function. Then:
\begin{itemize}
\item We say that \(\sfd\) is an \emph{extended pseudodistance}, and that \((\X,\sfd)\) is an \emph{extended pseudometric space},
if \(\sfd\) vanishes on the diagonal of \(\X\times\X\) (i.e.\ \(\sfd(x,x)=0\) for every \(x\in\X\)), it is symmetric (i.e.\ \(\sfd(x,y)=\sfd(y,x)\)
for every \(x,y\in\X\)) and it satisfies the triangle inequality (i.e.\ \(\sfd(x,y)\leq\sfd(x,z)+\sfd(z,y)\) for every \(x,y,z\in\X\)).
\item We say that \(\sfd\) is an \emph{extended distance}, and that \((\X,\sfd)\) is an \emph{extended metric space}, if \(\sfd\) is an extended
pseudodistance that vanishes exactly on the diagonal of \(\X\times\X\) (i.e.\ for any given \(x,y\in\X\) we have that \(\sfd(x,y)=0\) if and only if \(x=y\)).
\item We say that \(\sfd\) is a \emph{pseudodistance}, and that \((\X,\sfd)\) is a \emph{pseudometric space}, if \(\sfd\) is an extended pseudodistance such
that \(\sfd(x,y)<+\infty\) for every \(x,y\in\X\).
\item We say that \(\sfd\) is a \emph{distance}, and that \((\X,\sfd)\) is a \emph{metric space}, if \(\sfd\) is both an extended distance and a pseudodistance.
\end{itemize}
An extended pseudodistance \(\sfd\) on a set \(\X\) induces an equivalence relation \(\mathcal R_\sfd\) on \(\X\) as follows: given any \(x,y\in\X\), we declare
that \(x\mathcal R_\sfd y\) if and only if \(\sfd(x,y)=0\). We denote the associated quotient space by \(\X/\sfd\coloneqq\X/\mathcal R_\sfd\). Moreover, \(\sfd\)
induces an extended distance \([\sfd]\) on \(\X/\sfd\) as follows:
\[
[\sfd](\pi_\sfd(x),\pi_\sfd(y))\coloneqq\sfd(x,y)\quad\text{ for every }x,y\in\X,
\]
where \(\pi_\sfd\colon\X\to\X/\sfd\) denotes the canonical projection map. We say that \([\sfd]\) is the extended distance \emph{induced by \(\sfd\)}.
Note that \(\sfd\) is a pseudodistance if and only if \([\sfd]\) is a distance.
\subsubsection*{Short maps}
We say that a map \(\varphi\colon(\X,\sfd_\X)\to(\Y,\sfd_\Y)\) between two extended pseudometric spaces \((\X,\sfd_\X)\)
and \((\Y,\sfd_\Y)\) is a \emph{short map} provided it is non-expansive, meaning that
\[
\sfd_\Y(\varphi(x),\varphi(y))\leq\sfd_\X(x,y)\quad\text{ for every }x,y\in\X.
\]
If in addition the map \(\varphi\) is also non-contractive (so that \(\sfd_\Y(\varphi(x),\varphi(y))=\sfd_\X(x,y)\)
for all \(x,y\in\X\)), then we say that \(\varphi\) is \emph{distance preserving}.
\begin{remark}\label{rmk:initial_top_Lip_metric_space}{\rm
For any metric space \((\X,\sfd)\), it can be easily checked that the initial topology of the set of all bounded short maps \(f\colon(\X,\sfd)\to\R\)
coincides with the topology induced by \(\sfd\).
\fr}\end{remark}
\subsubsection*{Completeness}
Let \((\X,\sfd)\) be an extended metric space. We remind that a net \((x_i)_{i\in I}\subseteq\X\) is called a \emph{Cauchy net}
in \((\X,\sfd)\) provided \(\inf_{i_0\in I}\sup\{\sfd(x_i,x_{i'}):i,i'\in I,i_0\preceq i,i'\}=0\). The space \((\X,\sfd)\)
is said to be \emph{complete} if every Cauchy net \((x_i)_{i\in I}\) in \((\X,\sfd)\) \emph{converges} to a (necessarily unique)
element \(x\in\X\), i.e.\ \(\inf_{i_0\in I}\sup\{\sfd(x_i,x):i\in I,i_0\preceq i\}=0\). The completeness of an extended metric
space can be equivalently checked on Cauchy sequences only.
\medskip

On a given extended metric space \((\X,\sfd)\), we define the equivalence relation \(\mathcal F_\sfd\) as follows: given any \(x,y\in\X\),
we declare that \(x\mathcal F_\sfd y\) if and only if \(\sfd(x,y)<+\infty\). We then define \(\mathscr F_\sfd\coloneqq\X/\mathcal F_\sfd\).
Note that \(({\rm Z},\sfd|_{{\rm Z}\times{\rm Z}})\) is a metric space for every \({\rm Z}\in\mathscr F_\sfd\), and \(\sfd(z,x)=+\infty\)
for all \(z\in{\rm Z}\) and \(x\in\X\setminus{\rm Z}\). Clearly, we have that \((\X,\sfd)\) is complete if and only if
\(({\rm Z},\sfd|_{{\rm Z}\times{\rm Z}})\) is complete for every \({\rm Z}\in\mathscr F_\sfd\).
\begin{proposition}[Completion of an extended metric space]\label{prop:completion_extmet}
Let \((\X,\sfd)\) be an extended metric space. Then there exist a complete extended metric space \((\bar\X,\bar\sfd)\)
and a short map \({\sf i}\colon\X\to\bar\X\) such that the following universal property holds: if
\((\Y,\sfd_\Y)\) is a complete extended metric space and \(\varphi\colon\X\to\Y\) is a short map, then there exists a unique
short map \(\bar\varphi\colon\bar\X\to\Y\) such that the diagram
\[\begin{tikzcd}
\X \arrow[r,"{\sf i}"] \arrow[dr,swap,"\varphi"] & \bar\X \arrow[d,"\bar\varphi"] \\
& \Y
\end{tikzcd}\]
commutes. Moreover, the map \({\sf i}\colon\X\to\bar\X\) is distance preserving and its image \({\sf i}(\X)\) is dense in \(\bar\X\).
\end{proposition}
\begin{proof}
For any \({\rm Z}\in\mathscr F_\sfd\), we denote by \((\bar{\rm Z},\bar\sfd_{\rm Z})\), together with \({\sf i}_{\rm Z}\colon{\rm Z}\to\bar{\rm Z}\),
the usual metric completion of \(({\rm Z},\sfd|_{{\rm Z}\times{\rm Z}})\). Then we endow \(\bar\X\coloneqq\bigsqcup_{{\rm Z}\in\mathscr F_\sfd}\bar{\rm Z}\)
with the extended distance \(\bar\sfd\), which we define as
\[
\bar\sfd(x,y)\coloneqq\left\{\begin{array}{ll}
\bar\sfd_{\rm Z}(x,y)\\
+\infty
\end{array}\quad\begin{array}{ll}
\text{ if }x,y\in\bar{\rm Z}\text{ for some }{\rm Z}\in\mathscr F_\sfd,\\
\text{ otherwise.}
\end{array}\right.
\]
Observe that \((\bar\X,\bar\sfd)\) is a complete extended metric space. We denote by \({\sf i}\colon\X\to\bar\X\)
the unique map satisfying \({\sf i}|_{\rm Z}={\sf i}_{\rm Z}\) for every \({\rm Z}\in\mathscr F_\sfd\). Since each
map \({\sf i}_{\rm Z}\colon{\rm Z}\to\bar{\rm Z}\) is distance preserving and with dense image, we deduce that also
\({\rm i}\colon\X\to\bar\X\) is distance preserving and with dense image. Finally, let us check the validity of the
universal property. Fix a complete extended metric space \((\Y,\sfd_\Y)\) and a short map \(\varphi\colon\X\to\Y\).
In particular, we know that each restricted map \(\varphi|_{\rm Z}\colon{\rm Z}\to\Y\) is a short map taking values into
some component \({\rm W}_{\rm Z}\in\mathscr F_{\sfd_\Y}\); thus, since \(({\rm W}_{\rm Z},\sfd_\Y|_{{\rm W}_{\rm Z}\times{\rm W}_{\rm Z}})\)
is a complete metric space, \(\varphi|_{\rm Z}\) can be uniquely extended to a short map
\(\bar\varphi_{\rm Z}\colon\bar{\rm Z}\to{\rm W}_{\rm Z}\subseteq\Y\). Therefore, letting
\(\bar\varphi\colon\bar\X\to\Y\) be given by \(\bar\varphi|_{\bar{\rm Z}}\coloneqq\bar\varphi_{\rm Z}\) for every
\({\rm Z}\in\mathscr F_\sfd\), it is clear that \(\bar\varphi\) is the unique short map from \(\bar\X\) to \(\Y\)
such that \(\bar\varphi\circ{\sf i}=\varphi\). The statement is achieved.
\end{proof}
\subsubsection*{Length and geodesic extended metric spaces}
Let \((\X,\sfd)\) be a given extended metric space. Then the \emph{\(\sfd\)-length} \(\ell(\gamma)\in[0,+\infty]\)
of a continuous curve \(\gamma\colon[0,1]\to\X\) is defined as
\[
\ell(\gamma)\coloneqq\sup\bigg\{\sum_{i=1}^n\sfd(\gamma_{t_i},\gamma_{t_{i-1}})\;\bigg|\;n\in\N,\,0=t_0<t_1<\ldots<t_n=1\bigg\}.
\]
If \(\ell(\gamma)<+\infty\), then the curve \(\gamma\) is said to be \emph{\(\sfd\)-rectifiable}.
The \emph{extended length distance} \(\sfd_\ell\) induced by \(\sfd\) is defined as
\begin{equation}\label{eq:def_d_ell}
\sfd_\ell(x,y)\coloneqq\inf\big\{\ell(\gamma)\;\big|\;\gamma\colon[0,1]\to\X\text{ is }\sfd\text{-rectifiable, }
(\gamma_0,\gamma_1)=(x,y)\big\}\quad\text{ for every }x,y\in\X.
\end{equation}
We say that \((\X,\sfd)\) is a \emph{length space} if \(\sfd_\ell=\sfd\). Clearly, it holds that \((\X,\sfd)\) is a length
space if and only if \(({\rm Z},\sfd|_{{\rm Z}\times{\rm Z}})\) is a length space for every \({\rm Z}\in\mathscr F_\sfd\).
\begin{remark}\label{rmk:geo_functor_mor}{\rm
If \((\X,\sfd)\) is an extended metric space, \((\Y,\sfd_\Y)\) is a length extended metric space
and \(\varphi\colon(\X,\sfd)\to(\Y,\sfd_\Y)\) is a short map, then also \(\varphi\colon(\X,\sfd_\ell)\to(\Y,\sfd_\Y)\)
is a short map. To prove it, it suffices to check that \(\sfd_\Y(\varphi(x),\varphi(y))\leq\sfd_\ell(x,y)\) for every
\(x,y\in\X\) with \(\sfd_\ell(x,y)<+\infty\). Given any \(\varepsilon>0\), we find a \(\sfd\)-rectifiable
curve \(\gamma\colon[0,1]\to\X\) such that \((\gamma_0,\gamma_1)=(x,y)\) and \(\ell(\gamma)\leq\sfd_\ell(x,y)+\varepsilon\).
Hence, for any partition \(0=t_0<t_1<\ldots<t_n=1\) of the interval \([0,1]\), we can estimate
\[
\sum_{i=1}^n\sfd_\Y(\varphi(\gamma_{t_i}),\varphi(\gamma_{t_{i-1}}))\leq\sum_{i=1}^n\sfd(\gamma_{t_i},\gamma_{t_{i-1}})\leq\ell(\gamma)\leq\sfd_\ell(x,y)+\varepsilon.
\]
Taking the supremum over all partitions of \([0,1]\) and exploiting the fact that \(\sfd_\Y\) is a length extended distance,
we deduce that \(\sfd_\Y(\varphi(x),\varphi(y))\leq\ell(\varphi\circ\gamma)\leq\sfd_\ell(x,y)+\varepsilon\),
whence the sought inequality \(\sfd_\Y(\varphi(x),\varphi(y))\leq\sfd_\ell(x,y)\) follows by the arbitrariness
of \(\varepsilon>0\).
\fr}\end{remark}
By a \emph{geodesic} \(\gamma\colon[0,1]\to\X\) in an extended metric space \((\X,\sfd)\) we mean a continuous curve
such that \(\ell(\gamma)=\sfd(\gamma_0,\gamma_1)<+\infty\). We say that \((\X,\sfd)\) is a \emph{geodesic space} if
for every \(x,y\in\X\) with \(\sfd(x,y)<+\infty\) there exists a geodesic \(\gamma\colon[0,1]\to\X\) in \((\X,\sfd)\)
joining \(x\) and \(y\), i.e.\ \((\gamma_0,\gamma_1)=(x,y)\). Clearly, every geodesic space is a length space, and
\((\X,\sfd)\) is a geodesic space if and only if \(({\rm Z},\sfd|_{{\rm Z}\times{\rm Z}})\) is a geodesic space for
every \({\rm Z}\in\mathscr F_\sfd\).
\subsection{Limits and colimits in \texorpdfstring{\({\bf Set}\)}{Set}, \texorpdfstring{\({\bf Top}\)}{Top}
and \texorpdfstring{\({\bf Ext\Psi Met}\)}{ExtPMet}}\label{s:lim_Set_Top_ExtPMet}
In this section, we focus on three categories: of sets, of topological spaces, and of extended pseudometric spaces.
More specifically, we discuss their bicompleteness, together with an explicit description of all limits and colimits;
cf.\ Appendix \ref{app:categories} for a reminder on this terminology. The bicompleteness of the category of sets
and of the one of topological spaces is shown e.g.\ in \cite{riehl2017category}. The bicompleteness of the category
of extended pseudometric spaces seems to be a folklore result, which can be easily proved.
\subsubsection*{Limits and colimits in \texorpdfstring{\({\bf Set}\)}{Set}}
We denote by \({\bf Set}\) the category of sets, where the morphisms between two sets \(\X\) and \(\Y\)
are given by all maps from \(\X\) to \(\Y\). It admits all (co)products and (co)equalisers:
\begin{itemize}
\item Let \(\{\X_i\}_{i\in I}\) be a set of sets. Then \((\prod\X_\star,\{\pi_i\}_{i\in I})\)
is the product of \(\{\X_i\}_{i\in I}\) in \({\bf Set}\), where \(\prod\X_\star\) denotes the Cartesian product
\(\prod_{i\in I}\X_i\) and \(\pi_i\colon\prod\X_\star\to\X_i\) the canonical projection map, given by
\(\pi_i(x_\star)\coloneqq x_i\) for every \(x_\star=(x_i)_{i\in I}\in\prod\X_\star\). 
\item Let \(\X\), \(\Y\) be sets and \(\varphi,\psi\colon\X\to\Y\) be maps. Define
\(\X_{\varphi,\psi}\coloneqq\{x\in\X:\varphi(x)=\psi(x)\}\) and denote by \(\iota_{\varphi,\psi}\colon\X_{\varphi,\psi}\to\X\)
the inclusion map. Then \((\X_{\varphi,\psi},\iota_{\varphi,\psi})\) is the equaliser of \(\varphi\) and \(\psi\) in \({\bf Set}\).
\item Let \(\{\X_i\}_{i\in I}\) be a set of sets. Then \((\coprod\X_\star,\{\iota_i\}_{i\in I})\) is the
coproduct of \(\{\X_i\}_{i\in I}\) in \({\bf Set}\), where \(\coprod\X_\star\) denotes the disjoint union
\(\bigsqcup_{i\in I}\X_i\) and \(\iota_i\colon\X_i\to\coprod\X_\star\) the inclusion map. 
\item Let \(\X\), \(\Y\) be sets and \(\varphi,\psi\colon\X\to\Y\) be maps. We define the equivalence relation
\(\mathcal R_{\varphi,\psi}\) on \(\Y\) as the smallest equivalence relation such that
\(\varphi(x)\mathcal R_{\varphi,\psi}\psi(x)\) for all \(x\in\X\) (i.e.\ the intersection of all such equivalence relations).
We denote by \(\Y_{\varphi,\psi}\) the quotient set \(\Y/\mathcal R_{\varphi,\psi}\)
and by \(\pi_{\varphi,\psi}\colon\Y\to\Y_{\varphi,\psi}\) the projection map. Then \((\Y_{\varphi,\psi},\pi_{\varphi,\psi})\)
is the coequaliser of \(\varphi\) and \(\psi\) in \({\bf Set}\).
\end{itemize}
Therefore, the category \({\bf Set}\) is bicomplete, thanks to Theorem \ref{thm:exist_thm_lim}.
\subsubsection*{Limits and colimits in \texorpdfstring{\({\bf Top}\)}{Top}}
We denote by \({\bf Top}\) the category of topological spaces, where continuous maps are the morphisms.
Then \({\bf Top}\) is a bicomplete category. More specifically:
\begin{itemize}
\item Let \(\{(\X_i,\tau_i)\}_{i\in I}\) be a set of topological spaces.
Then \(((\prod\X_\star,\otimes_{i\in I}\tau_i),\{\pi_i\}_{i\in I})\) is the product of \(\{(\X_i,\tau_i)\}_{i\in I}\)
in \({\bf Top}\), where \((\prod\X_\star,\{\pi_i\}_{i\in I})\) denotes the product of \(\{\X_i\}_{i\in I}\) in \({\bf Set}\),
while \(\otimes_{i\in I}\tau_i\) is the product topology on \(\prod\X_\star\).
\item Let \((\X,\tau_\X)\), \((\Y,\tau_\Y)\) be topological spaces and \(\varphi,\psi\colon\X\to\Y\) be continuous maps.
Then \(((\X_{\varphi,\psi},\tau_\X\llcorner\X_{\varphi,\psi}),\iota_{\varphi,\psi})\) is the equaliser of \(\varphi\)
and \(\psi\) in \({\bf Top}\), where \((\X_{\varphi,\psi},\iota_{\varphi,\psi})\) denotes the equaliser of \(\varphi\)
and \(\psi\) in \({\bf Set}\), while \(\tau_\X\llcorner\X_{\varphi,\psi}\) is the relative topology of \(\tau_\X\).
\item Let \(\{(\X_i,\tau_i)\}_{i\in I}\) be a set of topological spaces. Then
\(((\coprod\X_\star,\bigsqcup_{i\in I}\tau_i),\{\iota_i\}_{i\in I})\) is the coproduct of \(\{(\X_i,\tau_i)\}_{i\in I}\)
in \({\bf Top}\), where \((\coprod\X_\star,\{\iota_i\}_{i\in I})\) denotes the coproduct of \(\{\X_i\}_{i\in I}\) in \({\bf Set}\),
while the topology \(\bigsqcup_{i\in I}\tau_i\) on \(\coprod\X_\star\) is defined as
\[
\bigsqcup_{i\in I}\tau_i\coloneqq\bigg\{\bigsqcup_{i\in I}U_i\;\bigg|\;U_i\in\tau_i\text{ for every }i\in I\bigg\}.
\]
\item Let \((\X,\tau_\X)\), \((\Y,\tau_\Y)\) be topological spaces and \(\varphi,\psi\colon\X\to\Y\) be continuous maps.
Then \(((\Y_{\varphi,\psi},\tau_{\varphi,\psi}),\pi_{\varphi,\psi})\) is the coequaliser of \(\varphi\) and \(\psi\)
in \({\bf Top}\), where \((\Y_{\varphi,\psi},\pi_{\varphi,\psi})\) denotes the coequaliser of \(\varphi\) and \(\psi\)
in \({\bf Set}\), while \(\tau_{\varphi,\psi}\) is the quotient topology of \(\tau_\Y\) on
\(\Y_{\varphi,\psi}=\Y/\mathcal R_{\varphi,\psi}\).
\end{itemize}
\subsubsection*{Limits and colimits in \texorpdfstring{\({\bf Ext\Psi Met}\)}{ExtPMet}}
We denote by \({\bf Ext\Psi Met}\) the category of extended pseudometric spaces, where short maps are the morphisms.
Then \({\bf Ext\Psi Met}\) is a bicomplete category. More specifically:
\begin{itemize}
\item Let \(\{(\X_i,\sfd_i)\}_{i\in I}\) be a set of extended pseudometric spaces.
Then \(((\prod\X_\star,\otimes_{i\in I}\sfd_i),\{\pi_i\}_{i\in I})\) is the product of \(\{(\X_i,\sfd_i)\}_{i\in I}\)
in \({\bf Ext\Psi Met}\), where \((\prod\X_\star,\{\pi_i\}_{i\in I})\) denotes the product of \(\{\X_i\}_{i\in I}\)
in \({\bf Set}\), while \(\otimes_{i\in I}\sfd_i\) is the extended pseudodistance on \(\prod\X_\star\) given by
\[
(\otimes_{i\in I}\sfd_i)(x_\star,y_\star)\coloneqq\sup_{i\in I}\sfd_i(x_i,y_i)\quad\text{ for every }x_\star,y_\star\in\prod\X_\star.
\]
\item Let \((\X,\sfd_\X)\), \((\Y,\sfd_\Y)\) be extended pseudometric spaces and \(\varphi,\psi\colon\X\to\Y\) be short maps.
Then \(((\X_{\varphi,\psi},\sfd_\X|_{\X_{\varphi,\psi}\times\X_{\varphi,\psi}}),\iota_{\varphi,\psi})\) is the equaliser of
\(\varphi\) and \(\psi\) in the category \({\bf Ext\Psi Met}\), where \((\X_{\varphi,\psi},\iota_{\varphi,\psi})\) denotes
the equaliser of \(\varphi\) and \(\psi\) in \({\bf Set}\).
\item Let \(\{(\X_i,\sfd_i)\}_{i\in I}\) be a set of extended pseudometric spaces. Then
\(((\coprod\X_\star,\bigsqcup_{i\in I}\sfd_i),\{\iota_i\}_{i\in I})\) is the coproduct of \(\{(\X_i,\sfd_i)\}_{i\in I}\)
in \({\bf Ext\Psi Met}\), where \((\coprod\X_\star,\{\iota_i\}_{i\in I})\) denotes the coproduct of \(\{\X_i\}_{i\in I}\)
in \({\bf Set}\), while the extended pseudodistance \(\bigsqcup_{i\in I}\sfd_i\) on \(\coprod\X_\star\) is defined as
\[
\Big(\bigsqcup_{i\in I}\sfd_i\Big)(x,y)\coloneqq
\left\{\begin{array}{ll}
\sfd_i(x,y)\\
+\infty
\end{array}\quad\begin{array}{ll}
\text{ if }x,y\in\X_i\text{ for some }i\in I,\\
\text{ otherwise.}
\end{array}\right.
\]
\item Let \((\X,\sfd_\X)\), \((\Y,\sfd_\Y)\) be extended pseudometric spaces and \(\varphi,\psi\colon\X\to\Y\) be short maps.
Let \((\Y_{\varphi,\psi},\pi_{\varphi,\psi})\) denote the coequaliser of \(\varphi\) and \(\psi\) in \({\bf Set}\).
Let us define the extended pseudodistance \(\sfd_{\varphi,\psi}\) on \(\Y_{\varphi,\psi}\) as
\[
\sfd_{\varphi,\psi}(\pi_{\varphi,\psi}(x),\pi_{\varphi,\psi}(y))\coloneqq
\inf_{(z_1,z'_1,\ldots,z_n,z'_n)}\sum_{i=1}^n\sfd_\Y(z_i,z'_i),
\]
where the infimum is taken over all \(n\in\N\) and  \((z_1,z'_1,\ldots,z_n,z'_n)\in\Y^{2n}\)
such that \(z_1=x\), \(z'_{i-1}\mathcal R_{\varphi,\psi}z_i\) for every \(i=2,\ldots,n\), and
\(z'_n=y\). Then \(((\Y_{\varphi,\psi},\sfd_{\varphi,\psi}),\pi_{\varphi,\psi})\) is the coequaliser
of \(\varphi\) and \(\psi\) in \({\bf Ext\Psi Met}\).
\end{itemize}
\subsection{Extended metric-topological spaces}
In this section, we first introduce the auxiliary notion of \emph{pre-extended pseudometric-topological space}. After that, we reiterate
some definitions and results regarding \emph{extended metric-topological spaces}, adopting the axiomatisation proposed by Ambrosio, Erbar
and Savar\'{e} in \cite[Definition 4.1]{AmbrosioErbarSavare16}. Some of the material we present is taken also from the works
\cite{Sav:22,Pas:Tai:25}, where extended metric-topological spaces have been researched further.
\subsubsection*{Pre-e.pm.t.\ spaces}
First, let us introduce the concept of pre-extended (pseudo)metric-topological space, by which we simply mean a topological space equipped with an extended
(pseudo)distance, with no compatibility requirements whatsoever.
\begin{definition}[Pre-extended pseudometric-topological space]\label{def:pre-e.pm.t.}
We say that a triple \((\X,\tau,\sfd)\) is a \emph{pre-extended pseudometric-topological space}, or a \emph{pre-e.pm.t.\ space} for short, provided
\begin{align*}
(\X,\tau)&\quad\text{ is a topological space,}\\
(\X,\sfd)&\quad\text{ is an extended pseudometric space}.
\end{align*}
If \((\X,\sfd)\) is an extended metric space, then we say that \((\X,\tau,\sfd)\) is a \emph{pre-extended metric-topological space}, or a \emph{pre-e.m.t.\ space}.
\end{definition}
\begin{remark}\label{rmk:d_lsc_implies_Haus}{\rm
If \((\X,\tau,\sfd)\) is a pre-e.m.t.\ space such that \(\sfd\colon\X\times\X\to[0,+\infty]\) is a \(\tau\otimes\tau\)-lower
semicontinuous function, then \((\X,\tau)\) is Hausdorff. Indeed, given any \( \bar x,\bar y\in\X\) with \(\bar x\neq\bar y\),
we have that \((\bar x,\bar y)\in\{(x,y)\in\X\times\X:\sfd(x,y)>0\}\in\tau\otimes\tau\), thus we find \(U,V\in\tau\)
such that \(\bar x\in U\), \(\bar y\in V\) and \(U\times V\subseteq\{(x,y)\in\X\times\X:\sfd(x,y)>0\}\), which means
that \(U\cap V=\varnothing\), as desired.
\fr}\end{remark}
\begin{definition}[Continuous-short map]
Let \((\X,\tau_\X,\sfd_\X)\), \((\Y,\tau_\Y,\sfd_\Y)\) be pre-e.pm.t.\ spaces. Then we say that a map \(\varphi\colon\X\to\Y\)
is a \emph{continuous-short map} provided
\begin{align*}
\varphi\colon(\X,\tau_\X)\to(\Y,\tau_\Y)&\quad\text{ is continuous,}\\
\varphi\colon(\X,\sfd_\X)\to(\Y,\sfd_\Y)&\quad\text{ is a short map}.
\end{align*}
Moreover, we denote by \(CS((\X,\tau_\X,\sfd_\X);(\Y,\tau_\Y,\sfd_\Y))\) the collection of all continuous-short maps
from \((\X,\tau_\X,\sfd_\X)\) to \((\Y,\tau_\Y,\sfd_\Y)\).
\end{definition}

Let \((\X,\tau,\sfd)\) be a pre-e.pm.t.\ space. In accordance with \cite{Sav:22}, we shall utilise the notation
\[
\LIP_{b,1}(\X,\tau,\sfd)\coloneqq CS((\X,\tau,\sfd);\R)\cap C_b(\X,\tau),
\]
where the target space \(\R\) is endowed with the Euclidean topology and distance.
We say that the extended pseudodistance \(\sfd\) \emph{can be \(\tau\)-recovered} provided it holds that
\begin{equation}\label{eq:d_recover_from_tau}
\sfd(x,y)=\sup\big\{|f(x)-f(y)|\;\big|\;f\in\LIP_{b,1}(\X,\tau,\sfd)\big\}\quad\text{ for every }x,y\in\X.
\end{equation}
Note that \eqref{eq:d_recover_from_tau} is equivalent to
\(\sfd(x,y)\leq\sup\{|f(x)-f(y)|:f\in\LIP_{b,1}(\X,\tau,\sfd)\}\) for every \(x,y\in\X\).
\begin{remark}\label{rmk:d_lsc}{\rm
If \((\X,\tau,\sfd)\) is a pre-e.pm.t.\ space such that \(\sfd\) can be \(\tau\)-recovered, then
\[
\sfd\colon\X\times\X\to[0,+\infty]\quad\text{ is }\tau\otimes\tau\text{-lower semicontinuous.}
\]
Indeed, \(\X\times\X\ni(x,y)\mapsto|f(x)-f(y)|\in\R\) is \(\tau\otimes\tau\)-continuous for all
\(f\in\LIP_{b,1}(\X,\tau,\sfd)\), thus \(\sfd\) is \(\tau\otimes\tau\)-lower semicontinuous since
it is a supremum of \(\tau\otimes\tau\)-continuous functions.
\fr}\end{remark}
\begin{definition}[Embedding of pre-e.pm.t.\ spaces]
Let \((\X,\tau_\X,\sfd_\X)\), \((\Y,\tau_\Y,\sfd_\Y)\) be pre-e.pm.t.\ spaces. Then we say that a map \(\iota\colon\X\to\Y\)
is an \emph{embedding} if the following conditions hold:
\begin{itemize}
\item[\(\rm i)\)] \(\iota\colon(\X,\tau_\X)\to(\Y,\tau_\Y)\) is a homeomorphism onto its image \(\iota(\X)\),
\item[\(\rm ii)\)] \(\iota\colon(\X,\sfd_\X)\to(\Y,\sfd_\Y)\) is distance preserving.
\end{itemize}
\end{definition}
\subsubsection*{Definition of e.m.t.\ space}
We recall the notion of extended (pseudo)metric-topological space \cite{AmbrosioErbarSavare16}:
\begin{definition}[Extended (pseudo)metric-topological space]\label{def:emt}
We say that a pre-e.pm.t.\ space \((\X,\tau,\sfd)\) is an \emph{extended pseudometric-topological space}, or an \emph{e.pm.t.\ space} for short, provided:
\begin{itemize}
\item[\(\rm i)\)] \(\tau\) is the initial topology of \(\LIP_{b,1}(\X,\tau,\sfd)\),
\item[\(\rm ii)\)] \(\sfd\) can be \(\tau\)-recovered.
\end{itemize}
If in addition \((\X,\sfd)\) is an extended metric space, then we say that \((\X,\tau,\sfd)\) is an \emph{extended metric-topological space}, or an \emph{e.m.t.\ space} for short.
\end{definition}

Note that if \((\X,\tau,\sfd)\) is an e.pm.t.\ space, then \(\tau\) is a fortiori the initial topology of \(C_b(\X,\tau)\), so that \((\X,\tau)\) is completely regular.
Assuming in addition that \((\X,\tau,\sfd)\) is an e.m.t.\ space, the fact that \(\sfd\) can be \(\tau\)-recovered implies that 
\(\sfd\colon\X\times\X\to[0,+\infty]\) is \(\tau\otimes\tau\)-lower semicontinuous (by Remark \ref{rmk:d_lsc}), so that in particular
\((\X,\tau)\) is Hausdorff (thus, a Tychonoff space) by Remark \ref{rmk:d_lsc_implies_Haus}.
\medskip

We say that \((\X,\tau,\sfd)\) is a \emph{one-point e.m.t.\ space} if \(\X\) is a singleton, so that necessarily \(\tau=\{\varnothing,\X\}\) and \(\sfd\colon\X\times\X\to\{0\}\).
We also have the \emph{empty e.m.t.\ space} \((\varnothing,\{\varnothing\},\varnothing\times\varnothing\to[0,+\infty])\).
\begin{remark}\label{rmk:sub_emt}{\rm
Let \((\X,\tau,\sfd)\) be an e.m.t.\ space and fix any \(E\subseteq\X\). Letting \(\tau\llcorner E\coloneqq\{U\cap E:U\in\tau\}\) be the relative topology,
it can be readily checked that \((E,\tau\llcorner E,\sfd|_{E\times E})\) is an e.m.t.\ space.
\fr}\end{remark}

The following result was obtained in \cite[Corollary 3.3]{Pas:Tai:25}, as a corollary of \cite[Theorem 2.4]{Mat:00}.
\begin{theorem}[Extension theorem]\label{thm:ext_CS}
Let \((\X,\tau,\sfd)\) be an e.m.t.\ space. Let \(K\subseteq\X\) be a \(\tau\)-compact set and \(f\in\LIP_{b,1}(K,\tau\llcorner K,\sfd|_{K\times K})\).
Then there exists \(\bar f\in\LIP_{b,1}(\X,\tau,\sfd)\) such that \(\bar f|_K=f\).
\end{theorem}

Let us now prove that, given a compact pre-e.m.t.\ space \((\X,\tau,\sfd)\), the \(\tau\otimes\tau\)-lower semicontinuity of \(\sfd\)
is not only a necessary but also a sufficient condition for \((\X,\tau,\sfd)\) being an e.m.t.\ space:
\begin{proposition}\label{prop:charact_cpt_emt}
Let \((\X,\tau,\sfd)\) be a pre-e.m.t.\ space such that \((\X,\tau)\) is compact. Then \((\X,\tau,\sfd)\) is an e.m.t.\ space
if and only if \(\sfd\colon\X\times\X\to[0,+\infty]\) is \(\tau\otimes\tau\)-lower semicontinuous.
\end{proposition}
\begin{proof}
Necessity follows from Remark \ref{rmk:d_lsc}. To prove sufficiency, assume \(\sfd\) is \(\tau\otimes\tau\)-lower semicontinuous.
Let us check that \(\sfd\) can be \(\tau\)-recovered. Fix any \(x,y\in\X\) and \(\lambda\in(0,\sfd(x,y))\). Then the
truncated distance \(\sfd\wedge\lambda\colon\X\times\X\to[0,\lambda]\) is \(\tau\otimes\tau\)-lower semicontinuous.
We define \(\tilde f\colon\{x,y\}\to\R\) as \(\tilde f(x)\coloneqq 0\) and \(\tilde f(y)\coloneqq\lambda\). Since
\(|\tilde f(x)-\tilde f(y)|=\lambda=(\sfd\wedge\lambda)(x,y)\), we have that \(\tilde f\colon(\X,\tau,\sfd\wedge\lambda)\to\R\)
is a continuous-short map. Therefore, we know from \cite[Corollary 2.5]{Mat:00} that there exists a continuous-short map
\(f\colon(\X,\tau,\sfd\wedge\lambda)\to[0,\lambda]\) such that \(f(x)=\tilde f(x)=0\) and \(f(y)=\tilde f(y)=\lambda\).
Since \(\sfd\wedge\lambda\leq\sfd\), we have that \(f\in\LIP_{b,1}(\X,\tau,\sfd)\). In particular, we have that
\(\sup\{|g(x)-g(y)|:g\in\LIP_{b,1}(\X,\tau,\sfd)\}\geq|f(x)-f(y)|=\lambda\). By the arbitrariness of \(\lambda,x,y\),
we conclude that \(\sfd\) can be \(\tau\)-recovered. In particular, \(\LIP_{b,1}(\X,\tau,\sfd)\) separates the points
of \(\X\), thus Remark \ref{rmk:criterion_initial_cpt} guarantees that \(\tau\) is the initial topology of \(\LIP_{b,1}(\X,\tau,\sfd)\).
All in all, we have shown that \((\X,\tau,\sfd)\) is an e.m.t.\ space. The statement is then achieved.
\end{proof}

The compactness assumption in Proposition \ref{prop:charact_cpt_emt} cannot be dropped, as this example shows:
\begin{example}\label{ex:lsc_but_not_emt}{\rm
Assume \((\X,\sfd)\) is a non-compact metric space, and \(\tilde\sfd\) is a distance on \(\X\) with \(\tilde\sfd\leq\sfd\) whose
induced topology \(\tilde\tau\) is strictly coarser than the topology \(\tau\) induced by \(\sfd\); for example, we can consider the space
\(\X=L^\infty(0,1)\) with the distance \(\sfd\) induced by \(\|\cdot\|_{L^\infty(0,1)}\) and with the distance \(\tilde\sfd\) induced
by \(\|\cdot\|_{L^1(0,1)}\). Then we claim that
\begin{align}
\label{eq:lsc_not_emt_1}
\tilde\sfd&\quad\text{ is }\tau\otimes\tau\text{-lower semicontinuous,}\\
\label{eq:lsc_not_emt_2}
(\X,\tau,\tilde\sfd)&\quad\text{ is not an e.m.t.\ space.}
\end{align}
To prove \eqref{eq:lsc_not_emt_1}: if \((x_n)_n,(y_n)_n\subseteq\X\) and \(x,y\in\X\) satisfy \(\lim_n\sfd(x_n,x)=\lim_n\sfd(y_n,y)=0\), then
\[
\tilde\sfd(x,y)\leq\varliminf_{n\to\infty}(\sfd(x,x_n)+\tilde\sfd(x_n,y_n)+\sfd(y_n,y))
=\varliminf_{n\to\infty}\tilde\sfd(x_n,y_n).
\]
To prove \eqref{eq:lsc_not_emt_2}: since \(\tilde\tau\subseteq\tau\), we have that \(\LIP_{b,1}(\X,\tau,\tilde\sfd)\) coincides with the set of
all bounded short maps \(f\colon(\X,\tilde\sfd)\to\R\). Hence, Remark \ref{rmk:initial_top_Lip_metric_space} ensures that the initial topology of
\(\LIP_{b,1}(\X,\tau,\tilde\sfd)\) coincides with \(\tilde\tau\), which is strictly coarser than \(\tau\), thus \((\X,\tau,\tilde\sfd)\)
is not an e.m.t.\ space.
\fr}\end{example}
\subsubsection*{Metrically-complete e.m.t.\ spaces}
We say that an e.m.t.\ space \((\X,\tau,\sfd)\) is \emph{metrically complete} if the extended metric space \((\X,\sfd)\) is complete,
while we say that it is a \emph{compact e.m.t.\ space} if \((\X,\tau)\) is compact. As we are going to see below, a sufficient
condition for metric completeness is the (topological) compactness.
\medskip

Recall that, for an e.m.t.\ space \((\X,\tau,\sfd)\), the family of pseudodistances \(\{\delta_f:f\in\LIP_{b,1}(\X,\tau,\sfd)\}\)
generates a uniform structure on \(\X\) that is compatible with the topology \(\tau\); see \cite[Section 2.3.4]{Pas:Tai:25}. 
Here, \(\delta_f\) is defined as \(\delta_f(x,y)\coloneqq|f(x)-f(y)|\) for every \(x,y\in\X\). We then obtain the following
comparison between metric completeness and completeness in the sense of uniform spaces:
\begin{proposition}\label{prop:cpt_emt_are_mc}
Let \((\X,\tau,\sfd)\) be an e.m.t.\ space. If \((\X,\tau,\sfd)\) is complete as a uniform space, then it is metrically complete.
In particular, every compact e.m.t.\ space is metrically complete.
\end{proposition}
\begin{proof}
Suppose \((\X,\tau,\sfd)\) is complete as a uniform space and let \((x_i)_{i\in I}\) be a Cauchy net in \((\X,\sfd)\).
By the definition of \(\delta_f\), we have that \((x_i)_{i\in I}\) is Cauchy with respect to \(\delta_f\) for all
\(f\in\LIP_{b,1}(\X,\tau,\sfd)\), thus with respect to the uniform structure generated by
\(\{\delta_f:f\in \LIP_{b,1}(\X,\tau,\sfd)\}\). By completeness, there exists \(x\in\X\) so that \((x_i)_{i\in I}\)
converges to \(x\) in \((\X,\tau)\). Let us show that it also converges with respect to \(\sfd\). Fix any \(\varepsilon>0\).
Since \((x_i)_{i\in I}\) is Cauchy, there exists \(i_\varepsilon\in I\) so that \(\sfd(x_i,x_j)<\varepsilon\) for every
\(i,j\in I\) with \(i_\varepsilon\preceq i,j\). Then the \(\tau\)-lower semicontinuity of \(\sfd(\cdot,x_i)\) yields
\[
\sfd(x,x_i)\leq\varliminf_{j\in I}\sfd(x_j,x_i)\leq\varepsilon\quad\text{ for every }i\in I\text{ with }i_\varepsilon\preceq i.
\]
Therefore, \((x_i)_{i\in I}\) converges in \((\X,\sfd)\), thus \((\X,\tau,\sfd)\) is metrically complete. The final
claim follows from the fact that the unique compatible uniformity of a compact space is complete, see \cite{Kelley75}.
\end{proof}
\subsubsection*{Length and geodesic e.m.t.\ spaces}
We say that an e.m.t.\ space \((\X,\tau,\sfd)\) is a \emph{length e.m.t.\ space} (resp.\ a \emph{geodesic e.m.t.\ space})
provided \((\X,\sfd)\) is a length space (resp.\ a geodesic space).
\medskip

As it was proved in \cite[Theorem 2.3.2(c)]{Sav:22}, the following implication holds: letting \(\sfd_\ell\) be defined as in \eqref{eq:def_d_ell},
we have that
\begin{equation}\label{eq:cpt_gives_length_emt}
(\X,\tau,\sfd)\text{ is a compact e.m.t.\ space}\quad\Longrightarrow\quad(\X,\tau,\sfd_\ell)\text{ is a geodesic e.m.t.\ space.}
\end{equation}
The next example shows that the compactness assumption in \eqref{eq:cpt_gives_length_emt} cannot be dropped, meaning that for a non-compact e.m.t.\ space
\((\X,\tau,\sfd)\) it can happen that \((\X,\tau,\sfd_\ell)\) is not an e.m.t.\ space.
\begin{example}\label{ex:not_emt}{\rm
Consider the set \(\X\coloneqq\{(0,0)\}\cup\{(1,0\})\cup\bigcup_{n\in\N}([0,1]\times\{1/n\})\subseteq\R^2\), together with the Euclidean topology
\(\tau\) and distance \(\sfd\). Since no continuous curve in \((\X,\sfd)\) joins \(a\coloneqq(0,0)\) and \(b\coloneqq(1,0)\),
we have that \(\sfd_\ell(a,b)=+\infty\). For any \(n\in\N\), the points \(a_n\coloneqq(0,1/n)\) and \(b_n\coloneqq(1,1/n)\) satisfy \(\sfd_\ell(a_n,b_n)=1\).
As \(a_n\to a\) and \(b_n\to b\) with respect to \(\tau\), it follows that \(\sfd_\ell\) is not \(\tau\otimes\tau\)-lower
semicontinuous, thus \((\X,\tau,\sfd_\ell)\) is not an e.m.t.\ space (Remark \ref{rmk:d_lsc}). Note also that
\(\sfd_\ell(x,\cdot)\) is \(\tau\)-lower semicontinuous for every \(x\in\X\). This shows that the
\(\tau\otimes\tau\)-lower semicontinuity of an extended distance cannot be checked `separately' (i.e.\ on a single variable).
\fr}\end{example}
\section{The category \texorpdfstring{\({\bf ExtMetTop}\)}{ExtMetTop}}
\subsection{Catalogue of categories}
Let us collect below those categories that are relevant to us:
\begin{itemize}
\item We denote by \({\bf Tych}\) the full subcategory of \({\bf Top}\) whose objects are all Tychonoff spaces.
\item We define \({\bf PreExt\Psi MetTop}\) as the category whose objects are the pre-e.pm.t.\ spaces and whose morphisms are the continuous-short maps.
\item We define \({\bf ExtMetTop}\) as the full subcategory of \({\bf PreExt\Psi MetTop}\) whose objects are the e.m.t.\ spaces.
\item We define \({\bf ExtMetTop}_{\rm mc}\) as the full subcategory of \({\bf ExtMetTop}\) whose objects are
all metrically-complete e.m.t.\ spaces.
\item Given any real number \(\lambda\in(0,\infty)\), we denote by \({\bf MetTop}_{\leq\lambda}\) the full subcategory of \({\bf ExtMetTop}\)
whose objects are those bounded e.m.t.\ spaces having diameter that does not exceed \(\lambda\). Namely, the objects of \({\bf MetTop}_{\leq\lambda}\) are those extended metric-topological
spaces \((\X,\tau,\sfd)\) such that \(\sfd(x,y)\leq\lambda\) for every \(x,y\in\X\).
\item We define \({\bf CptExtMetTop}\) as the full subcategory of \({\bf ExtMetTop}\) whose objects are all compact e.m.t.\ spaces.
By virtue of Proposition \ref{prop:cpt_emt_are_mc}, we have that \({\bf CptExtMetTop}\) is also a full subcategory of \({\bf ExtMetTop}_{\rm mc}\).
\item \({\bf ExtMetTop}_{\rm len}\) is the category of length e.m.t.\ spaces, while \({\bf CptExtMetTop}_{\rm geo}\) is the category of
compact geodesic e.m.t.\ spaces.
\end{itemize}
\begin{remark}\label{rmk:equiv_PreExtPMetTop}{\rm
The category \({\bf PreExt\Psi MetTop}\) could be equivalently defined as
\[
{\bf PreExt\Psi MetTop}\coloneqq{\bf SP}(U_{\bf Top},U_{\bf Ext\Psi Met}),
\]
where \(U_{\bf Top}\colon{\bf Top}\to{\bf Set}\) and \(U_{\bf Ext\Psi Met}\colon{\bf Ext\Psi Met}\to{\bf Set}\)
denote the forgetful functors. We refer to Definition \ref{def:strict_pullback} for the definition of the strict
pullback \({\bf SP}(U_{\bf Top},U_{\bf Ext\Psi Met})\).
\fr}\end{remark}
\begin{remark}[Comparison with other notions of compact spaces]\label{rmk:comparison_MetCH}{\rm
Let \((\X,\tau,\sfd)\) be a pre-e.m.t.\ space such that \((\X,\tau)\) is compact. Then we claim that the following conditions are equivalent:
\begin{itemize}
\item[\(\rm i)\)] \((\X,\tau,\sfd)\) is a compact e.m.t.\ space.
\item[\(\rm ii)\)] \((\X,\tau,\sfd)\) is a \emph{symmetric separated metric compact Hausdorff space} in the
sense of \cite[Definition 2.1, Remark 2.8]{AbbadiniHofmann25}, i.e.\ \(\sfd\colon\X\times\X\to[0,+\infty]\)
is \(\tau\otimes\tau\)-lower semicontinuous.
\item[\(\rm iii)\)] \((\X,\tau,\sfd)\) is a \emph{compact topometric space} in the sense of \cite[Definition 4.11]{BU10},
i.e.\ the set \(\{x\in\X:\sfd(x,C)\leq r\}\) is \(\tau\)-closed whenever \(C\subseteq\X\) is \(\tau\)-closed and \(r\in(0,+\infty)\).
\end{itemize}
The equivalence between i) and ii) is the content of Proposition \ref{prop:charact_cpt_emt}, while the equivalence
between ii) and iii) follows from \cite[Lemma 1.9]{Ben08b}.
In particular, \({\bf CptExtMetTop}\) coincides with the category \({\bf MetCH_{sep,sym}}\) of symmetric separated metric
compact Hausdorff spaces and \emph{continuous non-expansive maps} that is considered e.g.\ in \cite[Remark 2.8]{AbbadiniHofmann25}.
\fr}\end{remark}

In the following statement, we collect several basic properties of the category \({\bf ExtMetTop}\):
\begin{proposition}\label{prop:basic_ExtMetTop}
The following properties hold:
\begin{itemize}
\item[\(\rm i)\)] In \({\bf ExtMetTop}\), the empty e.m.t.\ space is the initial object, while any one-point e.m.t.\ space is a terminal object.
In particular, \({\bf ExtMetTop}\) does not have a zero object.
\item[\(\rm ii)\)] A morphism \(\varphi\colon(\X,\tau_\X,\sfd_\X)\to(\Y,\tau_\Y,\sfd_\Y)\) in \({\bf ExtMetTop}\) is an isomorphism if and only if
it is a distance-preserving homeomorphism (i.e.\ a surjective embedding of e.m.t.\ spaces). 
\item[\(\rm iii)\)] The category \({\bf ExtMetTop}\) is locally small, but it is not essentially small.
\item[\(\rm iv)\)] A morphism \(\varphi\colon(\X,\tau_\X,\sfd_\X)\to(\Y,\tau_\Y,\sfd_\Y)\) in \({\bf ExtMetTop}\) is a monomorphism if and only if
it is injective.
\item[\(\rm v)\)] A morphism \(\varphi\colon(\X,\tau_\X,\sfd_\X)\to(\Y,\tau_\Y,\sfd_\Y)\) in \({\bf ExtMetTop}\) is an epimorphism if and only if
its image \(\varphi(\X)\) is dense in \((\Y,\tau_\Y)\).
\item[\(\rm vi)\)] The category \({\bf ExtMetTop}\) is not balanced.
\end{itemize}
\end{proposition}
\begin{proof}
\ \\
\({\bf i)}\) For any \((\X,\tau,\sfd)\in{\rm Obj}({\bf ExtMetTop})\), we trivially have that the empty function from the empty e.m.t.\ space to
\((\X,\tau,\sfd)\) is continuous-short. Moreover, given a one-point e.m.t.\ space \((\Y,\tau_\Y,\sfd_\Y)\),  the constant
map from \(\X\) to \(\Y\) (when \(\X\neq\varnothing\), or the empty function from \(\X\) to \(\Y\) when \(\X=\varnothing\)) is the (unique)
continuous-short map. Thus, the empty e.m.t.\ space is the initial object of \({\bf ExtMetTop}\), while each one-point e.m.t.\ space
is a terminal object of \({\bf ExtMetTop}\).\\
\({\bf ii)}\) If \(\varphi\colon(\X,\tau_\X,\sfd_\X)\to(\Y,\tau_\Y,\sfd_\Y)\) is an isomorphism in \({\bf ExtMetTop}\) with inverse
\(\varphi^{-1}\), then we have in particular that \(\varphi\colon\X\to\Y\) is invertible (i.e.\ an
isomorphism in the category \({\bf Set}\) of sets) and \(\varphi^{-1}\colon\Y\to\X\) is its inverse in the set-theoretic sense. Since both \(\varphi\colon(\X,\tau_\X)\to(\Y,\tau_\Y)\)
and \(\varphi^{-1}\colon(\Y,\tau_\Y)\to(\X,\tau_\X)\) are continuous, we have that \(\varphi\colon(\X,\tau_\X)\to(\Y,\tau_\Y)\) is a homeomorphism.
Since \(\varphi\colon(\X,\sfd_\X)\to(\Y,\sfd_\Y)\) and \(\varphi^{-1}\colon(\Y,\sfd_\Y)\to(\X,\sfd_\X)\) are short maps, for any \(x,y\in\X\) we have
\[
\sfd_\Y(\varphi(x),\varphi(y))\leq\sfd_\X(x,y)=\sfd_\X(\varphi^{-1}(\varphi(x)),\varphi^{-1}(\varphi(y)))\leq\sfd_\Y(\varphi(x),\varphi(y)),
\]
which yields \(\sfd_\Y(\varphi(x),\varphi(y))=\sfd_\X(x,y)\). Hence, \(\varphi\colon(\X,\sfd_\X)\to(\Y,\sfd_\Y)\) is distance preserving.
Conversely, each distance-preserving homeomorphism is clearly an isomorphism in \({\bf ExtMetTop}\).\\
\({\bf iii)}\) Let \((\X,\tau_\X,\sfd_\X),(\Y,\tau_\Y,\sfd_\Y)\in{\rm Obj}({\bf ExtMetTop})\). Each morphism \(\varphi\colon(\X,\tau_\X,\sfd_\X)\to(\Y,\tau_\Y,\sfd_\Y)\)
in \({\bf ExtMetTop}\) is in particular a map between the sets \(\X\), \(\Y\). This guarantees that \({\bf ExtMetTop}\) is a locally-small category.
Next, fix a cardinal \(\kappa\) and consider the e.m.t.\ space \((\X_\kappa,\tau_\kappa,\sfd_\kappa)\),
where \(\X_\kappa\coloneqq\{0,1\}^\kappa\) is equipped with the product topology \(\tau_\kappa\) (each factor \(\{0,1\}\) being equipped with the discrete topology)
and with the discrete distance \(\sfd_\kappa\) (defined as \(\sfd_\kappa(a,b)\coloneqq 1\) whenever \(a,b\in\X_\kappa\) are distinct). It is easy to check that
\((\X_\kappa,\tau_\kappa,\sfd_\kappa)\) is an e.m.t.\ space (cf.\ Proposition \ref{prop:charact_cpt_emt}, or Section \ref{s:Tych_ExtMet} below).
Given that for two different cardinals \(\kappa_1\) and \(\kappa_2\) we have that \(\X_{\kappa_1}\) and \(\X_{\kappa_2}\) have different cardinalities, and taking ii) into account,
we know a fortiori that no isomorphism \((\X_{\kappa_1},\tau_{\kappa_1},\sfd_{\kappa_1})\to(\X_{\kappa_2},\tau_{\kappa_2},\sfd_{\kappa_2})\) in \({\bf ExtMetTop}\) exists.
Since the collection of all cardinals is a proper class, we infer that the skeleton of \({\bf ExtMetTop}\) is not a small category, or in other words that \({\bf ExtMetTop}\)
is not essentially small.\\
\({\bf iv)}\) Assume \(\varphi\colon(\X,\tau_\X,\sfd_\X)\to(\Y,\tau_\Y,\sfd_\Y)\) is a monomorphism in \({\bf ExtMetTop}\). We aim to show that \(\varphi\colon\X\to\Y\) is injective.
Let us argue by contradiction: suppose \(\varphi(x)=\varphi(y)\) for some \(x,y\in\X\) with \(x\neq y\). Fix a one-point e.m.t.\ space \((\{o\},\tau_o,\sfd_o)\) and define the
maps \(\psi_1,\psi_2\colon\{o\}\to\X\) as \(\psi_1(o)\coloneqq x\) and \(\psi_2(o)\coloneqq y\). Trivially, \(\psi_1\), \(\psi_2\) are different continuous-short maps and
\(\varphi\circ\psi_1=\varphi\circ\psi_2\), in contradiction with the assumption that \(\varphi\) is a monomorphism. Therefore, \(\varphi\) is injective.
Conversely, each injective continuous-short map is clearly a monomorphism in \({\bf ExtMetTop}\).\\
\({\bf v)}\) Assume \(\varphi\colon(\X,\tau_\X,\sfd_\X)\to(\Y,\tau_\Y,\sfd_\Y)\) is an epimorphism in \({\bf ExtMetTop}\). We aim to show that the map \(\varphi\colon\X\to(\Y,\tau_\Y)\)
has dense image. Let us argue by contradiction: suppose the closure \(C\) of \(\varphi(\X)\) in \((\Y,\tau_\Y)\) satisfies \(C\neq\Y\). Fix some \(y\in\Y\setminus C\). Applying
Theorem \ref{thm:alt_char_emt} iv), we can thus find a function \(\psi_1\in\LIP_{b,1}(\Y,\tau_\Y,\sfd_\Y)\) such that \(\psi_1|_C=0\) and \(\psi_1(y)\neq 0\).
In particular, \(\psi_1\) is a continuous-short map from \((\Y,\tau_\Y,\sfd_\Y)\) to the e.m.t.\ space \(\R\) (equipped with the Euclidean topology and distance).
Letting \(\psi_2\colon\Y\to\R\) be the null function \(\psi_2\coloneqq 0\), we then have that \(\psi_1,\psi_2\colon(\Y,\tau_\Y,\sfd_\Y)\to\R\) are different morphisms in 
\({\bf ExtMetTop}\) satisfying \(\psi_1\circ\varphi=0=\psi_2\circ\varphi\) (as \(\varphi(\X)\subseteq C\)), in contradiction with the assumption that \(\varphi\) is an epimorphism.
Therefore, \(\varphi(\X)\) is dense in \((\Y,\tau_\Y)\). Conversely, assume \(\varphi\colon\X\to(\Y,\tau_\Y)\) has dense image. Given an e.m.t.\ space \(({\rm Z},\tau_{\rm Z},\sfd_{\rm Z})\)
and continuous-short maps \(\psi_1,\psi_2\colon(\Y,\tau_\Y,\sfd_\Y)\to({\rm Z},\tau_{\rm Z},\sfd_{\rm Z})\) satisfying \(\psi_1\circ\varphi=\psi_2\circ\varphi\), we have that
\(\psi_1|_{\varphi(\X)}=\psi_2|_{\varphi(\X)}\). Since \(\psi_1,\psi_2\colon(\Y,\tau_\Y)\to({\rm Z},\tau_{\rm Z})\) are continuous, \(\varphi(\X)\) is dense in \((\Y,\tau_\Y)\)
and \(({\rm Z},\tau_{\rm Z})\) is Hausdorff, we deduce that \(\psi_1=\psi_2\), showing that \(\varphi\) is an epimorphism.\\
\({\bf vi)}\) We consider \(\mathbb Q\) and \(\R\) (together with the Euclidean topology and distance) as e.m.t.\ spaces. The inclusion \(\mathbb Q\hookrightarrow\R\) is a bimorphism
(as it is injective and with dense image), but it is not an isomorphism (as it is not surjective). This shows that \({\bf ExtMetTop}\) is not a balanced category.
\end{proof}

By arguing as in the proof of Proposition \ref{prop:basic_ExtMetTop} ii), it is possible to show that even in the category \({\bf PreExt\Psi MetTop}\)
the isomorphisms are exactly the distance-preserving homeomorphisms.
\begin{corollary}
The following properties hold:
\begin{itemize}
\item[\(\rm i)\)] In \({\bf CptExtMetTop}\), the empty e.m.t.\ space is the initial object, while any one-point e.m.t.\ space is a terminal object.
In particular, \({\bf CptExtMetTop}\) does not have a zero object.
\item[\(\rm ii)\)] A morphism \(\varphi\colon(\X,\tau_\X,\sfd_\X)\to(\Y,\tau_\Y,\sfd_\Y)\) in \({\bf CptExtMetTop}\) is an isomorphism if and only if
it is a distance-preserving homeomorphism. 
\item[\(\rm iii)\)] The category \({\bf CptExtMetTop}\) is locally small, but it is not essentially small.
\item[\(\rm iv)\)] A morphism \(\varphi\colon(\X,\tau_\X,\sfd_\X)\to(\Y,\tau_\Y,\sfd_\Y)\) in \({\bf CptExtMetTop}\) is a monomorphism if and only if
it is injective.
\item[\(\rm v)\)] A morphism \(\varphi\colon(\X,\tau_\X,\sfd_\X)\to(\Y,\tau_\Y,\sfd_\Y)\) in \({\bf CptExtMetTop}\) is an epimorphism if and only if
it is surjective.
\item[\(\rm vi)\)] The category \({\bf CptExtMetTop}\) is balanced.
\end{itemize}
\end{corollary}
\begin{proof}
\ \\
\({\bf i)}\) Since the empty e.m.t.\ space and all one-point e.m.t.\ spaces are compact, it follows from Proposition \ref{prop:basic_ExtMetTop} i).\\
\({\bf ii)}\) It follows from Proposition \ref{prop:basic_ExtMetTop} ii).\\
\({\bf iii)}\) Since \({\bf CptExtMetTop}\) is a subcategory of \({\bf ExtMetTop}\), which is locally small by Proposition \ref{prop:basic_ExtMetTop} iii),
it is clear that \({\bf CptExtMetTop}\) is locally small. On the other hand, the topological spaces \((\X_\kappa,\tau_\kappa)\) that are constructed in the
proof of Proposition \ref{prop:basic_ExtMetTop} iii) are compact by Tychonoff's theorem, whence it follows that \({\bf CptExtMetTop}\) is not essentially small.\\
\({\bf iv)}\) It follows from the proof of Proposition \ref{prop:basic_ExtMetTop} iv), since one-point e.m.t.\ spaces are compact.\\
\({\bf v)}\) It follows from Proposition \ref{prop:basic_ExtMetTop} v) that the epimorphisms in \({\bf CptExtMetTop}\) are exactly the continuous-short maps
having dense image. However, if \(\varphi\colon(\X,\tau_\X,\sfd_\X)\to(\Y,\tau_\Y,\sfd_\Y)\) is a continuous-short map between compact e.m.t.\ spaces such
that \(\varphi(\X)\) is dense in \((\Y,\tau_\Y)\), then we have in fact that \(\varphi(\X)=\Y\). Indeed, \(\varphi(\X)\) is compact (being the continuous image
of a compact space), thus accordingly it is closed in \((\Y,\tau_\Y)\) (since \(\tau_\Y\) is a Hausdorff topology).\\
\({\bf vi)}\) It easily follows from ii), iv) and v).
\end{proof}
\subsection{The e.m.t.-fication functor}\label{s:emt-fication}
In this section, we obtain a result that has a fundamental r\^{o}le in this paper: we provide a canonical way to associate to any given
pre-e.pm.t.\ space \((\X,\tau,\sfd)\) an e.m.t.\ space \((\check\X,\check\tau,\check\sfd)={\sf emt}((\X,\tau,\sfd))\), which we refer to
as the \emph{e.m.t.-fication} of \((\X,\tau,\sfd)\). 
\begin{proposition}\label{prop:emt-fication_constr}
Let \((\X,\tau,\sfd)\) be a pre-e.pm.t.\ space. Then there exist an e.m.t.\ space \((\check\X,\check\tau,\check\sfd)\)
and a continuous-short map \({\sf c}_{(\X,\tau,\sfd)}\colon(\X,\tau,\sfd)\to(\check\X,\check\tau,\check\sfd)\)
having the following universal property: given an e.m.t.\ space \((\Y,\tau_\Y,\sfd_\Y)\) and a continuous-short
map \(\varphi\colon(\X,\tau,\sfd)\to(\Y,\tau_\Y,\sfd_\Y)\), there exists a unique continuous-short map
\(\tilde\varphi\colon(\check\X,\check\tau,\check\sfd)\to(\Y,\tau_\Y,\sfd_\Y)\) such that
\[\begin{tikzcd}
(\X,\tau,\sfd) \arrow[r,"{\sf c}_{(\X,\tau,\sfd)}"] \arrow[dr,swap,"\varphi"] & (\check\X,\check\tau,\check\sfd) \arrow[d,"\tilde\varphi"] \\
& (\Y,\tau_\Y,\sfd_\Y)
\end{tikzcd}\]
is a commutative diagram. Moreover, we have that the map
\begin{equation}\label{eq:bij_CS_emt}
CS((\X,\tau,\sfd);(\Y,\tau_\Y,\sfd_\Y))\ni\varphi\mapsto\tilde\varphi\in CS((\check\X,\check\tau,\check\sfd);(\Y,\tau_\Y,\sfd_\Y))
\end{equation}
is a bijection.
\end{proposition}
\begin{proof}
We define the extended pseudodistance \(\tilde\sfd\) on \(\X\) as
\[
\tilde\sfd(x,y)\coloneqq\sup\big\{|f(x)-f(y)|\;\big|\;f\in\LIP_{b,1}(\X,\tau,\sfd)\big\}\quad\text{ for every }x,y\in\X.
\]
We endow the quotient space \(\check\X\coloneqq\X/\tilde\sfd\) with the extended distance \(\check\sfd\coloneqq[\tilde\sfd]\) induced by \(\tilde\sfd\)
and with the initial topology \(\check\tau\) of \(\LIP_{b,1}(\check\X,\tilde\tau,\check\sfd)\), where \(\tilde\tau\) denotes the
quotient topology of \(\tau\). Moreover, we denote by \({\sf c}={\sf c}_{(\X,\tau,\sfd)}\colon\X\to\check\X\)
the quotient map. Let us check that \((\check\X,\check\tau,\check\sfd)\) is an e.m.t.\ space. Since
\(\LIP_{b,1}(\check\X,\check\tau,\check\sfd)=\LIP_{b,1}(\check\X,\tilde\tau,\check\sfd)\), we have that
\(\check\tau\) is the initial topology of \(\LIP_{b,1}(\check\X,\check\tau,\check\sfd)\). Moreover, for any
\(f\in\LIP_{b,1}(\X,\tau,\tilde\sfd)\) there exists a unique function \([f]\colon\check\X\to\R\) such that
\[\begin{tikzcd}
\X \arrow[r,"f"] \arrow[d,swap,"{\sf c}"] & \R \\
\check\X \arrow[ur,swap,"{[}f{]}"] &
\end{tikzcd}\]
is a commutative diagram, and \([f]\in\LIP_{b,1}(\check\X,\check\tau,\check\sfd)\). Given that \(\LIP_{b,1}(\X,\tau,\sfd)\subseteq\LIP_{b,1}(\X,\tau,\tilde\sfd)\),
for any \(x,y\in\X\) we have
\begin{align*}
\check\sfd({\sf c}(x),{\sf c}(y))&=\tilde\sfd(x,y)=\sup\big\{|f(x)-f(y)|\;\big|\;f\in\LIP_{b,1}(\X,\tau,\sfd)\big\}\\
&=\sup\big\{|[f]({\sf c}(x))-[f]({\sf c}(y))|\;\big|\;f\in\LIP_{b,1}(\X,\tau,\sfd)\big\}\\
&\leq\sup\big\{|g({\sf c}(x))-g({\sf c}(y))|\;\big|\;g\in\LIP_{b,1}(\check\X,\check\tau,\check\sfd)\big\}
\leq\check\sfd({\sf c}(x),{\sf c}(y)),
\end{align*}
so that \(\check\sfd(w,z)=\sup\{|g(w)-g(z)|:g\in\LIP_{b,1}(\check\X,\check\tau,\check\sfd)\}\)
for all \(w,z\in\check\X\). All in all, \((\check\X,\check\tau,\check\sfd)\)
is an e.m.t.\ space. Note that \(\tilde\sfd\leq\sfd\) and \(\check\tau\subseteq\tilde\tau\), thus
\({\sf c}\colon(\X,\tau,\sfd)\to(\check\X,\check\tau,\check\sfd)\) is a continuous-short map.

Let us pass to the verification of the universal property. Fix an e.m.t.\ space \((\Y,\tau_\Y,\sfd_\Y)\)
and a continuous-short map \(\varphi\colon(\X,\tau,\sfd)\to(\Y,\tau_\Y,\sfd_\Y)\). Given any \(x,y\in\X\), we can estimate
\begin{equation}\label{eq:est_d_Y}\begin{split}
\sfd_\Y(\varphi(x),\varphi(y))&=\sup\big\{|g(\varphi(x))-g(\varphi(y))|\;\big|\;g\in\LIP_{b,1}(\Y,\tau_\Y,\sfd_\Y)\big\}\\
&\leq\sup\big\{|f(x)-f(y)|\;\big|\;f\in\LIP_{b,1}(\X,\tau,\sfd)\big\}=\tilde\sfd(x,y).
\end{split}\end{equation}
In particular, \(\sfd_\Y(\varphi(x),\varphi(y))=0\) whenever \(x,y\in\X\) satisfy \(\tilde\sfd(x,y)=0\). Recalling
that \(\check\X=\X/\tilde\sfd\), we deduce that there exists a unique map \(\tilde\varphi\colon\check\X\to\Y\) such that
\[\begin{tikzcd}
\X \arrow[r,"{\sf c}"] \arrow[dr,swap,"\varphi"] & \check\X \arrow[d,"\tilde\varphi"] \\
& \Y
\end{tikzcd}\]
is a commutative diagram. We claim that \(\tilde\varphi\colon(\check\X,\check\tau)\to(\Y,\tau_\Y)\) is continuous.
Since \(\tau_\Y\) is the initial topology of \(\LIP_{b,1}(\Y,\tau_\Y,\sfd_\Y)\), the claim is equivalent to the
continuity of \(f\circ\tilde\varphi\colon(\check\X,\check\tau)\to\R\) for every \(f\in\LIP_{b,1}(\Y,\tau_\Y,\sfd_\Y)\).
For any such function \(f\), it holds that \(f\circ\varphi\in\LIP_{b,1}(\X,\tau,\sfd)\) and thus
\(f\circ\tilde\varphi=[f\circ\varphi]\in\LIP_{b,1}(\check\X,\check\tau,\check\sfd)\). In particular,
\(f\circ\tilde\varphi\) is \(\check\tau\)-continuous, thus the claim is proved. Next, by applying \eqref{eq:est_d_Y}
we obtain that
\begin{align*}
\sfd_\Y\big(\tilde\varphi({\sf c}(y)),\tilde\varphi({\sf c}(x))\big)&=\sfd_\Y(\varphi(x),\varphi(y))
\leq\tilde\sfd(x,y)=\check\sfd({\sf c}(x),{\sf c}(y))\quad\text{ for every }x,y\in\X,
\end{align*}
which implies that \(\tilde\varphi\colon(\check\X,\check\sfd)\to(\Y,\sfd_\Y)\) is a short map thanks to the surjectivity
of \({\sf c}\). All in all, \(\tilde\varphi\colon(\check\X,\check\tau,\check\sfd)\to(\Y,\tau_\Y,\sfd_\Y)\)
is a continuous-short map. Therefore, the universal property is proved.

Finally, let us check that \(CS((\X,\tau,\sfd);(\Y,\tau_\Y,\sfd_\Y))\ni\varphi\mapsto\tilde\varphi\in CS((\check\X,\check\tau,\check\sfd);(\Y,\tau_\Y,\sfd_\Y))\)
is a bijection. To prove injectivity, just note that if \(\varphi_1,\varphi_2\in CS((\X,\tau,\sfd);(\Y,\tau_\Y,\sfd_\Y))\) satisfy \(\tilde\varphi_1=\tilde\varphi_2\),
then \(\varphi_1(x)=\tilde\varphi_1({\sf c}(x))=\tilde\varphi_2({\sf c}(x))=\varphi_2(x)\) for every \(x\in\X\), so that \(\varphi_1=\varphi_2\). Moreover, given any
function \(\psi\in CS((\check\X,\check\tau,\check\sfd);(\Y,\tau_\Y,\sfd_\Y))\), we have that \(\varphi\coloneqq\psi\circ{\sf c}\in CS((\X,\tau,\sfd);(\Y,\tau_\Y,\sfd_\Y))\)
satisfies \(\tilde\varphi=\psi\) thanks to the universal property, thus also surjectivity is proved.
\end{proof}

As the proof of Proposition \ref{prop:emt-fication_constr} shows, it holds that
\begin{equation}\label{eq:c_surj}
{\sf c}_{(\X,\tau,\sfd)}\colon\X\to\check\X\quad\text{ is surjective.}
\end{equation}
\begin{remark}\label{rmk:emt_recover}{\rm
The proof of Proposition \ref{prop:emt-fication_constr} also shows that \((\check\X,\check\tau,\check\sfd)\) possesses the following feature:
\emph{if \(\sfd\) can be \(\tau\)-recovered, then \(\tilde\sfd=\sfd\) and \(\check\sfd({\sf c}(x),{\sf c}(y))=\sfd(x,y)\) for every \(x,y\in\X\).}
\fr}\end{remark}
\begin{remark}\label{rmk:emt_cpt}{\rm
Another byproduct of the proof of Proposition \ref{prop:emt-fication_constr} is the following: \emph{if \((\X,\tau,\sfd)\) is a pre-e.pm.t.\ space with
\((\X,\tau)\) compact, then \((\check\X,\check\tau)\) is compact and a quotient of \((\X,\tau)\).} Indeed, \(\tilde\tau\) is compact (as a quotient of \(\tau\)) and
the proof of Proposition \ref{prop:emt-fication_constr} shows that \(\LIP_{b,1}(\check\X,\tilde\tau,\check\sfd)\) separates the points of \(\check\X\), so that Remark
\ref{rmk:criterion_initial_cpt} ensures that \(\check\tau=\tilde\tau\), thus accordingly \((\check\X,\check\tau)\) is compact.
\fr}\end{remark}
\begin{lemma}\label{lem:emt-fication_morphisms}
Let \((\X,\tau_\X,\sfd_\X)\) and \((\Y,\tau_\Y,\sfd_\Y)\) be pre-e.pm.t.\ spaces. Let \(\varphi\colon(\X,\tau_\X,\sfd_\X)\to(\Y,\tau_\Y,\sfd_\Y)\) be a continuous-short map.
Let \((\check\X,\check\tau_\X,\check\sfd_\X)\), \((\check\Y,\check\tau_\Y,\check\sfd_\Y)\), \({\sf c}_{(\X,\tau_\X,\sfd_\X)}\) and \({\sf c}_{(\Y,\tau_\Y,\sfd_\Y)}\) be
given by Proposition \ref{prop:emt-fication_constr}. Then there exists a unique continuous-short map
\(\check\varphi\colon(\check\X,\check\tau_\X,\check\sfd_\X)\to(\check\Y,\check\tau_\Y,\check\sfd_\Y)\) such that
\begin{equation}\label{eq:diagram_check_phi}\begin{tikzcd}
(\X,\tau_\X,\sfd_\X) \arrow[d,swap,"{\sf c}_{(\X,\tau_\X,\sfd_\X)}"] \arrow[r,"\varphi"] & (\Y,\tau_\Y,\sfd_\Y) \arrow[d,"{\sf c}_{(\Y,\tau_\Y,\sfd_\Y)}"] \\
(\check\X,\check\tau_\X,\check\sfd_\X) \arrow[r,swap,"\check\varphi"] & (\check\Y,\check\tau_\Y,\check\sfd_\Y)
\end{tikzcd}\end{equation}
is a commutative diagram.
\end{lemma}
\begin{proof}
Note that \({\sf c}_{(\Y,\tau_\Y,\sfd_\Y)}\circ\varphi\colon(\X,\tau_\X,\sfd_\X)\to(\check\Y,\check\tau_\Y,\check\sfd_\Y)\) is a continuous-short map.
Hence, Proposition \ref{prop:emt-fication_constr} ensures that
\(\check\varphi\coloneqq({\sf c}_{(\Y,\tau_\Y,\sfd_\Y)}\circ\varphi)^\sim\colon(\check\X,\check\tau_\X,\check\sfd_\X)\to(\check\Y,\check\tau_\Y,\check\sfd_\Y)\)
is the unique continuous-short map such that the diagram in \eqref{eq:diagram_check_phi} commutes.
\end{proof}
\begin{definition}[The e.m.t.-fication functor]\label{def:e.m.t.-fication_functor}
We define the \emph{e.m.t.-fication functor}
\[
{\sf emt}\colon{\bf PreExt\Psi MetTop}\to{\bf ExtMetTop}
\]
as follows: to any object \((\X,\tau,\sfd)\in{\rm Obj}({\bf PreExt\Psi MetTop})\) we assign the object
\[
{\sf emt}((\X,\tau,\sfd))\coloneqq(\check\X,\check\tau,\check\sfd)\in{\rm Obj}({\bf ExtMetTop})
\]
given by Proposition \ref{prop:emt-fication_constr}, and to any morphism \(\varphi\colon(\X,\tau_\X,\sfd_\X)\to(\Y,\tau_\Y,\sfd_\Y)\)
in \({\bf PreExt\Psi MetTop}\) we assign the morphism \({\sf emt}(\varphi)\coloneqq\check\varphi\colon{\sf emt}((\X,\tau_\X,\sfd_\X))\to{\sf emt}((\Y,\tau_\Y,\sfd_\Y))\)
in \({\bf ExtMetTop}\) given by Lemma \ref{lem:emt-fication_morphisms}. 
\end{definition}
\begin{theorem}\label{thm:emt-fication}
The e.m.t.-fication functor \({\sf emt}\) is left adjoint to the inclusion functor
\[
\iota\colon{\bf ExtMetTop}\hookrightarrow{\bf PreExt\Psi MetTop}.
\]
In particular, \({\bf ExtMetTop}\) is a reflective subcategory of \({\bf PreExt\Psi MetTop}\), with reflector \({\sf emt}\).
Moreover, the assignment
\[
{\rm Obj}({\bf PreExt\Psi MetTop})\ni(\X,\tau,\sfd)\mapsto\iota\circ{\sf c}_{(\X,\tau,\sfd)}\colon(\X,\tau,\sfd)\to(\iota\circ{\sf emt})((\X,\tau,\sfd))
\]
given by Proposition \ref{prop:emt-fication_constr} yields a natural transformation from the identity functor \({\rm id}_{{\bf PreExt\Psi MetTop}}\)
of \({\bf PreExt\Psi MetTop}\) to the functor \(\iota\circ{\sf emt}\).
\end{theorem}
\begin{proof}
Fix an arbitrary object \((\X,\tau,\sfd)\) of \({\bf ExtMetTop}\) and define
\[
\varepsilon_{(\X,\tau,\sfd)}\coloneqq{\rm id}_{(\X,\tau,\sfd)}\colon({\sf emt}\circ\iota)((\X,\tau,\sfd))=(\X,\tau,\sfd)\to(\X,\tau,\sfd).
\]
Given an object \((\Y,\tau_\Y,\sfd_\Y)\) of \({\bf PreExt\Psi MetTop}\) and a morphism
\(\varphi\colon{\sf emt}((\Y,\tau_\Y,\sfd_\Y))\to(\X,\tau,\sfd)\) in \({\bf ExtMetTop}\),
it follows from the last part of the statement of Proposition \ref{prop:emt-fication_constr} that there exists
a unique morphism \(\psi\colon(\Y,\tau_\Y,\sfd_\Y)\to\iota((\X,\tau,\sfd))=(\X,\tau,\sfd)\) in \({\bf PreExt\Psi MetTop}\)
such that
\[\begin{tikzcd}[column sep = huge]
(\Y,\tau_\Y,\sfd_\Y) \arrow[r,"{\sf c}_{(\Y,\tau_\Y,\sfd_\Y)}"] \arrow[rd,swap,"\psi"] & {\sf emt}((\Y,\tau_\Y,\sfd_\Y)) \arrow[d,swap,"{\sf emt}(\psi)"] \arrow[dr,"\varphi"] & \\
& ({\sf emt}\circ\iota)((\X,\tau,\sfd)) \arrow[r,swap,"\varepsilon_{(\X,\tau,\sfd)}"] & (\X,\tau,\sfd)
\end{tikzcd}\]
is a commutative diagram. This proves that \({\sf emt}\dashv\iota\). Moreover, to show that the assignment \((\X,\tau,\sfd)\mapsto\iota\circ{\sf c}_{(\X,\tau,\sfd)}\)
is a natural transformation from \({\rm id}_{{\bf PreExt\Psi MetTop}}\) to \(\iota\circ{\sf emt}\), note that
\[\begin{tikzcd}[column sep = huge]
(\X,\tau_\X,\sfd_\X) \arrow[r,"{\sf c}_{(\X,\tau_\X,\sfd_\X)}"] \arrow[d,swap,"\varphi"] & {\sf emt}((\X,\tau_\X,\sfd_\X)) \arrow[r,hook,"\iota"] \arrow[d,swap,"{\sf emt}(\varphi)"] &
(\iota\circ{\sf emt})((\X,\tau_\X,\sfd_\X)) \arrow[d,"(\iota\circ{\sf emt})(\varphi)"] \\
(\Y,\tau_\Y,\sfd_\Y) \arrow[r,swap,"{\sf c}_{(\Y,\tau_\Y,\sfd_\Y)}"] & {\sf emt}((\Y,\tau_\Y,\sfd_\Y)) \arrow[r,hook,swap,"\iota"] & (\iota\circ{\sf emt})((\Y,\tau_\Y,\sfd_\Y))
\end{tikzcd}\]
is a commutative diagram for every morphism \(\varphi\colon(\X,\tau_\X,\sfd_\X)\to(\Y,\tau_\Y,\sfd_\Y)\) in \({\bf PreExt\Psi MetTop}\),
thanks to Lemma \ref{lem:emt-fication_morphisms}. Therefore, the statement is achieved.
\end{proof}
\subsection{Bicompleteness of \texorpdfstring{\({\bf ExtMetTop}\)}{ExtMetTop}}
We now prove the bicompleteness of \({\bf PreExt\Psi MetTop}\) and \({\bf ExtMetTop}\). The first step is to show the bicompleteness
of \({\bf PreExt\Psi MetTop}\), which -- heuristically speaking -- boils down to the fact that each (co)limit either in \({\bf Top}\) or in
\({\bf Ext\Psi Met}\) agrees (as a set) with the corresponding (co)limit in \({\bf Set}\), as it is evident from Section \ref{s:lim_Set_Top_ExtPMet}.
\begin{theorem}\label{thm:PreExtPMetTop_bicompl}
The category \({\bf PreExt\Psi MetTop}\) is bicomplete.
\end{theorem}
\begin{proof}
We recall from Remark \ref{rmk:equiv_PreExtPMetTop} that \({\bf PreExt\Psi MetTop}={\bf SP}(U,V)\),
where \(U\colon{\bf Top}\to{\bf Set}\) and \(V\colon{\bf Ext\Psi Met}\to{\bf Set}\) denote the forgetful functors.
We know from Section \ref{s:lim_Set_Top_ExtPMet} that \({\bf Top}\) and \({\bf Ext\Psi Met}\) are bicomplete categories.
Since the results stated in Section \ref{s:lim_Set_Top_ExtPMet} show also that every product (resp.\ equaliser) in
\({\bf Top}\) or \({\bf Ext\Psi Met}\) is just the product (resp.\ equaliser) in \({\bf Set}\) equipped with a
topology or an extended pseudodistance, and likewise for coproducts and coequalisers, it readily follows from
Lemma \ref{lem:limits_SP} that \({\bf SP}(U,V)\) admits all (co)products and (co)equalisers. Applying Theorem
\ref{thm:exist_thm_lim}, we thus conclude that \({\bf PreExt\Psi MetTop}={\bf SP}(U,V)\) is bicomplete.
\end{proof}
\begin{remark}\label{rmk:descript_lim_Pre}{\rm
Since in Section \ref{s:lim_Set_Top_ExtPMet} we have given explicit descriptions of all (co)products and
(co)equalisers in \({\bf Top}\) and in \({\bf Ext\Psi Met}\), by applying Lemma \ref{lem:limits_SP} one
can describe all (co)products and (co)equalisers in \({\bf PreExt\Psi MetTop}\). However, for brevity's
sake we avoid doing it.
\fr}\end{remark}

It is now immediate to obtain the main result of this paper:
\begin{theorem}\label{thm:ExtMetTop_bicompl}
The category \({\bf ExtMetTop}\) is bicomplete.
\end{theorem}
\begin{proof}
Since \({\bf ExtMetTop}\) is a reflective subcategory of \({\bf PreExt\Psi MetTop}\) by Theorem \ref{thm:emt-fication},
and \({\bf PreExt\Psi MetTop}\) is a bicomplete category by Theorem \ref{thm:PreExtPMetTop_bicompl}, we deduce from
Corollary \ref{cor:compl_preserved} that also \({\bf ExtMetTop}\) is bicomplete.
\end{proof}

Applying Proposition \ref{prop:create_lim} and taking Remark \ref{rmk:descript_lim_Pre} into account,
it is actually possible to give an explicit description of all (co)products and (co)equalisers in \({\bf ExtMetTop}\),
but we do not do it.
\subsection{The compactification functor}\label{s:compactification}
In this section, we prove the existence of a compactification functor for e.m.t.\ spaces, akin to the
Stone--\v{C}ech compactification functor for topological spaces.
\begin{theorem}[Compactification of an e.m.t.\ space]\label{thm:compactif_emt}
Let \((\X,\tau,\sfd)\) be an e.m.t.\ space. Then there exist a compact e.m.t.\ space \((\gamma\X,\gamma\tau,\gamma\sfd)\)
and a continuous-short map
\[
\iota=\iota_{(\X,\tau,\sfd)}\colon(\X,\tau,\sfd)\to(\gamma\X,\gamma\tau,\gamma\sfd)
\]
such that the following universal property is satisfied: given a compact e.m.t.\ space \((\Y,\tau_\Y,\sfd_\Y)\) and a continuous-short map
\(\varphi\colon(\X,\tau,\sfd)\to(\Y,\tau_\Y,\sfd_\Y)\), there exists a unique continuous-short map
\(\varphi_\gamma\colon(\gamma\X,\gamma\tau,\gamma\sfd)\to(\Y,\tau_\Y,\sfd_\Y)\) such that
\begin{equation}\label{eq:diagram_cpt_emt}\begin{tikzcd}
(\X,\tau,\sfd) \arrow[r,"\iota"] \arrow[dr,swap,"\varphi"] & (\gamma\X,\gamma\tau,\gamma\sfd) \arrow[d,"\varphi_\gamma"] \\
& (\Y,\tau_\Y,\sfd_\Y)
\end{tikzcd}\end{equation}
is a commutative diagram. Moreover, the following properties hold:
\begin{itemize}
\item[\(\rm i)\)] The image \(\iota(\X)\) of \(\iota\) is dense in \((\gamma\X,\gamma\tau)\).
\item[\(\rm ii)\)] \(\iota\colon(\X,\tau,\sfd)\to(\gamma\X,\gamma\tau,\gamma\sfd)\) is an embedding of e.m.t.\ spaces.
\item[\(\rm iii)\)] The map \(CS((\X,\tau,\sfd);(\Y,\tau_\Y,\sfd_\Y))\ni\varphi\mapsto\varphi_\gamma\in CS((\gamma\X,\gamma\tau,\gamma\sfd);(\Y,\tau_\Y,\sfd_\Y))\)
is bijective.
\end{itemize}
\end{theorem}
\begin{proof}
Let \((\beta\X,\beta\tau)\) be the Stone--\v{C}ech compactification of \((\X,\tau)\), with embedding \(i\colon\X\hookrightarrow\beta\X\).
Any \(f\in C_b(\X,\tau)\) induces a function \(f_\beta\in C(\beta\X,\beta\tau)\), since the image of \(f\) is contained
in a compact subset of \(\R\). We now define the extended pseudodistance \(\beta\sfd\) on \(\beta\X\) as
\[
(\beta\sfd)(x,y)\coloneq\sup_{f\in\LIP_{b,1}(\X,\tau,\sfd)}|f_\beta(x)-f_\beta(y)|\quad\text{ for every }x,y\in\beta\X.
\]
As \(f_\beta\in\LIP_{b,1}(\beta\X,\beta\tau,\beta\sfd)\) for all \(f\in\LIP_{b,1}(\X,\tau,\sfd)\), it follows that \(\beta\sfd\)
can be \(\beta\tau\)-recovered.
Note that \(i\colon(\X,\tau,\sfd)\to(\beta\X,\beta\tau,\beta\sfd)\) is a continuous-short map, since it is continuous and we have
\begin{equation}\label{eq:i_isom}\begin{split}
(\beta\sfd)(i(x),i(y))&=\sup_{f\in\LIP_{b,1}(\X,\tau,\sfd)}\big|f_\beta(i(x))-f_\beta(i(y))\big|\\
&=\sup_{f\in\LIP_{b,1}(\X,\tau,\sfd)}|f(x)-f(y)|=\sfd(x,y)\quad\text{ for every }x,y\in\X.
\end{split}\end{equation}
We apply the e.m.t.-fication functor to the pre-e.pm.t.\ space \((\beta\X,\beta\tau,\beta\sfd)\): define \({\sf c}\coloneqq{\sf c}_{(\beta\X,\beta\tau,\beta\sfd)}\),
\[
(\gamma\X,\gamma\tau,\gamma\sfd)\coloneqq{\sf emt}((\beta\X,\beta\tau,\beta\sfd)),\qquad
\iota\coloneqq{\sf c}\circ i\colon(\X,\tau,\sfd)\to(\gamma\X,\gamma\tau,\gamma\sfd).
\]
Thanks to Remark \ref{rmk:emt_cpt}, \((\gamma\X,\gamma\tau,\gamma\sfd)\) is a compact e.m.t.\ space. Moreover, \(\iota\) is a
continuous-short map (as a composition of continuous-short maps). Next, let us check the universal property. Fix a compact e.m.t.\ space
\((\Y,\tau_\Y,\sfd_\Y)\) and a continuous-short map \(\varphi\colon(\X,\tau,\sfd)\to(\Y,\tau_\Y,\sfd_\Y)\). Given any \(g\in\LIP_{b,1}(\Y,\tau_\Y,\sfd_\Y)\),
we have \(g\circ\varphi\in\LIP_{b,1}(\X,\tau,\sfd)\), \(g\circ\varphi_\beta\in C(\beta\X,\beta\tau)\) and \(g\circ\varphi_\beta\circ i=g\circ\varphi\),
thus \(g\circ\varphi_\beta=(g\circ\varphi)_\beta\in\LIP_{b,1}(\beta\X,\beta\tau,\beta\sfd)\). Since \(\sfd_\Y\) can be \(\tau_\Y\)-recovered, we deduce that
\[
\sfd_\Y(\varphi_\beta(x),\varphi_\beta(y))=\sup_{g\in\LIP_{b,1}(\Y,\tau_\Y,\sfd_\Y)}|g(\varphi_\beta(x))-g(\varphi_\beta(y))|
\leq(\beta\sfd)(x,y)\quad\text{ for every }x,y\in\beta\X,
\]
so that \(\varphi_\beta\colon(\beta\X,\beta\tau,\beta\sfd)\to(\Y,\tau_\Y,\sfd_\Y)\) is a continuous-short map. Hence, by Proposition \ref{prop:emt-fication_constr}
we know that there exists a unique continuous-short map \(\varphi_\gamma\colon(\gamma\X,\gamma\tau,\gamma\sfd)\to(\Y,\tau_\Y,\sfd_\Y)\) such that
\[\begin{tikzcd}
(\X,\tau,\sfd) \arrow[r,hook,"i"] \arrow[dr,swap,"\varphi"] & (\beta\X,\beta\tau,\beta\sfd) \arrow[r,"{\sf c}"]
\arrow[d,"\varphi_\beta"] & (\gamma\X,\gamma\tau,\gamma\sfd) \arrow[dl,"\varphi_\gamma"] \\
& (\Y,\tau_\Y,\sfd_\Y) &
\end{tikzcd}\]
is a commutative diagram. Using the fact that the map in \eqref{eq:bij_CS_emt} is bijective, we deduce that \(\varphi_\gamma\) is the unique
continuous-short map for which \eqref{eq:diagram_cpt_emt} is a commutative diagram, thus proving the universal property. To conclude the proof,
let us check that i), ii) and iii) hold. Given that \(i(\X)\) is dense in \((\beta\X,\beta\tau)\) and
\({\sf c}\colon(\beta\X,\beta\tau)\to(\gamma\X,\gamma\tau)\) is a continuous surjective map, we deduce that
\(\iota(\X)={\sf c}(i(\X))\) is dense in \((\gamma\X,\gamma\tau)\), proving i). Moreover, we have already observed
that \(\beta\sfd\) can be \(\beta\tau\)-recovered, thus by applying Remark \ref{rmk:emt_recover} and \eqref{eq:i_isom}
we deduce that \(\iota\) is a distance-preserving map, thus in particular it is injective.
Since \(\iota\) is also continuous, to show that \(\iota\colon\X\to\iota(\X)\) is a homeomorphism it is
sufficient to check that it is an open map. We know that \(\{{\rm Coz}(f):f\in\LIP_{b,1}(\X,\tau,\sfd)\}\)
is a basis for the topology \(\tau\), where the cozero set of the function \(f\) is defined as
\({\rm Coz}(f)\coloneq\{f\neq 0\}\). Given that \(\iota({\rm Coz}(f))={\rm Coz}(f_\gamma)\cap\iota(\X)\) and
\(f_\gamma\in C(\gamma\X,\gamma\tau)\), we conclude that \(\iota({\rm Coz}(f))\) is open in \(\iota(\X)\),
whence it follows that \(\iota\colon\X\to\iota(\X)\) is an open map and thus ii) is proved. Finally,
it is clear that the map \(\varphi\mapsto\varphi_\gamma\)
is both injective (as \(\varphi_\gamma=\psi_\gamma\) trivially implies that \(\varphi=\varphi_\gamma\circ\iota=\psi_\gamma\circ\iota=\psi\))
and surjective (as a consequence of ii)), thus accordingly also iii) is proved.
\end{proof}
\begin{lemma}\label{lem:cpt_morph}
Let \((\X,\tau_\X,\sfd_\X)\), \((\Y,\tau_\Y,\sfd_\Y)\) be e.m.t.\ spaces. Let
\(\varphi\colon(\X,\tau_\X,\sfd_\X)\to(\Y,\tau_\Y,\sfd_\Y)\) be a continuous-short map.
Let \((\gamma\X,\gamma\tau_\X,\gamma\sfd_\X)\), \((\gamma\Y,\gamma\tau_\Y,\gamma\sfd_\Y)\),
\(\iota_\X\colon\X\to\gamma\X\) and \(\iota_\Y\colon\Y\to\gamma\Y\) be given by Theorem \ref{thm:compactif_emt}.
Then there exists a unique continuous-short map
\[
\gamma\varphi\colon(\gamma\X,\gamma\tau_\X,\gamma\sfd_\X)\to(\gamma\Y,\gamma\tau_\Y,\gamma\sfd_\Y)
\]
such that the following diagram commutes:
\begin{equation}\label{eq:diagram_cpt}\begin{tikzcd}
(\X,\tau_\X,\sfd_\X) \arrow[r,"\iota_\X"] \arrow[d,swap,"\varphi"] & (\gamma\X,\gamma\tau_\X,\gamma\sfd_\X) \arrow[d,"\gamma\varphi"] \\
(\Y,\tau_\Y,\sfd_\Y) \arrow[r,swap,"\iota_\Y"] & (\gamma\Y,\gamma\tau_\Y,\gamma\sfd_\Y)
\end{tikzcd}\end{equation}
\end{lemma}
\begin{proof}
Note that \(\iota_\Y\circ\varphi\colon(\X,\tau_\X,\sfd_\X)\to(\gamma\Y,\gamma\tau_\Y,\gamma\sfd_\Y)\)
is a continuous-short map. Hence, Theorem \ref{thm:compactif_emt} ensures that
\(\gamma\varphi\coloneqq(\iota_\Y\circ\varphi)_\gamma\colon(\gamma\X,\gamma\tau_\X,\gamma\sfd_\X)\to(\gamma\Y,\gamma\tau_\Y,\gamma\sfd_\Y)\)
is the unique continuous-short map such that the diagram in \eqref{eq:diagram_cpt} commutes.
\end{proof}
\begin{definition}[The compactification functor]
We define the \emph{compactification functor}
\[
\gamma\colon{\bf ExtMetTop}\to{\bf CptExtMetTop}
\]
as follows: to any object \((\X,\tau,\sfd)\in{\rm Obj}({\bf ExtMetTop})\) we assign the object
\[
\gamma((\X,\tau,\sfd))\coloneqq(\gamma\X,\gamma\tau,\gamma\sfd)\in{\rm Obj}({\bf CptExtMetTop})
\]
given by Theorem \ref{thm:compactif_emt}, and to any morphism \(\varphi\colon(\X,\tau_\X,\sfd_\X)\to(\Y,\tau_\Y,\sfd_\Y)\)
in \({\bf ExtMetTop}\) we assign the morphism
\(\gamma(\varphi)\coloneqq\gamma\varphi\colon(\gamma\X,\gamma\tau_\X,\gamma\sfd_\X)\to(\gamma\Y,\gamma\tau_\Y,\gamma\sfd_\Y)\)
in \({\bf CptExtMetTop}\) given by Lemma \ref{lem:cpt_morph}.
\end{definition}
\begin{theorem}
The compactification functor \(\gamma\) is left adjoint to the inclusion functor
\[
\iota\colon{\bf CptExtMetTop}\hookrightarrow{\bf ExtMetTop}.
\]
In particular, \({\bf CptExtMetTop}\) is a reflective subcategory of \({\bf ExtMetTop}\), with reflector \(\gamma\).
Moreover, the assignment
\[
{\rm Obj}({\bf ExtMetTop})\ni(\X,\tau,\sfd)\mapsto\iota\circ\iota_{(\X,\tau,\sfd)}\colon(\X,\tau,\sfd)\to(\iota\circ\gamma)((\X,\tau,\sfd))
\]
yields a natural transformation from the identity functor \({\rm id}_{{\bf ExtMetTop}}\) to the functor \(\iota\circ\gamma\).
\end{theorem}
\begin{proof}
It follows along the lines of Theorem \ref{thm:emt-fication}. To show that \(\gamma\dashv\iota\), one uses
Theorem \ref{thm:compactif_emt} iii). The fact that \((\X,\tau,\sfd)\mapsto\iota\circ\iota_{(\X,\tau,\sfd)}\) is a natural transformation
follows from Lemma \ref{lem:cpt_morph}.
\end{proof}
\begin{definition}[The `generalised' compactification functor]
We define the functor \(\bar\gamma\) as
\[
\bar\gamma\coloneqq\gamma\circ{\sf emt}\colon{\bf PreExt\Psi MetTop}\to{\bf CptExtMetTop}.
\]
Moreover, given any pre-e.pm.t.\ space \((\X,\tau,\sfd)\), we define the continuous-short map
\(\bar\iota_{(\X,\tau,\sfd)}\) as
\[
\bar\iota_{(\X,\tau,\sfd)}\coloneqq\iota_{{\sf emt}((\X,\tau,\sfd))}\circ{\sf c}_{(\X,\tau,\sfd)}
\colon(\X,\tau,\sfd)\to\bar\gamma((\X,\tau,\sfd)).
\]
\end{definition}

It follows from Proposition \ref{prop:emt-fication_constr} and Theorem \ref{thm:compactif_emt} that the ensuing universal property holds:
given a compact e.m.t.\ space \((\Y,\tau_\Y,\sfd_\Y)\) and a continuous-short map \(\varphi\colon(\X,\tau,\sfd)\to(\Y,\tau_\Y,\sfd_\Y)\),
there exists a unique continuous-short map \(\varphi_{\bar\gamma}\colon\bar\gamma((\X,\tau,\sfd))\to(\Y,\tau_\Y,\sfd_\Y)\) such that
\[\begin{tikzcd}[column sep=huge]
(\X,\tau,\sfd) \arrow[r,"\bar\iota_{(\X,\tau,\sfd)}"] \arrow[dr,swap,"\varphi"] & \bar\gamma((\X,\tau,\sfd)) \arrow[d,"\varphi_{\bar\gamma}"] \\
& (\Y,\tau_\Y,\sfd_\Y)
\end{tikzcd}\]
is a commutative diagram.
\medskip

We point out that the proof of Theorem \ref{thm:compactif_emt} shows that the topological space \((\gamma\X,\gamma\tau)\)
is a quotient of the Stone--\v{C}ech compactification \((\beta\X,\beta\tau)\). More precisely, we have that
\(\gamma\X=\beta\X/\sim\), where the equivalence relation \(\sim\) is the following: given any \(x,y\in\beta\X\), we declare that
\begin{equation}\label{eq:gammaX_as_quotient_of_betaX}
x\sim y\quad\Longleftrightarrow\quad(\beta f)(x)=(\beta f)(y)\text{ for every }f\in\LIP_{b,1}(\X,\tau,\sfd).
\end{equation}
Consequently, a natural question arises: are \((\gamma\X,\gamma\tau)\) and \((\beta\X,\beta\tau)\) homeomorphic?
The next example shows that the answer is (in general) negative.
\begin{example}[where \(\beta\X\) and \(\gamma\X\) are not homeomorphic]\label{ex:betaX_vs_gammaX}{\rm
Let us consider the e.m.t.\ space \((\N,\tau,\sfd)\), where \(\tau\) is the discrete topology on \(\N\) and \(\sfd(n,m)\coloneqq\big|\frac{1}{n}-\frac{1}{m}\big|\)
for every \(n,m\in\N\). Note that \((\N,\sfd)\) is in fact a metric space and \(\tau\) is the topology induced by \(\sfd\). We claim that
\begin{equation}\label{eq:betaX_vs_gammaX_claim}
(\beta f)(\omega)=(\beta f)(\eta)\quad\text{ for every }\omega,\eta\in\beta\N\setminus i(\N)\text{ and }f\in\LIP_{b,1}(\N,\tau,\sfd).
\end{equation}
To prove it, we consider the representation of \(\beta\N\) as the space of ultrafilters on \(\N\) and we recall that
\((\beta f)(\omega)=\omega\text{-}\lim_n f(n)\). Since \((n)_n\) is a \(\sfd\)-Cauchy sequence, we have that \((f(n))_n\subseteq\R\)
is Cauchy whenever \(f\in\LIP_{b,1}(\N,\tau,\sfd)\), so that \(\lim_n f(n)\in\R\) exists and thus \(\omega\text{-}\lim_n f(n)=\lim_n f(n)\)
for every non-principal ultrafilter \(\omega\) on \(\N\). This proves the validity of the claim \eqref{eq:betaX_vs_gammaX_claim}.
Taking \eqref{eq:gammaX_as_quotient_of_betaX} into account, it is then easy to check that \((\gamma\X,\gamma\tau)\) is (homeomorphic to)
the one-point compactification of \((\N,\tau)\), thus in particular \(\gamma\N\) consists of countably many points. On the other hand,
it is well known that \(\beta\N\) is uncountable, thus \((\gamma\N,\gamma\tau)\) and \((\beta\N,\beta\tau)\) cannot be homeomorphic.
\fr}\end{example}

Next, we recall the definition of the \emph{Gelfand compactification} \((\hat\X,\hat\tau,\hat\sfd)\) of an e.m.t.\ space \((\X,\tau,\sfd)\)
introduced by Savar\'{e} in \cite[Section 2.1.7]{Sav:22}. Furthermore, we prove (cf.\ Proposition \ref{prop:consist_cpt}) that
\((\hat\X,\hat\tau,\hat\sfd)\) satisfies the universal property stated in Proposition \ref{prop:emt-fication_constr}, and thus accordingly
\((\hat\X,\hat\tau,\hat\sfd)\) coincides with \((\gamma\X,\gamma\tau,\gamma\sfd)\), up to a unique isomorphism. We remind the following definitions:
\begin{itemize}
\item \(\hat\X\) is the set of all \emph{characters} of the space \(\LIP_b(\X,\tau,\sfd)\) of bounded real-valued \(\tau\)-continuous \(\sfd\)-Lipschitz
functions, i.e.\ the set of all non-zero elements \(\theta\) of the dual of the normed space \((\LIP_b(\X,\tau,\sfd),\|\cdot\|_{C_b(\X,\tau)})\) satisfying
\(\theta(fg)=\theta(f)\theta(g)\) for every \(f,g\in\LIP_b(\X,\tau,\sfd)\).
\item The topology \(\hat\tau\) on \(\hat\X\) is defined as the restriction of the weak\(^*\) topology.
\item The extended distance \(\hat\sfd\) on \(\hat\X\) is defined as
\[
\hat\sfd(\theta,\sigma)\coloneqq\sup\big\{|\theta(f)-\sigma(f)|\;\big|\;f\in\LIP_{b,1}(\X,\tau,\sfd)\big\}\quad\text{ for every }\theta,\sigma\in\hat\X.
\]
\item The embedding map \(\hat\iota\colon\X\hookrightarrow\hat\X\) is given by
\[
\hat\iota(x)(f)\coloneqq f(x)\quad\text{ for every }x\in\X\text{ and }f\in\LIP_b(\X,\tau,\sfd).
\]
\end{itemize}
As it was proved in \cite[Theorem 2.1.34]{Sav:22}, the triple \((\hat\X,\hat\tau,\hat\sfd)\) is a compact e.m.t.\ space, \(\iota\colon\X\hookrightarrow\hat\X\)
is an embedding of e.m.t.\ spaces, and the image \(\iota(\X)\) is weakly\(^*\) dense in \(\hat\X\).
\begin{proposition}[Consistency with {\cite[Theorem 2.1.34]{Sav:22}}]\label{prop:consist_cpt}
Let \((\X,\tau,\sfd)\) be an e.m.t.\ space. Then there exists a unique isomorphism
of e.m.t.\ spaces \(\Phi\colon(\gamma\X,\gamma\tau,\gamma\sfd)\to(\hat\X,\hat\tau,\hat\sfd)\) such that
\[\begin{tikzcd}
\X  \arrow[r,hook,"\iota"] \arrow[dr,hook,swap,"\hat\iota"] & \gamma\X \arrow[d,"\Phi"] \\
& \hat\X
\end{tikzcd}\]
is a commutative diagram.
\end{proposition}
\begin{proof}
It suffices to check that \(((\hat\X,\hat\tau,\hat\sfd),\hat\iota)\) verifies the universal property stated in Theorem \ref{prop:emt-fication_constr}.
Fix a compact e.m.t.\ space \((\Y,\tau_\Y,\sfd_\Y)\) and a continuous-short map \(\varphi\colon(\X,\tau,\sfd)\to(\Y,\tau_\Y,\sfd_\Y)\).
Note that \(g\circ\varphi\in\LIP_{b,1}(\X,\tau,\sfd)\) for every \(g\in\LIP_{b,1}(\Y,\tau_\Y,\sfd_\Y)\). We claim that there exists a unique map \(\hat\varphi\colon\hat\X\to\Y\) such that
\begin{equation}\label{eq:char_hat_phi}
g(\hat\varphi(\hat x))=\hat x(g\circ\varphi)\quad\text{ for every }g\in\LIP_{b,1}(\Y,\tau_\Y,\sfd_\Y).
\end{equation}
Since \(\LIP_{b,1}(\Y,\tau_\Y,\sfd_\Y)\) separates the points of \(\Y\), it is clear that if such a map \(\hat\varphi\) exists, then it is unique.
To prove existence, fix \(\hat x\in\hat\X\) and take a net \((x_i)_{i\in I}\subseteq\X\) such that \(\hat\iota(x_i)\) weakly\(^*\) converges to \(\hat x\).
Since \((\Y,\tau_\Y)\) is compact and Hausdorff, up to passing to a non-relabelled subnet we can also assume that \(\lim_{i\in I}\varphi(x_i)=y_{\hat x}\)
for some \(y_{\hat x}\in\Y\), where the limit is taken with respect to the topology \(\tau_\Y\). Therefore, we can compute
\[
\hat x(g\circ\varphi)=\lim_{i\in I}\hat\iota(x_i)(g\circ\varphi)=\lim_{i\in I}g(\varphi(x_i))=g(y_{\hat x}),
\]
thus by letting \(\hat\varphi(\hat x)\coloneqq y_{\hat x}\) for every \(\hat x\in\hat\X\) we obtain the sought map \(\hat\varphi\) satisfying \eqref{eq:char_hat_phi}.
Next, let us check that \(\hat\varphi\colon(\hat\X,\hat\tau,\hat\sfd)\to(\Y,\tau_\Y,\sfd_\Y)\) is a continuous-short map. To prove continuity, fix a
weakly\(^*\) converging net \(x_i\overset{*}{\rightharpoonup}x\) in \(\hat\X\). Then we have that
\[
(g\circ\hat\varphi)(\hat x)=\hat x(g\circ\varphi)=\lim_{i\in I}\hat x_i(g\circ\varphi)=\lim_{i\in I}(g\circ\hat\varphi)(\hat x_i)\quad\text{ for every }g\in\LIP_{b,1}(\Y,\tau_\Y,\sfd_\Y).
\]
Since \(\tau_\Y\) is the initial topology of \(\LIP_{b,1}(\Y,\tau_\Y,\sfd_\Y)\), we deduce that \(\hat\varphi(\hat x)=\lim_{i\in I}\hat\varphi(\hat x_i)\) in
\((\Y,\tau_\Y)\), thus proving the continuity of \(\hat\varphi\colon(\hat\X,\hat\tau)\to(\Y,\tau_\Y)\). Moreover, for any \(\hat x,\hat y\in\hat\X\) we have that
\[\begin{split}
\sfd_\Y(\hat\varphi(\hat x),\hat\varphi(\hat y))&=\sup_{g\in\LIP_{b,1}(\Y,\tau_\Y,\sfd_\Y)}\big|g(\hat\varphi(\hat x))-g(\hat\varphi(\hat y))\big|
=\sup_{g\in\LIP_{b,1}(\Y,\tau_\Y,\sfd_\Y)}\big|\hat x(g\circ\varphi)-\hat y(g\circ\varphi)\big|\\
&\leq\sup_{f\in\LIP_{b,1}(\X,\tau,\sfd)}|\hat x(f)-\hat y(f)|=\hat\sfd(\hat x,\hat y),
\end{split}\]
which shows that \(\hat\varphi\colon(\hat\X,\hat\sfd)\to(\Y,\sfd_\Y)\) is a short map. Finally, we claim that \(\hat\varphi\colon\hat\X\to\Y\) is the only
continuous-short map such that \(\hat\varphi\circ\hat\iota=\varphi\). For any \(x\in\X\), we have that
\[
g((\hat\varphi\circ\hat\iota)(x))=\hat\iota(x)(g\circ\varphi)=g(\varphi(x))\quad\text{ for every }g\in\LIP_{b,1}(\Y,\tau_\Y,\sfd_\Y),
\]
so that \((\hat\varphi\circ\hat\iota)(x)=\varphi(x)\) thanks to the fact that \(\LIP_{b,1}(\Y,\tau_\Y,\sfd_\Y)\) separates the points of \(\Y\).
This shows that \(\hat\varphi\circ\hat\iota=\varphi\). Given any continuous-short map \(\psi\colon(\hat\X,\hat\tau,\hat\sfd)\to(\Y,\tau_\Y,\sfd_\Y)\)
such that \(\psi\circ\hat\iota=\varphi\), we have that \(\psi\) and \(\hat\varphi\) agree on \(\hat\iota(\X)\). Since \(\psi\), \(\hat\varphi\)
are continuous and \(\hat\iota(\X)\) is dense in \((\hat\X,\hat\tau)\), we deduce that \(\psi=\hat\varphi\), showing uniqueness of \(\hat\varphi\).
All in all, the statement is achieved.
\end{proof}
\subsection{\texorpdfstring{\({\bf Tych}\)}{Tych} and \texorpdfstring{\({\bf ExtMet}\)}{ExtMet} are coreflective
subcategories of \texorpdfstring{\({\bf ExtMetTop}\)}{ExtMetTop}}\label{s:Tych_ExtMet}\label{s:Tych_and_ExtMet}
In this section, we show that \({\bf ExtMetTop}\) encompasses both the category \({\bf Tych}\) of Tychonoff spaces
and the category \({\bf ExtMet}\) of extended metric spaces, corresponding to `purely-topological' and `purely-metric'
e.m.t.\ spaces, respectively, in a sense that we shall make precise.
\subsubsection*{\texorpdfstring{\({\bf Tych}\)}{Tych} is a coreflective subcategory of \texorpdfstring{\({\bf ExtMetTop}\)}{ExtMetTop}}
For clarity of presentation, it is convenient to use the following unified notation: given any \(\lambda\in(0,\infty]\), we denote
\[
{\bf ExtMetTop}_{\leq\lambda}\coloneqq\left\{\begin{array}{ll}
{\bf MetTop}_{\leq\lambda}\\
{\bf ExtMetTop}
\end{array}\quad\begin{array}{ll}
\text{ if }\lambda<\infty,\\
\text{ if }\lambda=\infty.
\end{array}\right.
\]
For any set \(\X\), we define the \emph{\(\lambda\)-discrete distance} \(\sfd_{\lambda\text{-discr}}^\X\) on \(\X\) as
\[
\sfd_{\lambda\text{-discr}}^\X(x,y)\coloneqq\left\{\begin{array}{ll}
0\\
\lambda
\end{array}\quad\begin{array}{ll}
\text{ for every }(x,y)\in\X\times\X\text{ such that }x=y,\\
\text{ for every }(x,y)\in\X\times\X\text{ such that }x\neq y.
\end{array}\right.
\]
It can be readily checked that \((\X,\tau,\sfd_{\lambda\text{-discr}}^\X)\) is an e.m.t.\ space for
every Tychonoff space \((\X,\tau)\). Moreover, if \(\varphi\colon(\X,\tau_\X)\to(\Y,\tau_\Y)\) is a
continuous map between Tychonoff spaces, then we have that
\(\varphi\colon(\X,\tau_\X,\sfd_{\lambda\text{-discr}}^\X)\to(\Y,\tau_\Y,\sfd_{\lambda\text{-discr}}^\Y)\)
is a continuous-short map. In view of these observations, it makes sense to define the
\emph{\(\lambda\)-discretisation functor} \(\mathfrak D_\lambda\colon{\bf Tych}\to{\bf ExtMetTop}_{\leq\lambda}\) as
\begin{align*}
\mathfrak D_\lambda((\X,\tau))\coloneqq(\X,\tau,\sfd_{\lambda\text{-discr}}^\X)&\quad
\text{ for every object }(\X,\tau)\text{ in }{\bf Tych},\\
\mathfrak D_\lambda(\varphi)\coloneqq\varphi\colon\mathfrak D_\lambda((\X,\tau_\X))\to\mathfrak D_\lambda((\Y,\tau_\Y))&\quad
\text{ for every morphism }\varphi\colon(\X,\tau_\X)\to(\Y,\tau_\Y)
\text{ in }{\bf Tych}.
\end{align*}
\begin{theorem}\label{thm:Tych_in_ExtMetTop}
Let \(\lambda\in(0,\infty]\) be given. Then the functor \(\mathfrak D_\lambda\colon{\bf Tych}\to{\bf ExtMetTop}_{\leq\lambda}\) is fully faithful
and injective on objects. Moreover, \(\mathfrak D_\lambda\) is the left adjoint to the forgetful functor
\[
U_\lambda\colon{\bf ExtMetTop}_{\leq\lambda}\to{\bf Tych}.
\]
In particular, \(\mathfrak D_\lambda\) exhibits \({\bf Tych}\) as a coreflective subcategory of \({\bf ExtMetTop}_{\leq\lambda}\).
\end{theorem}
\begin{proof}
Trivially, \(\mathfrak D_\lambda\) is injective on objects. The fact that \(\mathfrak D_\lambda\) is fully faithful readily follows from the
observation that any continuous map \(\varphi\colon(\X,\tau_\X)\to(\Y,\tau_\Y)\) between Tychonoff spaces gives a continuous-short map
\(\varphi\colon(\X,\tau_\X,\sfd_{\lambda\text{-discr}}^\X)\to(\Y,\tau_\Y,\sfd_{\lambda\text{-discr}}^\Y)\). Let us now check that \(\mathfrak D_\lambda\dashv U_\lambda\).
To this aim, fix \((\X,\tau,\sfd)\in{\rm Obj}({\bf ExtMetTop}_{\leq\lambda})\). Note that the identity \({\rm id}_\X\colon(\X,\sfd_{\lambda\text{-discr}}^\X)\to(\X,\sfd)\)
is a short map, thus \(\varepsilon_{(\X,\tau,\sfd)}\coloneqq{\rm id}_\X\colon(\X,\tau,\sfd_{\lambda\text{-discr}}^\X)\to(\X,\tau,\sfd)\)
is a morphism in \({\bf ExtMetTop}_{\leq\lambda}\). Next, fix an object \((\Y,\tau_\Y)\in{\rm Obj}({\bf Tych})\) and a morphism
\(\varphi\colon\mathfrak D_\lambda((\Y,\tau_\Y))=(\Y,\tau_\Y,\sfd_{\lambda\text{-discr}}^\Y)\to(\X,\tau,\sfd)\)
in \({\bf ExtMetTop}_{\leq\lambda}\). Clearly, the unique map \(\psi\colon\Y\to\X\) such that \(\varepsilon_{(\X,\tau,\sfd)}\circ\psi=\varphi\) is given
by \(\psi_\varphi\coloneqq\varphi\). Since \(\psi_\varphi\colon(\Y,\tau_\Y)\to U_\lambda((\X,\tau,\sfd))=(\X,\tau)\) is continuous and
\(\mathfrak D_\lambda(\psi_\varphi)=\psi_\varphi\), we have that \(\psi_\varphi\colon(\Y,\tau_\Y)\to(\X,\tau)\) is the unique morphism in
\({\bf Tych}\) such that \(\varepsilon_{(\X,\tau,\sfd)}\circ\mathfrak D_\lambda(\psi_\varphi)=\varphi\).
\end{proof}
\begin{remark}{\rm
Given any \(\lambda,\sigma\in(0,\infty]\) with \(\lambda\leq\sigma\), we obtain a functor
\[
\mathfrak D_{\lambda,\sigma}\colon{\bf Tych}\to{\bf ExtMetTop}_{\leq\sigma},
\]
composing \(\mathfrak D_\lambda\colon{\bf Tych}\to{\bf ExtMetTop}_{\leq\lambda}\) with the inclusion \({\bf ExtMetTop}_{\leq\lambda}\hookrightarrow{\bf ExtMetTop}_{\leq\sigma}\).
Note that \(\mathfrak D_{\lambda,\lambda}=\mathfrak D_\lambda\). It is easy to check that \(\mathfrak D_{\lambda,\sigma}\) is fully faithful and injective on objects, but
\(\mathfrak D_{\lambda,\sigma}\) is left adjoint to the forgetful functor \({\bf ExtMetTop}_{\leq\sigma}\to{\bf Tych}\) if and only if \(\lambda=\sigma\). In particular,
to exhibit \({\bf Tych}\) as a coreflective subcategory of \({\bf ExtMetTop}\) it is necessary to consider the \(\infty\)-discretisation functor \(\mathfrak D_\infty\).
\fr}\end{remark}
\subsubsection*{\texorpdfstring{\({\bf ExtMet}\)}{ExtMet} is a coreflective subcategory of \texorpdfstring{\({\bf ExtMetTop}\)}{ExtMetTop}}
Given an extended metric space \((\X,\sfd)\), we denote by \(\tau_\sfd\) the topology on \(\X\) that is induced by \(\sfd\). Note that
\((\X,\tau_\sfd,\sfd)\) is an e.m.t.\ space by Remark \ref{rmk:initial_top_Lip_metric_space}. Moreover, if \(\varphi\colon(\X,\sfd_\X)\to(\Y,\sfd_\Y)\) is a short map between extended metric spaces,
then it is clear that \(\varphi\colon(\X,\tau_{\sfd_\X},\sfd_\X)\to(\Y,\tau_{\sfd_\Y},\sfd_\Y)\) is a continuous-short map. Thus, it makes sense to
define the functor \({\rm T}\colon{\bf ExtMet}\to{\bf ExtMetTop}\) as
\begin{align*}
{\rm T}((\X,\sfd))\coloneqq(\X,\tau_\sfd,\sfd)&\quad\text{ for every object }(\X,\sfd)\text{ in }{\bf ExtMet},\\
{\rm T}(\varphi)\coloneqq\varphi\colon{\rm T}((\X,\sfd_\X))\to{\rm T}((\Y,\sfd_\Y))&\quad\text{ for every morphism }\varphi\colon(\X,\sfd_\X)\to(\Y,\sfd_\Y)
\text{ in }{\bf ExtMet}.
\end{align*}

By arguing in a similar manner as we did for Theorem \ref{thm:Tych_in_ExtMetTop}, and using the fact that for any e.m.t.\ space \((\X,\tau,\sfd)\) the topology \(\tau\)
is coarser that \(\tau_\sfd\), one can readily obtain the following result:
\begin{theorem}
The functor \({\rm T}\colon{\bf ExtMet}\to{\bf ExtMetTop}\) is fully faithful and injective on objects. Moreover, \({\rm T}\) is the left adjoint to the
forgetful functor \({\bf ExtMetTop}\to{\bf ExtMet}\). In particular, \({\rm T}\) exhibits \({\bf ExtMet}\) as a coreflective subcategory of \({\bf ExtMetTop}\).
\end{theorem}
\begin{remark}{\rm
The following more general statement holds. Fix a full subcategory \({\bf C}\) of \({\bf ExtMet}\). We denote by \({\bf C}_{\rm +top}\) the full subcategory of
\({\bf ExtMetTop}\) whose objects are given by those e.m.t.\ spaces \((\X,\tau,\sfd)\) such that \((\X,\sfd)\in{\rm Obj}({\bf C})\). Then \({\rm T}\colon{\bf ExtMet}\to{\bf ExtMetTop}\)
restricts to a functor \({\rm T}_{\bf C}\colon{\bf C}\to{\bf C}_{\rm +top}\) that is fully faithful, injective on objects, and left adjoint to the forgetful
functor \({\bf C}_{\rm +top}\to{\bf C}\). In particular, \({\rm T}_{\bf C}\) exhibits \({\bf C}\) as a coreflective subcategory of \({\bf C}_{\rm +top}\).
\fr}\end{remark}
\subsection{Truncation, metric-completion and geodesification functors}
Here, we discuss functors that are particularly useful when addressing problems in geometric analysis on e.m.t.\ spaces.
\subsubsection*{The \texorpdfstring{\(\lambda\)}{lambda}-truncation functor}
Let \(\lambda\in(0,\infty)\) be a fixed real number. Given an e.m.t.\ space \((\X,\tau,\sfd)\), it is straightforward to check that \((\X,\tau,\sfd\wedge\lambda)\)
is an e.m.t.\ space. Also, if \(\varphi\colon(\X,\tau_\X,\sfd_\X)\to(\Y,\tau_\Y,\sfd_\Y)\) is a continuous-short map between e.m.t.\ spaces, then it is clear that
\(\varphi\colon(\X,\sfd_\X\wedge\lambda)\to(\Y,\sfd_\Y\wedge\lambda)\) is a short map, so that \(\varphi\colon(\X,\tau_\X,\sfd_\X\wedge\lambda)\to(\Y,\tau_\Y,\sfd_\Y\wedge\lambda)\)
is continuous-short. Accordingly, it makes sense to define the \emph{\(\lambda\)-truncation functor} \(\mathfrak T_\lambda\colon{\bf ExtMetTop}\to{\bf MetTop}_{\leq\lambda}\) as
\begin{align*}
\mathfrak T_\lambda((\X,\tau,\sfd))\coloneqq(\X,\tau,\sfd\wedge\lambda)&\quad\text{ for every object }(\X,\tau,\sfd)\text{ in }{\bf ExtMetTop},\\
\mathfrak T_\lambda(\varphi)\coloneqq\varphi&\quad\text{ for every morphism }\varphi\colon(\X,\tau_\X,\sfd_\X)\to(\Y,\tau_\Y,\sfd_\Y)\text{ in }{\bf ExtMetTop}.
\end{align*}
Observe that, letting \(\mathfrak D_\lambda\) be as in Section \ref{s:Tych_and_ExtMet}, it holds that
\[
\mathfrak D_\lambda=\mathfrak T_\lambda\circ\mathfrak D_\infty\quad\text{ for every }\lambda\in(0,\infty).
\]
Furthermore, we have the following result (whose proof we omit):
\begin{theorem}
Let \(\lambda\in(0,\infty)\) be given. Then the functor \(\mathfrak T_\lambda\colon{\bf ExtMetTop}\to{\bf MetTop}_{\leq\lambda}\)
is the left adjoint to the inclusion functor \({\bf MetTop}_{\leq\lambda}\hookrightarrow{\bf ExtMetTop}\).
In particular, \({\bf MetTop}_{\leq\lambda}\) is a reflective subcategory of \({\bf ExtMetTop}\), with reflector \(\mathfrak T_\lambda\).
\end{theorem}
\subsubsection*{The metric-completion functor}
The following definition is taken from \cite[Corollary 2.1.36]{Sav:22}:
\begin{definition}[Metric completion of an e.m.t.\ space]
Let \((\X,\tau,\sfd)\) be an e.m.t.\ space. Then we define its \emph{metric completion} \((\bar\X,\bar\tau,\bar\sfd)\) as follows:
\(\bar\X\) is the \(\gamma\sfd\)-closure of \(\iota_{(\X,\tau,\sfd)}(\X)\) in \(\gamma\X\), and we equip it with the topology
\(\bar\tau\coloneqq(\gamma\tau)\llcorner\bar\X\) and the extended distance \(\bar\sfd\coloneqq(\gamma\sfd)|_{\bar\X\times\bar\X}\).
\end{definition}

Note that \((\bar\X,\bar\tau,\bar\sfd)\) is a metrically-complete e.m.t.\ space. Moreover, if \((\X,\tau_\X,\sfd_\X)\), \((\Y,\tau_\Y,\sfd_\Y)\)
are e.m.t.\ spaces and \(\varphi\colon(\X,\tau_\X,\sfd_\X)\to(\Y,\tau_\Y,\sfd_\Y)\) is a continuous-short map, then \(\bar\varphi|_{\bar\X}\colon\bar\X\to\bar\Y\)
is the unique continuous-short map such that
\begin{equation}\label{eq:diagr_mc}\begin{tikzcd}[column sep=huge]
(\X,\tau_\X,\sfd_\X) \arrow[r,"\iota_{(\X,\tau_\X,\sfd_\X)}"] \arrow[d,swap,"\varphi"] & (\bar\X,\bar\tau_\X,\bar\sfd_\X) \arrow[d,"\bar\varphi"] \\
(\Y,\tau_\Y,\sfd_\Y) \arrow[r,swap,"\iota_{(\Y,\tau_\Y,\sfd_\Y)}"] & (\bar\Y,\bar\tau_\Y,\bar\sfd_\Y)
\end{tikzcd}\end{equation}
is a commutative diagram. Indeed, the inclusion \((\gamma\varphi)(\bar\X)\subseteq\bar\Y\) follows from the fact that
if \(x\in\bar\X\) and \((x_n)_{n\in\N}\subseteq\X\) are chosen so that \(\bar\sfd_\X(\iota_{(\X,\tau_\X,\sfd_\X)}(x_n),x)\to 0\) as \(n\to\infty\), then we have
\[
\bar\sfd_\Y(\iota_{(\Y,\tau_\Y,\sfd_\Y)}(\varphi(x_n)),(\gamma\varphi)(x))\leq\bar\sfd_\X(\iota_{(\X,\tau_\X,\sfd_\X)}(x_n),x)\to 0\quad\text{ as }n\to\infty.
\]
It is then clear that \(\bar\varphi\) is the unique continuous-short map making the diagram in \eqref{eq:diagr_mc} commute.
\begin{definition}[The metric completion functor]
We define the \emph{metric completion functor}
\[
{\sf mc}\colon{\bf ExtMetTop}\to{\bf ExtMetTop}_{\rm mc}
\]
as follows: to any object \((\X,\tau,\sfd)\in{\rm Obj}({\bf ExtMetTop})\) we assign the object
\[
{\sf mc}((\X,\tau,\sfd))\coloneqq(\bar\X,\bar\tau,\bar\sfd)\in{\rm Obj}({\bf ExtMetTop}_{\rm mc}),
\]
and to any morphism \(\varphi\colon(\X,\tau_\X,\sfd_\X)\to(\Y,\tau_\Y,\sfd_\Y)\) in \({\bf ExtMetTop}\) we assign
the morphism
\[
{\sf mc}(\varphi)\coloneqq\bar\varphi\colon(\bar\X,\bar\tau_\X,\bar\sfd_\X)\to(\bar\Y,\bar\tau_\Y,\bar\sfd_\Y)
\quad\text{ in }{\bf ExtMetTop}_{\rm mc}.
\]
\end{definition}

We omit the (easy) proof of the following statement:
\begin{theorem}
The metric completion functor \({\sf mc}\colon{\bf ExtMetTop}\to{\bf ExtMetTop}_{\rm mc}\) is left adjoint to the inclusion functor
\({\bf ExtMetTop}_{\rm mc}\hookrightarrow{\bf ExtMetTop}\). In particular, we have that \({\bf ExtMetTop}_{\rm mc}\)
is a reflective subcategory of \({\bf ExtMetTop}\), with reflector \({\sf mc}\).
\end{theorem}
\begin{remark}{\rm
Let us highlight a significant dissimilarity between the completion functor in the category of extended metric spaces (which maps an extended metric
space \((\X,\sfd)\) to its completion \((\bar\X,\bar\sfd)\) as in Proposition \ref{prop:completion_extmet}) and the metric completion functor \({\sf mc}\)
in \({\bf ExtMetTop}\). Any given extended metric space \((\X,\sfd)\) is (canonically isomorphic to) the coproduct of
\(\{({\rm Z},\sfd|_{{\rm Z}\times{\rm Z}}):{\rm Z}\in\mathscr F_\sfd\}\) in \({\bf ExtMet}\), thus its metric completion
\((\bar\X,\bar\sfd)\) is the coproduct of
\(\{(\bar{\rm Z},\bar\sfd|_{\bar{\rm Z}\times\bar{\rm Z}}):{\rm Z}\in\mathscr F_\sfd\}\) in \({\bf ExtMet}\),
consistently with the general principle that left adjoint functors (such as the completion functor)
commute with colimits (such as the coproduct). By contrast, an e.m.t.\ space \((\X,\tau,\sfd)\) is not always
the coproduct of \(\{({\rm Z},\tau\llcorner{\rm Z},\sfd|_{{\rm Z}\times{\rm Z}}):{\rm Z}\in\mathscr F_\sfd\}\)
in \({\bf ExtMetTop}\), and the metric completion \((\bar\X,\bar\tau,\bar\sfd)\) is not always the coproduct
of \(\{(\bar{\rm Z},\bar\tau\llcorner\bar{\rm Z},\bar\sfd|_{\bar{\rm Z}\times\bar{\rm Z}}):{\rm Z}\in\mathscr F_\sfd\}\)
in \({\bf ExtMetTop}_{\rm mc}\), as easy examples illustrate (for instance, equip a Tychonoff space whose topology is not discrete
with the \(\infty\)-discrete distance). This means, roughly speaking, that -- in general -- the metric completion of an
e.m.t.\ space can be obtained only by considering the e.m.t.\ space `as a whole'; in other words, the metric completion
cannot be constructed starting from the metric completions of the components \({\rm Z}\in\mathscr F_\sfd\), differently
from what happens with the completion of an extended metric space.
\fr}\end{remark}
\subsubsection*{The geodesification functor}
Lastly, we introduce the geodesification functor, which allows us to reduce many problems for compact e.m.t.\ spaces
to the setting of compact geodesic e.m.t.\ spaces.
\begin{definition}[The geodesification functor]
We define the \emph{geodesification functor}
\[
{\sf geo}\colon{\bf CptExtMetTop}\to{\bf CptExtMetTop}_{\rm geo}
\]
as follows: to any object \((\X,\tau,\sfd)\in{\rm Obj}({\bf CptExtMetTop})\) we assign the object
\begin{equation}\label{eq:geodes_functor}
{\sf geo}((\X,\tau,\sfd))\coloneqq(\X,\tau,\sfd_\ell)\in{\rm Obj}({\bf CptExtMetTop}_{\rm geo}),
\end{equation}
and to any morphism \(\varphi\colon(\X,\tau_\X,\sfd_\X)\to(\Y,\tau_\Y,\sfd_\Y)\) in \({\bf CptExtMetTop}\) we assign
the morphism \({\sf geo}(\varphi)\coloneqq\varphi\colon(\X,\tau_\X,(\sfd_\X)_\ell)\to(\Y,\tau_\Y,(\sfd_\Y)_\ell)\) in \({\bf CptExtMetTop}_{\rm geo}\).
\end{definition}

Note that \eqref{eq:geodes_functor} is well posed thanks to \eqref{eq:cpt_gives_length_emt}. Moreover,
if \(\varphi\colon(\X,\tau_\X,\sfd_\X)\to(\Y,\tau_\Y,\sfd_\Y)\) is a morphism in \({\bf CptExtMetTop}\), then a fortiori
\(\varphi\colon(\X,\tau_\X,\sfd_\X)\to(\Y,\tau_\Y,(\sfd_\Y)_\ell)\) is a continuous-short map (as \(\sfd_\Y\leq(\sfd_\Y)_\ell\))
and accordingly also \(\varphi\colon(\X,\tau_\X,(\sfd_\X)_\ell)\to(\Y,\tau_\Y,(\sfd_\Y)_\ell)\) is a continuous-short
map by Remark \ref{rmk:geo_functor_mor}. All in all, we have shown that the functor \({\sf geo}\) is well defined.
\medskip

We omit the (easy) proof of the following statement:
\begin{theorem}
The geodesification functor \({\sf geo}\colon{\bf CptExtMetTop}\to{\bf CptExtMetTop}_{\rm geo}\) is left adjoint to the inclusion functor
\({\bf CptExtMetTop}_{\rm geo}\hookrightarrow{\bf CptExtMetTop}\). In particular, we have that \({\bf CptExtMetTop}_{\rm geo}\)
is a reflective subcategory of \({\bf CptExtMetTop}\), with reflector \({\sf geo}\).
\end{theorem}

In the above discussion we consider only e.m.t.\ spaces that are compact, on account of \eqref{eq:cpt_gives_length_emt}.
On the other hand, it is not clear to us whether it is possible to construct a functor sending each (possibly non-compact)
e.m.t.\ space to some associated length e.m.t.\ space, if any.
\begin{remark}{\rm
The category \({\bf ExtMetTop}_{\rm len}\) admits all equalisers and products, as one can readily check by showing
that all equalisers and products in \({\bf ExtMetTop}\) preserve the property of being a length e.m.t.\ space;
we omit the details. Therefore, \({\bf ExtMetTop}_{\rm len}\) is a complete category, thanks to
Theorem \ref{thm:exist_thm_lim}. However, we do not know whether \({\bf ExtMetTop}_{\rm len}\) is cocomplete.
\fr}\end{remark}
\appendix
\section{Compendium of category theory}\label{app:categories}
In this appendix, we collect those concepts and results in category theory that have an essential r\^{o}le in this paper.
We refer to \cite{riehl2017category} for a detailed account of category theory. See also \cite{MacLane98,kashiwara2005categories,Awodey10}.
\medskip

By a functor between categories, we always mean a covariant functor. A functor that is injective (resp.\ surjective)
on hom-sets is said to be \emph{faithful} (resp.\ \emph{full}). A functor that is both full and faithful is said to
be \emph{fully faithful}. Given a category \({\bf C}\) and a subcategory \({\bf S}\) of \({\bf C}\), the inclusion
functor \({\bf S}\hookrightarrow{\bf C}\) is faithful and injective on objects, and it is a full functor if and only
if \({\bf S}\) is a full subcategory of \({\bf C}\). For any given collection \(A\) of objects of \({\bf C}\), there exists
a unique full subcategory of \({\bf C}\) whose objects are exactly those in \(A\). Given a morphism \(\varphi\colon a\to b\)
in \({\bf C}\), we say that:
\begin{itemize}
\item \(\varphi\) is a \emph{monomorphism} if \(\psi_1=\psi_2\) whenever \(c\in{\rm Obj}({\bf C})\) and \(\psi_1,\psi_2\colon c\to a\)
are morphisms in \({\bf C}\) satisfying \(\varphi\circ\psi_1=\varphi\circ\psi_2\).
\item \(\varphi\) is an \emph{epimorphism} if \(\psi_1=\psi_2\) whenever \(c\in{\rm Obj}({\bf C})\) and \(\psi_1,\psi_2\colon b\to c\)
are morphisms in \({\bf C}\) satisfying \(\psi_1\circ\varphi=\psi_2\circ\varphi\).
\item \(\varphi\) is a \emph{bimorphism} if it is both a monomorphism and an epimorphism.
\item \(\varphi\) is an \emph{isomorphism} if there exists a (necessarily unique) morphism \(\varphi^{-1}\colon b\to a\) in \({\bf C}\) such that
\(\varphi^{-1}\circ\varphi={\rm id}_a\) and \(\varphi\circ\varphi^{-1}={\rm id}_b\). The morphism \(\varphi^{-1}\) is called the \emph{inverse} of \(\varphi\).
\end{itemize}
A category \({\bf C}\) is said to be \emph{balanced} if every bimorphism in \({\bf C}\) is an isomorphism.
\medskip

An \emph{initial object} of a category \({\bf C}\) is an object \(i\) of \({\bf C}\) such that for any \(a\in{\rm Obj}({\bf C})\)
there exists a unique morphism \(i\to a\). Dually, a \emph{terminal object} of \({\bf C}\) is an object \(t\) of \({\bf C}\) such
that for any \(a\in{\rm Obj}({\bf C})\) there exists a unique morphism \(a\to t\). Initial and terminal objects do not always
exist, but when they do, they are unique up to a unique isomorphism.
\subsubsection*{Comma categories}
Let \({\bf A}\), \({\bf B}\) be given categories. Given an object \(b\) of \({\bf B}\) and a functor \(F\colon{\bf A}\to{\bf B}\),
\[
F\downarrow b
\]
denotes the \emph{comma category} of \(F\) and \(b\), which is defined as follows:
\begin{itemize}
\item The objects of \(F\downarrow b\) are all pairs \((a,\varphi)\), where \(a\) is an object of \({\bf A}\) and \(\varphi\colon F(a)\to b\) is a morphism in \({\bf B}\).
\item The morphisms between two objects \((a,\varphi)\) and \((a',\varphi')\) of \(F\downarrow b\) are given by those morphisms \(\alpha\colon a\to a'\) in \({\bf A}\)
such that
\[\begin{tikzcd}
F(a) \arrow[r,"F(\alpha)"] \arrow[dr,swap,"\varphi"] & F(a') \arrow[d,"\varphi'"] \\
& b
\end{tikzcd}\]
is a commutative diagram.
\end{itemize}
Likewise, one can define the comma category \(b\downarrow F\), whose objects are pairs \((a,\varphi)\) with \(\varphi\colon b\to F(a)\)
and whose morphisms \(\alpha\colon(a,\varphi)\to(a',\varphi')\) are those morphisms \(\alpha\colon a\to a'\) such that \(F(\alpha)\circ\varphi=\varphi'\).
More generally, it is possible to define the comma category \(F\downarrow G\) of any two given functors \(F\colon{\bf A}\to{\bf B}\) and \(G\colon{\bf C}\to{\bf B}\),
see e.g.\ \cite[Exercise 1.3.vi]{riehl2017category}.
\subsubsection*{Strict pullback}
Given two categories \({\bf A}\) and \({\bf B}\), their \emph{product category} \({\bf A}\times{\bf B}\)
is defined as follows:
\begin{itemize}
\item The objects of \({\bf A}\times{\bf B}\) are all pairs \((a,b)\), where \(a\) (resp.\ \(b)\)
is an object of \({\bf A}\) (resp.\ of \({\bf B}\)).
\item The morphisms between two objects \((a,b)\) and \((a',b')\) of \({\bf A}\times{\bf B}\) are given
by all pairs \((\varphi,\psi)\), where \(\varphi\colon a\to a'\) is a morphism in \({\bf A}\) and
\(\psi\colon b\to b'\) is a morphism in \({\bf B}\).
\end{itemize}
The composition of two morphisms \((\varphi,\psi)\colon(a,b)\to(a',b')\) and \((\varphi',\psi')\colon(a',b')\to(a'',b'')\)
in \({\bf A}\times{\bf B}\) is defined as \((\varphi',\psi')\circ(\varphi,\psi)\coloneqq(\varphi'\circ\varphi,\psi'\circ\psi)\colon(a,b)\to(a'',b'')\).
We denote by \(\pi_1\colon{\bf A}\times{\bf B}\to{\bf A}\) and \(\pi_2\colon{\bf A}\times{\bf B}\to{\bf B}\)
the first and second \emph{projection functors}, respectively, which are defined as
\[\begin{split}
\pi_1((a,b))\coloneqq a,\quad\pi_2((a,b))\coloneqq b&\quad\text{ for every object }(a,b)\text{ of }{\bf A}\times{\bf B},\\
\pi_1((\varphi,\psi))\coloneqq\varphi,\quad\pi_2((\varphi,\psi))\coloneqq\psi&\quad\text{ for every morphism }
(\varphi,\psi)\text{ in }{\bf A}\times{\bf B}.
\end{split}\]
\begin{definition}[Strict pullback of functors]\label{def:strict_pullback}
Let \({\bf A}\), \({\bf B}\) and \({\bf C}\) be given categories. Let \(F\colon{\bf A}\to{\bf C}\) and
\(G\colon{\bf B}\to{\bf C}\) be two functors. Then we define the \emph{strict pullback} of \(F\) and \(G\)
as the following subcategory \({\bf SP}(F,G)\) of \({\bf A}\times{\bf B}\):
\begin{itemize}
\item The objects of \({\bf SP}(F,G)\) are all those \((a,b)\in{\rm Obj}({\bf A}\times{\bf B})\) such that \(F(a)=G(b)\).
\item The morphisms in \({\bf SP}(F,G)\) are all those morphisms \((\varphi,\psi)\colon(a,b)\to(a',b')\) in \({\bf A}\times{\bf B}\)
such that \((a,b),(a',b')\in{\rm Obj}({\bf SP}(F,G))\) and \(F(\varphi)=G(\psi)\).
\end{itemize}
\end{definition}

Equivalently, \({\bf SP}(F,G)\) is the full subcategory of the comma category \(F\downarrow G\) whose objects are those
triples \((a,b,h)\in{\rm Obj}(F\downarrow G)\) such that \(F(a)=G(b)\) and \(h={\rm id}_{F(a)}\colon F(a)\to F(a)=G(b)\).
\subsubsection*{Natural transformations}
Let \({\bf A}\), \({\bf B}\) be given categories. Let \(F,G\colon{\bf A}\to{\bf B}\) be two functors.
Then a \emph{natural transformation} from \(F\) to \(G\) is a collection \(\eta_\star\) of morphisms
\(\eta_a\colon F(a)\to G(a)\) in \({\bf B}\), indexed over \(a\in{\rm Obj}({\bf A})\), such that the diagram
\[\begin{tikzcd}
F(a) \arrow[r,"\eta_a"] \arrow[d,swap,"F(\varphi)"] & G(a) \arrow[d,"G(\varphi)"] \\
F(b) \arrow[r,swap,"\eta_b"] & G(b)
\end{tikzcd}\]
commutes for every morphism \(\varphi\colon a\to b\) in \({\bf A}\). The \emph{functor category} \({\bf B}^{\bf A}\)
is defined as follows: the objects of \({\bf B}^{\bf A}\) are all functors from \({\bf A}\) to \({\bf B}\), while the
morphisms of \({\bf B}^{\bf A}\) are all natural transformations between functors from \({\bf A}\) to \({\bf B}\).
Given three functors \(F,G,H\colon{\bf A}\to{\bf B}\), a natural transformation \(\eta_\star\) from \(F\) to \(G\), and
a natural transformation \(\zeta_\star\) from \(G\) to \(H\), the composition \(\xi_\star\coloneqq\zeta_\star\circ\eta_\star\)
is the natural transformation given by \(\xi_a\coloneqq\zeta_a\circ\eta_a\colon F(a)\to H(a)\) for every \(a\in{\rm Obj}({\bf A})\).
\subsubsection*{Cones and cocones}
By a \emph{diagram} of type \({\bf J}\) in a given category \({\bf C}\), we mean a functor \(D\colon{\bf J}\to{\bf C}\)
that is defined on a small category \({\bf J}\), which we refer to as an \emph{index category} (it is worth stressing that,
differently from some other authors, we require all index categories to be small). In other words, the diagrams of type
\({\bf J}\) in \({\bf C}\) are exactly the objects of the functor category \({\bf C}^{\bf J}\), which we thus rename to
\emph{category of diagrams}. The \emph{diagonal functor} \(\Delta_{{\bf J},{\bf C}}\colon{\bf C}\to{\bf C}^{\bf J}\)
is defined as follows:
\begin{itemize}
\item For any \(a\in{\rm Obj}({\bf C})\), we define \(\Delta_{{\bf J},{\bf C}}(a)\in{\rm Obj}({\bf C}^{\bf J})\)
as \(\Delta_{{\bf J},{\bf C}}(a)(i)\coloneqq a\) for every \(i\in{\rm Obj}({\bf J})\) and
\(\Delta_{{\bf J},{\bf C}}(a)(\phi)\coloneqq{\rm id}_a\) for every morphism \(\phi\colon i\to j\) in \({\bf J}\).
\item For any morphism \(\varphi\colon a\to b\) in \({\bf C}\), we define the natural transformation
\(\Delta_{{\bf J},{\bf C}}(\varphi)_\star\) as \(\Delta_{{\bf J},{\bf C}}(\varphi)_i\coloneqq\varphi\)
for every \(i\in{\rm Obj}({\bf J})\).
\end{itemize}
The \emph{category of cones} to \(D\) of type \({\bf J}\) is defined as \({\bf Cone}(D)\coloneqq\Delta_{{\bf J},{\bf C}}\downarrow D\),
whereas the \emph{category of cocones} from \(D\) of type \({\bf J}\) is defined as
\({\bf Cocone}(D)\coloneqq D\downarrow\Delta_{{\bf J},{\bf C}}\).
Videlicet, a cone to \(D\) (resp.\ a cocone from \(D\)) of type \({\bf J}\) is given by a pair \((c,\varphi_\star)\), where
\(c\in{\rm Obj}({\bf C})\) and \(\varphi_\star=\{\varphi_i:i\in{\rm Obj}({\bf J})\}\) is a collection of morphisms
\(\varphi_i\colon c\to D(i)\) (resp.\ \(\varphi_i\colon D(i)\to c\)) satisfying \(D(\phi)\circ\varphi_i=\varphi_j\)
(resp.\ \(\varphi_j\circ D(\phi)=\varphi_i\)) for every morphism \(\phi\colon i\to j\) in \({\bf J}\).
Given another category \({\bf A}\) and a functor \(F\colon{\bf C}\to{\bf A}\), we have that \(F\circ D\colon{\bf J}\to{\bf A}\)
is a diagram of type \({\bf J}\) in \({\bf A}\). Moreover, we denote by
\[
F_{D,{\rm cone}}\colon{\bf Cone}(D)\to{\bf Cone}(F\circ D),\qquad F_{D,{\rm cocone}}\colon {\bf Cocone}(D)\to{\bf Cocone}(F\circ D)
\]
the functors sending cones to cones and cocones to cocones, respectively, that are induced by \(F\):
\[\begin{split}
F_{D,{\rm cone}}((c,\varphi_\star))&\coloneqq(F(c),F(\varphi_\star))\quad\text{ for every cone }(c,\varphi_\star)\text{ to }D,\\
F_{D,{\rm cocone}}((c,\varphi_\star))&\coloneqq(F(c),F(\varphi_\star))\quad\text{ for every cocone }(c,\varphi_\star)\text{ from }D,
\end{split}\]
where we set \(F_{D,{\rm cone}}(\varphi_\star)\coloneqq\{F(\varphi_i):i\in{\rm Obj}({\bf J})\}\eqqcolon F_{D,{\rm cocone}}(\varphi_\star)\).
\subsubsection*{Limits and colimits}
Let \(D\colon{\bf J}\to{\bf C}\) be a diagram of type \({\bf J}\) in a category \({\bf C}\). 
A terminal object of \({\bf Cone}(D)\) (resp.\ an initial object of \({\bf Cocone}(D)\))
is called a \emph{limit} (resp.\ a \emph{colimit}) of \(D\). In other words, limits are universal
cones to \(D\), while colimits are universal cocones from \(D\). If \({\bf J}\) is a finite category, we speak of finite (co)limits.
Some relevant examples of (co)limits are the following:
\begin{itemize}
\item Terminal (resp.\ initial) objects are limits (resp.\ colimits) of empty diagrams.
\item Limits (resp.\ colimits) of diagrams indexed by discrete categories (i.e.\ whose only morphisms are the identity morphisms)
are called \emph{products} (resp. \emph{coproducts}).
\item Limits (resp.\ colimits) of the diagram \({\bf J}_\rightrightarrows\), which consists of two objects \(i\), \(j\) and of
two parallel morphisms from \(i\) to \(j\), are called \emph{equalisers} (resp.\ \emph{coequalisers}).
\item If the index category \({\bf J}_\preceq\) is induced by a directed set \((I,\preceq)\) (i.e.\ \({\rm Obj}({\bf J}_\preceq)=I\)
and, given any \(i,j\in I\), we have that \(i\preceq j\) if and only if there is a unique morphism \(i\to j\) in \({\bf J}_\preceq\)),
then limits (resp.\ colimits) of the diagram \({\bf J}_\preceq\) are called \emph{inverse limits} (resp.\ \emph{direct limits}).
\end{itemize}
(Co)limits do not always exist, but when they do, they are unique up to a unique isomorphism.
\subsubsection*{Completeness and cocompleteness}
The following definitions concern the existence of (co)limits:
\begin{itemize}
\item We say that a category is \emph{finitely complete} (resp.\ \emph{finitely cocomplete}) if it admits all finite limits
(resp.\ all finite colimits).
\item We say that a category is \emph{complete} (resp.\ \emph{cocomplete}) if it admits all limits (resp.\ colimits).
\item A category that is both complete and cocomplete is said to be \emph{bicomplete}.
\end{itemize}
We point out that in the definition of (co/bi)completeness, we ask for the existence only of small (co)limits
(i.e.\ only of (co)limits of diagrams of type \({\bf J}\), where \({\bf J}\) is a small index category).
Let us also recall a useful criterion for checking the (co)completeness of a category \cite[\S V.2 Theorem 1]{MacLane98}:
\begin{theorem}[Existence theorem for (co)limits]\label{thm:exist_thm_lim}
A category is complete (resp.\ cocomplete) if and only if it admits all products and equalisers (resp.\ all coproducts
and coequalisers).
\end{theorem}

Given two functors \(F\colon{\bf A}\to{\bf C}\) and \(G\colon{\bf B}\to{\bf C}\), and a diagram \(D\colon{\bf J}\to{\bf SP}(F,G)\),
it is easy to check that we have two isomorphisms of categories
\[\begin{split}
{\rm I}_{F,G}^{\rm cone}&\colon{\bf SP}(F_{\pi_1\circ D,{\rm cone}},G_{\pi_2\circ D,{\rm cone}})\to{\bf Cone}(D),\\
{\rm I}_{F,G}^{\rm cocone}&\colon{\bf SP}(F_{\pi_1\circ D,{\rm cocone}},G_{\pi_2\circ D,{\rm cocone}})\to{\bf Cocone}(D)
\end{split}\]
satisfying \({\rm I}_{F,G}^{\rm cone}(((c^1,\varphi^1_\star),(c^2,\varphi^2_\star)))=
((c^1,c^2),(\varphi^1_\star,\varphi^2_\star))={\rm I}_{F,G}^{\rm cocone}(((c^1,\varphi^1_\star),(c^2,\varphi^2_\star)))\).
\begin{lemma}\label{lem:limits_SP}
Let \({\bf A}\), \({\bf B}\) and \({\bf C}\) be given categories. Let \(F\colon{\bf A}\to{\bf C}\) and
\(G\colon{\bf B}\to{\bf C}\) be two functors. Let \(D\colon{\bf J}\to{\bf SP}(F,G)\) be a diagram such that
the following conditions are satisfied:
\begin{itemize}
\item[\(\rm i)\)] \(\pi_1\circ D\colon{\bf J}\to{\bf A}\) has a limit \(L_{\bf A}\) (resp.\ a colimit \(C_{\bf A}\)).
\item[\(\rm ii)\)] \(\pi_2\circ D\colon{\bf J}\to{\bf B}\) has a limit \(L_{\bf B}\) (resp.\ a colimit \(C_{\bf B}\)).
\item[\(\rm iii)\)] \(F_{\pi_1\circ D,{\rm cone}}(L_{\bf A})=G_{\pi_2\circ D,{\rm cone}}(L_{\bf B})\)
(resp.\ \(F_{\pi_1\circ D,{\rm cocone}}(C_{\bf A})=G_{\pi_2\circ D,{\rm cocone}}(C_{\bf B})\)).
\end{itemize}
Then it holds that \(D\colon{\bf J}\to{\bf SP}(F,G)\) has limit \({\rm I}_{F,G}^{\rm cone}((L_{\bf A},L_{\bf B}))\)
(resp.\ colimit \({\rm I}_{F,G}^{\rm cocone}((C_{\bf A},C_{\bf B}))\)).
\end{lemma}
\begin{proof}
We discuss only limits, as the argument for colimits is similar. Since \(L_{\bf A}\) and \(L_{\bf B}\)
are terminal objects of \({\bf Cone}(\pi_1\circ D)\) and \({\bf Cone}(\pi_2\circ D)\), respectively, it is clear that
\((L_{\bf A},L_{\bf B})\) is a terminal object of \({\bf Cone}(\pi_1\circ D)\times{\bf Cone}(\pi_2\circ D)\).
Given that \((L_{\bf A},L_{\bf B})\in{\rm Obj}({\bf SP}(F_{\pi_1\circ D,{\rm cone}},G_{\pi_2\circ D,{\rm cone}}))\)
by iii), we have a fortiori that \((L_{\bf A},L_{\bf B})\) is a terminal object of
\({\bf SP}(F_{\pi_1\circ D,{\rm cone}},G_{\pi_2\circ D,{\rm cone}})\). By recalling that \({\rm I}_{F,G}^{\rm cone}\)
is an isomorphism of categories, we conclude that \({\rm I}_{F,G}^{\rm cone}((L_{\bf A},L_{\bf B}))\) is a terminal
object of \({\bf Cone}(D)\), or in other words that it is a limit of the diagram \(D\colon{\bf J}\to{\bf SP}(F,G)\).
\end{proof}
\subsubsection*{Reflective and coreflective subcategories}
Let \({\bf A}\), \({\bf B}\) be categories. Let \(F\colon{\bf A}\to{\bf B}\) and \(G\colon{\bf B}\to{\bf A}\) be given functors.
Then we say that \(G\) is the \emph{left adjoint} to \(F\) (or \(F\) is the \emph{right adjoint} to \(G\)),
\[
G\dashv F
\]
for short, if for any \(a\in{\rm Obj}({\bf A})\) there exists a morphism \(\varepsilon_a\colon G(F(a))\to a\) in \({\bf A}\) such that
the following property holds: given any \(b\in{\rm Obj}({\bf B})\) and any morphism \(\varphi\colon G(b)\to a\) in \({\bf A}\),
there exists a unique morphism \(\psi\colon b\to F(a)\) in \({\bf B}\) such that
\[\begin{tikzcd}
G(b) \arrow[d,swap,"G(\psi)"] \arrow[dr,"\varphi"] & \\
G(F(a)) \arrow[r,swap,"\varepsilon_a"] & a
\end{tikzcd}\]
is a commutative diagram. If in addition \(F\colon{\bf A}\to{\bf B}\) (resp.\ \(G\colon{\bf B}\to{\bf A}\)) is a fully faithful functor,
then we say that \(F\) exhibits \({\bf A}\) as a \emph{reflective subcategory} of \({\bf B}\) (resp.\ \(G\) exhibits \({\bf B}\)
as a \emph{coreflective subcategory} of \({\bf A}\)), with \emph{reflector} \(G\) (resp.\ with \emph{coreflector} \(F\)).
For a proof of the following result, we refer to \cite[Proposition 4.5.15]{riehl2017category}:
\begin{proposition}\label{prop:create_lim}
Let \({\bf A}\), \({\bf B}\) be given categories. Assume \(F\colon{\bf A}\to{\bf B}\) is a functor that exhibits \({\bf A}\)
as a reflective (resp.\ coreflective) subcategory of \({\bf B}\), with reflector (resp.\ coreflector) \(G\). Then:
\begin{itemize}
\item[\(\rm i)\)] If \(D\colon{\bf J}\to{\bf A}\) is a diagram such that \(F\circ D\colon{\bf J}\to{\bf B}\) has a limit \(L\) (resp.\ a colimit \(C\)),
then \(D\) has a limit \(\tilde L\) (resp.\ a colimit \(\tilde C\)) and \(F_{D,{\rm cone}}(\tilde L)=L\) (resp.\ \(F_{D,{\rm cocone}}(\tilde C)=C\)).
\item[\(\rm ii)\)] If \(D\colon{\bf J}\to{\bf A}\) is a diagram such that \(F\circ D\colon{\bf J}\to{\bf B}\) has a colimit \(C\) (resp.\ a limit \(L\)),
then \(D\) has colimit \(G_{D,{\rm cocone}}(C)\) (resp.\ has limit \(G_{D,{\rm cone}}(L)\)).
\end{itemize}
\end{proposition}

The property in i) is typically formulated by saying that \(F\) `creates' all limits (resp.\ colimits) that \({\bf B}\) admits.
A useful, immediate consequence of Proposition \ref{prop:create_lim} is the following result:
\begin{corollary}\label{cor:compl_preserved}
Let \({\bf A}\), \({\bf B}\) be categories such that \({\bf A}\) is either a reflective or a coreflective subcategory of \({\bf B}\).
Assume \({\bf B}\) is complete (resp.\ cocomplete). Then \({\bf A}\) is complete (resp.\ cocomplete).
\end{corollary}
\section{Equivalent definitions of extended metric-topological space}\label{app:equiv_emt}
It is well known that Tychonoff spaces admit several equivalent characterisations. For instance,
given a Hausdorff space \((\X,\tau)\), the following conditions are equivalent:
\begin{itemize}
\item \((\X,\tau)\) is completely regular (thus, a Tychonoff space),
\item \(\tau\) is the initial topology of \(C_b(\X,\tau)\),
\item \(\tau\) is induced by a family of pseudodistances on \(\X\) (or, equivalently, \((\X,\tau)\)
is uniformisable),
\item \(i\colon\X\to\beta\X\) is a homeomorphism onto its image,
\item \((\X,\tau)\) is homeomorphic to a subset of some compact Hausdorff space.
\end{itemize}
In the next result, we obtain several characterisations of an e.m.t.\ space that are akin to the above
equivalent reformulations of being a Tychonoff space (and, in fact, Theorem \ref{thm:alt_char_emt} implies
the above characterisations of Tychonoff spaces, by taking as \(\sfd\) the \(\infty\)-discrete distance
\(\sfd_{\infty\text{-discr}}^\X\) on \(\X\)).
\medskip

Let us introduce a notation that we will use in the next result. Given a pre-e.m.t.\ space \((\X,\tau,\sfd)\),
a set \(C\subseteq\X\) and a point \(\bar x\in\X\), we define the quantity \(\sfd(\bar x,C)\in[0,+\infty]\) as
\[
\sfd(\bar x,C)\coloneqq\inf\{\sfd(\bar x,y)\;|\;y\in C\}.
\]
Note that if \((\X,\tau,\sfd)\) is an e.m.t.\ space, \(C\) is \(\tau\)-closed and \(\bar x\notin C\),
then \(\sfd(\bar x,C)>0\). Indeed, \(C\) is in particular \(\sfd\)-closed (as \(\tau\) is coarser than
the topology induced by \(\sfd\)), which gives \(\sfd(\bar x,C)\neq 0\).
\begin{theorem}[Alternative characterisations of an e.m.t.\ space]\label{thm:alt_char_emt}
Let \((\X,\tau,\sfd)\) be a pre-e.m.t.\ space such that \((\X,\tau)\) is Hausdorff. Then the following conditions are equivalent:
\begin{itemize}
\item[\(\rm i)\)] \((\X,\tau,\sfd)\) is an e.m.t.\ space.
\item[\(\rm ii)\)] Given a pre-e.pm.t.\ space \((\Y,\tau_\Y,\sfd_\Y)\) and a map
\(\varphi\colon\Y\to\X\) with \(f\circ\varphi\in\LIP_{b,1}(\Y,\tau_\Y,\sfd_\Y)\) for every
\(f\in\LIP_{b,1}(\X,\tau,\sfd)\), it holds that \(\varphi\colon(\Y,\tau_\Y,\sfd_\Y)\to(\X,\tau,\sfd)\) is continuous-short.
\item[\(\rm iii)\)] Given a \(\tau\)-closed set \(C\subseteq\X\) and a point \(\bar x\in\X\setminus C\), there exists a function \(f\in\LIP_{b,1}(\X,\tau,\sfd)\)
such that \(f|_C=0\) and \(f(\bar x)>0\), and for any \(\tau\)-compact set \(K\subseteq\X\) it holds that each function
in \(\LIP_{b,1}(K,\tau\llcorner K,\sfd|_{K\times K})\) can be extended to a function in \(\LIP_{b,1}(\X,\tau,\sfd)\).
\item[\(\rm iv)\)] Given a \(\tau\)-closed set \(C\subseteq\X\) and a point \(\bar x\in\X\setminus C\), there exists a function \(f\in\LIP_{b,1}(\X,\tau,\sfd)\)
such that \(f|_C=0\) and \(f(\bar x)>0\), and if in addition the set \(C\) is \(\tau\)-compact, then for any \(\lambda\in(0,\sfd(\bar x,C)]\cap(0,+\infty)\)
the function \(f\) can be chosen so that \(f(\bar x)=\lambda\).
\item[\(\rm v)\)] There exists a set \(\Lambda\) of bounded pseudodistances on \(\X\) such that \(\tau\) is the initial
topology of the family \(\{\delta(\cdot,x):\delta\in\Lambda,\,x\in\X\}\), and
\(\sfd(x,y)=\sup\{\delta(x,y):\delta\in\Lambda\}\) for every \(x,y\in\X\).
\item[\(\rm vi)\)] \({\sf c}_{(\X,\tau,\sfd)}\colon(\X,\tau,\sfd)\to{\sf emt}((\X,\tau,\sfd))\) is an isomorphism
in \({\bf PreExt\Psi MetTop}\).
\item[\(\rm vii)\)] \(\bar\iota_{(\X,\tau,\sfd)}\colon(\X,\tau,\sfd)\to\bar\gamma((\X,\tau,\sfd))\) is an
embedding of pre-e.m.t.\ spaces.
\item[\(\rm viii)\)] There exists an embedding of pre-e.m.t.\ spaces of \((\X,\tau,\sfd)\) into a compact e.m.t.\ space.
\end{itemize}
\end{theorem}
\begin{proof}
\ \\
\({\bf i)}\Longrightarrow{\bf ii)}\) Let \((\Y,\tau_\Y,\sfd_\Y)\) be a pre-e.pm.t.\ space and \(\varphi\colon\Y\to\X\)
a map with \(f\circ\varphi\in\LIP_{b,1}(\Y,\tau_\Y,\sfd_\Y)\) for all \(f\in\LIP_{b,1}(\X,\tau,\sfd)\). Since \(\tau\)
coincides with the initial topology of \(\LIP_{b,1}(\X,\tau,\sfd)\), we deduce that \(\varphi\colon(\Y,\tau_\Y)\to(\X,\tau)\)
is continuous. Moreover, for any \(y,z\in\Y\) we have that
\[
\sfd(\varphi(y),\varphi(z))=\sup\big\{|(f\circ\varphi)(y)-(f\circ\varphi)(z)|\;\big|\;f\in\LIP_{b,1}(\X,\tau,\sfd)\big\}
\leq\sfd_\Y(y,z),
\]
which shows that \(\varphi\colon(\Y,\sfd_\Y)\to(\X,\sfd)\) is a short map. Consequently, ii) is proved.
\smallskip

\noindent\({\bf ii)}\Longrightarrow{\bf i)}\) First, let us check that \(\tau\) is the initial topology of \(\LIP_{b,1}(\X,\tau,\sfd)\).
Recalling Remark \ref{rmk:charact_initial_top}, it suffices to show that if \((\Y,\tau_\Y)\) is a topological space
and \(\varphi\colon\Y\to\X\) is a map such that \(f\circ\varphi\colon(\Y,\tau_\Y)\to(\X,\tau)\) is continuous for
every \(f\in\LIP_{b,1}(\X,\tau,\sfd)\), then \(\varphi\colon(\Y,\tau_\Y)\to(\X,\tau)\) is continuous. We equip \(\Y\)
with the \(\infty\)-discrete distance \(\sfd_\Y\coloneqq\sfd^\Y_{\infty\text{-discr}}\).
Clearly, any map from \((\Y,\tau_\Y,\sfd_\Y)\) to \(\R\) is short, thus in particular
\(f\circ\varphi\in\LIP_{b,1}(\Y,\tau_\Y,\sfd_\Y)\) for every \(f\in\LIP_{b,1}(\X,\tau,\sfd)\), and accordingly
\(\varphi\colon(\Y,\tau_\Y)\to(\X,\tau)\) is continuous by ii). Next, let us prove (via a contradiction argument)
that \(\sfd\) can be \(\tau\)-recovered. Suppose that \(\sfd(x,y)>s_{xy}\) for some \(x,y\in\X\), where we set
\(s_{xy}\coloneqq\sup\{|f(x)-f(y)|:f\in\LIP_{b,1}(\X,\tau,\sfd)\}\). We equip \({\rm Z}\coloneqq\{x,y\}\) with the discrete topology
\(\tau_{\rm Z}\) and the unique pseudodistance \(\sfd_{\rm Z}\) satisfying \(\sfd_{\rm Z}(x,y)=s_{xy}\). Also, we denote by \(\varphi\colon{\rm Z}\to\X\)
the inclusion map. For any \(f\in\LIP_{b,1}(\X,\tau,\sfd)\), the function \(f\circ\varphi\colon({\rm Z},\tau_{\rm Z})\to\R\)
is clearly continuous, and we have \(|(f\circ\varphi)(x)-(f\circ\varphi)(y)|\leq s_{xy}<\sfd(x,y)\), so that
\(f\circ\varphi\in\LIP_{b,1}({\rm Z},\tau_{\rm Z},\sfd_{\rm Z})\). However, the fact that
\(\sfd(\varphi(x),\varphi(y))=\sfd(x,y)>s_{xy}=\sfd_{\rm Z}(x,y)\) shows that \(\varphi\colon({\rm Z},\sfd_{\rm Z})\to(\X,\sfd)\)
is not a short map, contradicting ii). All in all, \((\X,\tau,\sfd)\) is an e.m.t.\ space, thus proving i).
\smallskip

\noindent\({\bf i)}\Longrightarrow{\bf iii)}\) By Theorem \ref{thm:compactif_emt} ii), we know that \(\iota(C)\) is a closed subset
of \(\iota(\X)\) equipped with the relative topology of \(\gamma\tau\). Letting \(K\subseteq\gamma\X\) be the
\(\gamma\tau\)-closure of \(\iota(C)\), it is then easy to check that \(\iota(C)=K\cap\iota(\X)\). Since
\((\gamma\X,\gamma\tau)\) is compact, we deduce that the set \(K\) is \(\gamma\tau\)-compact. Note also
that \(\iota(\bar x)\notin K\) and \((\gamma\sfd)(\iota(\bar x),K)>0\). Fix \(\lambda\in(0,(\gamma\sfd)(\iota(\bar x),K))\).
Define \(g\colon K\cup\{\iota(\bar x)\}\to\{0,\lambda\}\) as
\[
g(y)\coloneqq\left\{\begin{array}{ll}
0\\
\lambda
\end{array}\quad\begin{array}{ll}
\text{ if }y\in K,\\
\text{ if }y=\iota(\bar x).
\end{array}\right.
\]
Since \(K\cup\{\iota(\bar x)\}\) is \(\gamma\tau\)-compact and \(g\) is a bounded continuous-short function,
we know from Theorem \ref{thm:ext_CS} that there exists an extension \(\bar g\in\LIP_{b,1}(\gamma\X,\gamma\tau,\gamma\sfd)\)
of \(g\). Finally, let us define \(f\colon\X\to\R\) as \(f(z)\coloneqq\bar g(\iota(z))\) for every \(z\in\X\). Recalling the
properties of \(\iota\) stated in Theorem \ref{thm:compactif_emt}, it is easy to check that \(f\in\LIP_{b,1}(\X,\tau,\sfd)\).
Moreover, we have \(f(z)=\bar g(\iota(z))=0\) for every \(z\in C\) (as \(\iota(z)\in\iota(C)\subseteq K\))
and \(f(\bar x)=\bar g(\iota(\bar x))=\lambda>0\), thus obtaining the first statement in iii). The second statement
follows directly from Theorem \ref{thm:ext_CS}. Hence, iii) is proved. Note that in the above argument
we have not used the universal property of the compactification \((\gamma\X,\gamma\tau,\gamma\sfd)\),
but only the fact that \((\X,\tau,\sfd)\) can be embedded as a pre-e.m.t.\ space into a compact e.m.t.\ space.
\smallskip

\noindent\({\bf iii)}\Longrightarrow{\bf iv)}\) Let \(K\subseteq\X\) be a \(\tau\)-compact set and \(\bar x\in\X\setminus K\).
Note that \(\sfd(\bar x,K)>0\). Now, fix any \(\lambda\in(0,\sfd(\bar x,K)]\cap(0,+\infty)\). Define
\(\tilde f\colon K\cup\{\bar x\}\to\{0,\lambda\}\) as \(\tilde f(x)\coloneqq 0\) for every \(x\in K\) and \(\tilde f(\bar x)\coloneqq\lambda\).
Since \(K\cup\{\bar x\}\) is \(\tau\)-compact and \(\tilde f\) is continuous-short, it follows from
iii) that there exists an extension \(f\in\LIP_{b,1}(\X,\tau,\sfd)\) of \(\tilde f\). Then \(f|_K=0\)
and \(f(\bar x)=\lambda\). Thus, iv) is proved.
\smallskip

\noindent\({\bf iv)}\Longrightarrow{\bf i)}\) Fix any \(U\in\tau\) and \(\bar x\in U\). Using iv), we find a function \(f\in\LIP_{b,1}(\X,\tau,\sfd)\)
such that \(f|_{\X\setminus U}=0\) and \(f(\bar x)\neq 0\). In particular, we have \(x\in\{f\neq 0\}\subseteq U\). Since
\(\{f\neq 0\}=f^{-1}(\R\setminus\{0\})\) belongs to the initial topology \(\sigma\) of \(\LIP_{b,1}(\X,\tau,\sfd)\), we deduce that \(\tau\)
is contained in (and, thus, coincides with) the topology \(\sigma\). Next, let us check that \(\sfd\) can be \(\tau\)-recovered. To this aim,
fix any \(x,y\in\X\) with \(x\neq y\). Note that the singleton \(\{y\}\) is \(\tau\)-compact. Applying iv), for any given
\(\lambda\in(0,\sfd(x,y))\) we can thus find \(f_\lambda\in\LIP_{b,1}(\X,\tau,\sfd)\) such that \(f_\lambda(y)=0\) and \(f_\lambda(x)=\lambda\). Hence,
\[
\sfd(x,y)=\sup_{\lambda\in(0,\sfd(x,y))}\lambda=\sup_{\lambda\in(0,\sfd(x,y))}|f_\lambda(x)-f_\lambda(y)|\leq\sup\big\{|f(x)-f(y)|\;\big|\;f\in\LIP_{b,1}(\X,\tau,\sfd)\big\},
\]
which shows that \(\sfd\) can be \(\tau\)-recovered. All in all, \((\X,\tau,\sfd)\) is an e.m.t.\ space, thus i) is proved.
\smallskip

\noindent\({\bf i)}\Longrightarrow{\bf v)}\) For any finite subset \(F\) of \(\LIP_{b,1}(\X,\tau,\sfd)\),
we define the pseudodistance \(\delta_F\) on \(\X\) as
\[
\delta_F(x,y)\coloneqq\max_{f\in F}|f(x)-f(y)|\quad\text{ for every }x,y\in\X.
\]
Next, let us define \(\Lambda\coloneqq\{\delta_F:F\subseteq\LIP_{b,1}(\X,\tau,\sfd)\text{ finite}\}\).
We denote by \(\sigma\) the initial topology of \(\{\delta(\cdot,x):\delta\in\Lambda,\,x\in\X\}\). We have that
\(\sigma\subseteq\tau\), as each function \(\delta_F(\cdot,x)\) is \(\tau\)-continuous. To prove the converse inclusion,
fix any \(U\in\tau\) and \(x\in U\). Since \(\tau\) is the initial topology of \(\LIP_{b,1}(\X,\tau,\sfd)\), we can find
a finite subset \(F\) of \(\LIP_{b,1}(\X,\tau,\sfd)\) and a real number \(\varepsilon>0\) such that
\[
\delta_F^{-1}((-\infty,\varepsilon))=\bigcap_{f\in F}\{|f-f(x)|<\varepsilon\}
=\bigcap_{f\in F}\{f(x)-\varepsilon<f<f(x)+\varepsilon\}\subseteq U.
\]
Given that \(x\in\delta_F^{-1}((-\infty,\varepsilon))\in\sigma\), we deduce that \(\tau=\sigma\).
Finally, for any \(x,y\in\X\) we have
\[\begin{split}
\sfd(x,y)&=\sup\big\{|f(x)-f(y)|\;\big|\;f\in\LIP_{b,1}(\X,\tau,\sfd)\big\}
=\sup\big\{\delta_{\{f\}}(x,y)\;\big|\;f\in\LIP_{b,1}(\X,\tau,\sfd)\big\}\\
&\leq\sup_{\delta\in\Lambda}\delta(x,y)\leq\sfd(x,y),
\end{split}\]
thanks to the fact that \(\sfd\) can be \(\tau\)-recovered and \(\delta\leq\sfd\) for every \(\delta\in\Lambda\).
All in all, v) is proved.
\smallskip

\noindent\({\bf v)}\Longrightarrow{\bf i)}\) Note that \(\delta(\cdot,x)\in\LIP_{b,1}(\X,\tau,\sfd)\) for all
\(\delta\in\Lambda\) and \(x\in\X\). Since \(\tau\) is the initial topology of \(\{\delta(\cdot,x):\delta\in\Lambda,\,x\in\X\}\),
we have a fortiori that \(\tau\) is the initial topology of \(\LIP_{b,1}(\X,\tau,\sfd)\) and
\[
\sfd(x,y)=\sup_{\delta\in\Lambda}\delta(x,y)=\sup_{\delta\in\Lambda}|\delta(x,x)-\delta(y,x)|
\leq\sup\big\{|f(x)-f(y)|\;\big|\;f\in\LIP_{b,1}(\X,\tau,\sfd)\big\}
\]
for every \(x,y\in\X\). Therefore, \((\X,\tau,\sfd)\) is an e.m.t.\ space, thus accordingly i) is proved.
\smallskip

\noindent\({\bf i)}\Longrightarrow{\bf vi)}\) The quickest way to prove this implication is to examine the proof of
Proposition \ref{prop:emt-fication_constr}. Indeed, it is evident from the construction of \((\check\X,\check\tau,\check\sfd)\)
that if \((\X,\tau,\sfd)\) is an e.m.t.\ space, then we have \(\X=\check\X\), \(\tau=\tilde\tau=\check\tau\),
\(\sfd=\tilde\sfd=\check\sfd\) and \({\sf c}_{(\X,\tau,\sfd)}={\rm id}_\X\), which implies vi).
\smallskip

\noindent\({\bf vi)}\Longrightarrow{\bf i)}\) It follows from the elementary observation that the property of
being an e.m.t.\ space is invariant under isomorphisms in \({\bf PreExt\Psi MetTop}\).
\smallskip

\noindent\({\bf vi)}\Longrightarrow{\bf vii)}\) Since \(\iota_{{\sf emt}((\X,\tau,\sfd))}\colon{\sf emt}((\X,\tau,\sfd))\to\bar\gamma((\X,\tau,\sfd))\)
is an embedding of e.m.t.\ spaces by Theorem \ref{thm:compactif_emt} ii), we deduce that
\(\bar\iota_{(\X,\tau,\sfd)}=\iota_{{\sf emt}((\X,\tau,\sfd))}\circ{\sf c}_{(\X,\tau,\sfd)}\) is an
embedding of pre-e.m.t.\ spaces, thus showing the validity of vii).
\smallskip

\noindent\({\bf vii)}\Longrightarrow{\bf vi)}\) Let us argue by contradiction: assume that \({\sf c}_{(\X,\tau,\sfd)}\)
is not an isomorphism in the category \({\bf PreExt\Psi MetTop}\), thus either
\({\sf c}_{(\X,\tau,\sfd)}\colon(\X,\tau)\to(\check\X,\check\tau)\) is not a homeomorphism
or \({\sf c}_{(\X,\tau,\sfd)}\colon(\X,\sfd)\to(\check\X,\check\sfd)\) is not distance preserving. In the former case,
we have that the map \({\sf c}_{(\X,\tau,\sfd)}\colon\X\to\check\X\) is bijective (by \eqref{eq:c_surj} and the injectivity of
\(\iota_{{\sf emt}((\X,\tau,\sfd))}\circ{\sf c}_{(\X,\tau,\sfd)}=\bar\iota_{(\X,\tau,\sfd)}\)),
but its inverse \({\sf c}_{(\X,\tau,\sfd)}^{-1}\colon(\check\X,\check\tau)\to(\X,\tau)\) is not continuous,
thus in particular \(\bar\iota_{(\X,\tau,\sfd)}^{-1}\colon\bar\iota_{(\X,\tau,\sfd)}(\X)\to\X\) is not continuous
(otherwise \({\sf c}_{(\X,\tau,\sfd)}^{-1}=\bar\iota_{(\X,\tau,\sfd)}^{-1}\circ\iota_{{\sf emt}((\X,\tau,\sfd))}^{-1}\)
would be continuous). In the latter case, there exist points \(x,y\in\X\) such that
\(\check\sfd({\sf c}_{(\X,\tau,\sfd)}(x),{\sf c}_{(\X,\tau,\sfd)}(y))<\sfd(x,y)\), whence it follows that
\((\gamma\sfd)(\bar\iota_{(\X,\tau,\sfd)}(x),\bar\iota_{(\X,\tau,\sfd)}(y))<\sfd(x,y)\). All in all, we have shown that
\(\bar\iota_{(\X,\tau,\sfd)}\) is not an embedding of pre-e.m.t.\ spaces, in contradiction with vii). Hence, vi) is proved.
\smallskip

\noindent\({\bf vii)}\Longrightarrow{\bf viii)}\) Immediate, since \(\bar\gamma((\X,\tau,\sfd))\) is a compact e.m.t.\ space.
\smallskip

\noindent\({\bf viii)}\Longrightarrow{\bf iii)}\) It can be shown by arguing exactly as in the proof of the
implication \({\bf i)}\Longrightarrow{\bf iii)}\).
\end{proof}

Note that some implications in Theorem \ref{thm:alt_char_emt}, for example \({\bf vi)}\Longrightarrow{\bf i)}\) and
\({\bf vii)}\Longrightarrow{\bf vi)}\), hold even under the weaker assumption that \((\X,\tau,\sfd)\) is only a pre-e.pm.t.\ space.
\begin{remark}{\rm
It is worth pointing out that, as a consequence of the implication \({\bf i)}\Longrightarrow{\bf iv)}\)
in Theorem \ref{thm:alt_char_emt}, the following property holds: if \((\X,\tau,\sfd)\) is an e.m.t.\ space
and \(x,y\in\X\) satisfy \(\sfd(x,y)<+\infty\), then there exists \(f\in\LIP_{b,1}(\X,\tau,\sfd)\)
with \(f(x)-f(y)=\sfd(x,y)\). In other words, the supremum in the \(\tau\)-recovery property of \(\sfd\)
is achieved for points at finite distance.
\fr}\end{remark}
\begin{remark}[Comparison with topometric spaces]\label{rmk:comparison_topometric}{\rm
Following \cite[Definition 1.2]{Ben08b}, we say that a pre-e.m.t.\ space \((\X,\tau,\sfd)\) is a \emph{topometric space}
provided the following conditions hold:
\begin{itemize}
\item \(\tau\) is coarser than the topology induced by \(\sfd\).
\item \(\sfd\colon\X\times\X\to[0,+\infty]\) is \(\tau\otimes\tau\)-lower semicontinuous.
\end{itemize}
Every e.m.t.\ space is a topometric space (by Remark \ref{rmk:d_lsc}), but the converse fails: for example,
equip a Hausdorff space \((\X,\tau)\) that is not completely regular with the \(\infty\)-discrete distance \(\sfd_{\infty\text{-discr}}^\X\);
then \((\X,\tau,\sfd_{\infty\text{-discr}}^\X)\) is a topometric space, but it is not an e.m.t.\ space.

Furthermore, following \cite[Definition 2.3]{Ben13}, we say that a topometric space \((\X,\tau,\sfd)\) is \emph{completely regular}
provided \(\sfd\) can be \(\tau\)-recovered and \(\LIP_{b,1}(\X,\tau,\sfd)\) \emph{separates points from closed sets}, which means
that for any \(\tau\)-closed set \(C\subseteq\X\) and \(\bar x\in\X\setminus C\) there exists \(f\in\LIP_{b,1}(\X,\tau,\sfd)\) such that
\(f|_C=0\) and \(f(\bar x)>0\). Therefore, it follows from Theorem \ref{thm:alt_char_emt} and \cite[Corollary 2.6]{Ben13} that
\[
(\X,\tau,\sfd)\text{ is a completely-regular topometric space}\quad\Longleftrightarrow\quad(\X,\tau,\sfd)\text{ is an e.m.t.\ space.}
\]
It is also worth pointing out that the so-called \emph{Polish extended spaces} introduced by Ambrosio, Gigli and Savar\'{e} in
\cite[Definition 2.3]{AmbrosioGigliSavare11} are exactly those topometric spaces \((\X,\tau,\sfd)\) such that \((\X,\tau)\) is a Polish space
(i.e.\ \(\tau\) is induced by a complete separable distance) and \((\X,\sfd)\) is complete.
\fr}\end{remark}
%
%
%
%
%\bibliographystyle{abbrv}
%\bibliography{biblio.bib}
%
%

%
%
%
%
\end{document}